\documentclass[reqno]{amsart}
\usepackage[utf8]{inputenc}
\usepackage{color,mathtools}
\usepackage{array}
\usepackage{refstyle}
\usepackage{varwidth}
\usepackage{amstext}
\usepackage{amsthm}
\usepackage{amssymb}
\usepackage{graphicx}
\usepackage[unicode=true,pdfusetitle,
 bookmarks=true,bookmarksnumbered=true,bookmarksopen=false,
 breaklinks=false,pdfborder={0 0 1},backref=false,colorlinks=true]
 {hyperref}

\makeatletter

\AtBeginDocument{}
\AtBeginDocument{\providecommand\Secref[1]{\ref{Sec:#1}}}
\AtBeginDocument{\providecommand\Propref[1]{\ref{Prop:#1}}}
\AtBeginDocument{\providecommand\Corref[1]{\ref{Cor:#1}}}
\AtBeginDocument{\providecommand\Lemref[1]{\ref{Lem:#1}}}
\AtBeginDocument{\providecommand\Subsecref[1]{\ref{Subsec:#1}}}
\providecommand{\tabularnewline}{\\}
\newenvironment{cellvarwidth}[1][t]
    {\begin{varwidth}[#1]{\linewidth}}
    {\@finalstrut\@arstrutbox\end{varwidth}}
\RS@ifundefined{subsecref}
  {\newref{subsec}{name = \RSsectxt}}
  {}
\RS@ifundefined{thmref}
  {\def\RSthmtxt{theorem~}\newref{thm}{name = \RSthmtxt}}
  {}
\RS@ifundefined{lemref}
  {\def\RSlemtxt{lemma~}\newref{lem}{name = \RSlemtxt}}
  {}

\numberwithin{equation}{section}
\numberwithin{figure}{section}
\theoremstyle{plain}
\newtheorem{thm}{\protect\theoremname}
\numberwithin{thm}{section}
\theoremstyle{plain}

\theoremstyle{remark}
\newtheorem{rem}[thm]{\protect\remarkname}
\theoremstyle{plain}
\newtheorem{prop}[thm]{\protect\propositionname}
\theoremstyle{plain}
\newtheorem{cor}[thm]{\protect\corollaryname}
\theoremstyle{remark}
\newtheorem*{rem*}{\protect\remarkname}
\theoremstyle{definition}

\theoremstyle{plain}
\newtheorem{lem}[thm]{\protect\lemmaname}

\@ifundefined{date}{}{\date{}}

\usepackage{textgreek}
\usepackage{graphicx}
\usepackage{amsfonts}
\usepackage{amsthm}
\allowdisplaybreaks
\usepackage{cases}
\usepackage{tikz-cd}
\usepackage{cancel}
\usepackage{faktor}

\usepackage{refstyle}
\newref{cor}{name = corollary~, names = corollaries~, Name = Corollary~, Names = Corollaries~}
\newref{app}{name= appendix~, names = appendices~, Name = Appendix~, Names = Appendices~}
\newref{conj}{name= conjecture~, names = conjectures~, Name = Conjecture~, Names = Conjectures~}
\newref{exa}{name= example~, names = examples~, Name = Example~, Names = Examples~}
\newref{rem}{name= remark~, names = remarks~, Name = Remark~, Names = Remarks~} 
\newref{thm}{name= theorem~, names = theorems~, Name = Theorem~, Names = Theorems~}
\newref{lem}{name= lemma~, names = lemmas~, Name = Lemma~, Names = Lemmas~}
\newref{def}{name= definition~, names = definitions~, Name = Definition~, Names = Definitions~}
\newref{fact}{name= fact~, names = facts~, Name = Fact~, Names = Facts~}
\newref{blackbox}{name= blackbox~, names = blackboxes~, Name = Blackbox~, Names = Blackboxes~}
\newref{sec}{name= section~, names = sections~, Name = Section~, Names = Sections~}
\newref{subsec}{name= subsection~, names = subsections~, Name = Subsection~, Names = Subsections~}
\newref{prop}{name= proposition~, names = propositions~, Name = Proposition~, Names = Propositions~}

\usepackage[framemethod=TikZ]{mdframed}
 
\def\Xint#1{\mathchoice
{\XXint\displaystyle\textstyle{#1}}%
{\XXint\textstyle\scriptstyle{#1}}%
{\XXint\scriptstyle\scriptscriptstyle{#1}}%
{\XXint\scriptscriptstyle\scriptscriptstyle{#1}}%
\!\int}
\def\XXint#1#2#3{{\setbox0=\hbox{$#1{#2#3}{\int}$ }
\vcenter{\hbox{$#2#3$ }}\kern-.6\wd0}}

\def\dashint{\Xint-}

\usepackage{accents}

\providetoggle{nomsort}
\settoggle{nomsort}{true} 

\makeatletter
\iftoggle{nomsort}{%
    \let\old@@@nomenclature=\@@@nomenclature        
        \newcounter{@nomcount} \setcounter{@nomcount}{0}%
        \renewcommand\the@nomcount{\two@digits{\value{@nomcount}}}
        \def\@@@nomenclature[#1]#2#3{
          \addtocounter{@nomcount}{1}%
        \def\@tempa{#2}\def\@tempb{#3}%
          \protected@write\@nomenclaturefile{}%
          {\string\nomenclatureentry{\the@nomcount\nom@verb\@tempa @[{\nom@verb\@tempa}]%
          \begingroup\nom@verb\@tempb\protect\nomeqref{\theequation}%
          |nompageref}{\thepage}}%
          \endgroup
          \@esphack}%
      }{}
\makeatother

\usepackage{tensind}
\tensordelimiter{?}

\hypersetup{citecolor=purple,linkcolor=blue}

\newcommand\restr[2]{{
  \left.\kern-\nulldelimiterspace 
  #1 
  \vphantom{\big|} 
  \right|_{#2} 
  }}

\usepackage{extarrows}

\usepackage{enumerate}\usepackage{color}

\let\div\relax
\DeclareMathOperator\div{div}

\DeclareMathOperator\Div{div}
\DeclareMathOperator\curl{curl}

\DeclareMathOperator\supp{supp}

\newcommand{\dd}[0]{\partial}
\newcommand{\grad}[0]{\nabla}

\newcommand{\T}[0]{\mathbb T^d}

\theoremstyle{plain}

\newcommand{\eq}[1]{\begin{align}{#1}\end{align}}
\newcommand{\eqn}[1]{\begin{align*}{#1}\end{align*}}
\newcommand{\Rd}[0]{{\mathbb R^d}}
\newcommand{\Td}[0]{{\mathbb T^d}}
\newcommand{\Ig}[2]{I^{g;#1}_{#2}}

\makeatother

\providecommand{\conjecturename}{Conjecture}
\providecommand{\corollaryname}{Corollary}
\providecommand{\definitionname}{Definition}
\providecommand{\lemmaname}{Lemma}
\providecommand{\propositionname}{Proposition}
\providecommand{\remarkname}{Remark}
\providecommand{\theoremname}{Theorem}

\begin{document}
\global\long\def\divop{\operatorname{div}}%
\global\long\def\supp{\operatorname{supp}}%
\global\long\def\curl{\operatorname{curl}}%

\global\long\def\grad{\nabla}%
\global\long\def\ls{\lesssim}%

\global\long\def\dd{\partial}%

\global\long\def\Ig{I_{\#2}^{g;\#1}}%

\global\long\def\T{\mathbb{T}^{d}}%

\title[Convex integration above the Onsager exponent]{Convex integration above the Onsager exponent for the forced Euler equations}
\author{Aynur Bulut}
\address{Louisiana State University, 303 Lockett Hall, Baton Rouge, LA 70803}
\email{aynurbulut@lsu.edu}
\author{Manh Khang Huynh}
\address{Georgia Institute of Technology Department of Mathematics}
\email{mhuynh41@gatech.edu}
\author{Stan Palasek}
\address{UCLA Department of Mathematics}
\email{palasek@math.ucla.edu}
\begin{abstract}
We establish new non-uniqueness results for the Euler equations with external force on $\mathbb{T}^{d}$ $(d\geq3)$.  By introducing a novel alternating convex integration scheme, we construct non-unique, almost-everywhere smooth, Hölder-continuous solutions with regularity $\frac{1}{2}-$, which is notably above the Onsager threshold of $\frac{1}{3}$.

The solutions we construct differ significantly in nature from those which arise from the recent unstable vortex construction of Vishik; in particular, our solutions are genuinely $d$-dimensional ($d\geq3$), and give non-uniqueness results for any smooth data.  To the best of our knowledge, this is the first instance of a convex integration construction above the Onsager exponent.
\end{abstract}

\maketitle

\section{Introduction}

\setlength{\parskip}{0.05in}

We consider the incompressible Euler equations with external force $f:[0,T]\times\mathbb{T}^{d}\to\mathbb{R}^{d}$,
\begin{equation}
\begin{cases}
\partial_{t}v+\divop v\otimes v+\nabla p=f\\
\divop v=0
\end{cases}\label{eq:Euler_start}
\end{equation}
on the periodic domain $\mathbb{T}^{d}$, where $d\geq3$ is the spatial dimension, $v:[0,T]\times\mathbb{T}^{d}\to\mathbb{R}^{d}$
is the velocity field, and $p:[0,T]\times\mathbb{T}^{d}\to\mathbb{R}$ is the pressure.

When the forcing term $f$ is sufficiently regular, the classical theory shows that solutions of the Euler system (\ref{eq:Euler_start}) with $v\in C_t^0C_x^{1+\alpha}$ enjoy favorable regularity properties.  This includes, for instance, local well-posedness of the initial value problem and conservation of energy (in the sense that change in kinetic energy is balanced by work done by the force). 

A central question in the theory of weak solutions and fluid turbulence is whether these properties persist at lower regularities.  To formulate this more precisely, in the context of the framework formulated by Klainerman in \cite{klainermanNashUniqueContribution2016} (see also \cite{buckmasterConvexIntegrationPhenomenologies2019}, \cite{buckmasterNonuniquenessWeakSolutions2019b}), fixing a scale of function spaces $X^\alpha$, a number of critical regularity thresholds arise: 
\begin{itemize}
\item the Onsager threshold $\alpha_O$ for conservation of energy; 
\item the Nash threshold $\alpha_N$ separating flexibility and rigidity; 
\item the threshold $\alpha_U$ of regularity above which we are guaranteed uniqueness for the initial value problem; and 
\item the threshold $\alpha_{WP}$ of regularity above which the initial value problem is locally well-posed.
\end{itemize}  
With these definitions in hand, it is reasonable to expect that the ordering
\[
\alpha_{O}\leq\alpha_{N}\leq\alpha_{U}\leq\alpha_{WP}
\]
should hold.

In recent years, with the conventional choice $X^\alpha=C^\alpha$, substantial progress has been made in determining the critical exponents for the unforced Euler system.  Bourgain-Li \cite{bourgain2015strong} and Elgindi-Masmoudi \cite{elgindi2020infty} have shown that $\alpha_{WP}=1$.  On the other end of the scale, we have $\alpha_O\leq\frac13$ due to Constantin-E-Titi \cite{constantinOnsagerConjectureEnergy1994}.  The equality $\alpha_O=\frac13$, known as Onsager's conjecture, was proven recently by Isett \cite{isettProofOnsagerConjecture2018} using a convex integration approach pioneered by De~Lellis and
Székelyhidi Jr.\ \cite{delellisEulerEquationsDifferential2007} and advanced by many other authors; see \cite{delellisDissipativeContinuousEuler2013,delellisDissipativeEulerFlows2014,delellisEulerEquationsDifferential2007,buckmasterAnomalousDissipation5Holder2015,isettOlderContinuousEuler2014}
and the references therein.  The flexible construction of non-conservative Euler flows in $C^{\frac13-}$ can also be applied to exhibit non-uniqueness and an h-principle \cite{buckmasterOnsagerConjectureAdmissible2017}, thus establishing that $\min(\alpha_N,\alpha_U)\geq\frac13$.

However, determining the precise values of $\alpha_N$ and $\alpha_U$ for the unforced Euler system remains a difficult and unsolved problem.  Indeed, toward this end, Klainerman asks in \cite{klainermanNashUniqueContribution2016},
\begin{quotation}
``Can one extend convex integration methods to construct solutions
above the Onsager exponent?''
\end{quotation}

In this paper, we answer this question affirmatively in the case of the \emph{forced} Euler system (\ref{eq:Euler_start}).  Our main theorem is as follows.

\begin{thm}
\label{thm:thm1}With $d\geq3$, let $V_{1},V_{2},V_{3}\in C^\infty(\Td\to\Rd)$ be any divergence-free vector fields such that $\int_{\mathbb{T}^{d}}V_{1}=\int_{\mathbb{T}^{d}}V_{2}=\int_{\mathbb{T}^{d}}V_{3}$.

Then for every $\beta\in(0,\frac12)$ there exist $u,v\in C_{t,x}^{\beta-}$ and $F\in C_{t}^{0}C_{x}^{2\beta-}$
such that 
\begin{itemize}
\item $u,v$ are weak solutions to (\ref{eq:Euler_start}) with
common initial data $u(0)=v(0)=V_{1}$ and force $f=\div F$,
\item $u(T)=V_{3}$, $v(T)=V_{2}$, and
\item $u$, $v$, and $F$ are smooth for almost all times.
\end{itemize}
\end{thm}

In particular, for any choice of smooth data, there exists an external force such that uniqueness of the initial value problem fails.  

\begin{rem}
\label{rem:F_smooth} Note that when $\beta>\frac{1}{3}$, the combined
regularity of the velocity and force fields is sufficient to guarantee energy
balance (a proof of this fact in the spirit of
\cite{constantinOnsagerConjectureEnergy1994} is given in Appendix
\ref{app:Onsager-exponent}).  This justifies the claim that our solutions live
above the Onsager regularity threshold.  Moreover, while the force barely fails
to be continuous (in the sense that it is the divergence of a $C_x^{1-}$ tensor
field), the total work $\int_0^T\int_\Td f\cdot v\,dxdt$ is finite.
\end{rem}

To prove Theorem~\ref{thm:thm1}, we introduce a novel alternating convex
integration scheme (described in \Subsecref{strategy} below).  The solutions we
construct are significantly different in character from those emerging from the
unstable vortex constructed by Vishik
\cite{vishikInstabilityNonuniquenessCauchy2018,vishikInstabilityNonuniquenessCauchy2018b};
see Subsection \ref{comparison} below for details.  The main idea is that by
choosing the force appropriately, we can execute a convex integration scheme in
which, say, $u$ and $v$ are only perturbed on even and odd steps respectively.
As a result, the successive perturbations are more widely separated in
frequency space, so that stationary phase arguments (see, e.g.
Lemma~\ref{stationaryphase}) produce errors which satisfy improved estimates.

We note that for the unforced Euler system, there is a serious obstruction to convex integration above $\beta=\frac{1}{3}$, as a consequence of \cite{constantinOnsagerConjectureEnergy1994}.  Indeed, standard applications of convex integration schemes would be restricted to producing energy conservative solutions, which is at odds with the expected freedom in choosing the ``slow'' profile of the perturbations.  In the present work, we avoid this issue by letting a carefully constructed external force balance the excess energy pumped in by the high frequency perturbations.  

Before discussing the details of our approach, we make two remarks regarding bounds on the set of singular times for our constructed solutions $u$ and $v$, and the regularity of the constructed force $f$.

\begin{rem}
Let $\mathcal B\subset[0,T]$ be the minimal closed set of times such that $u,v,F\big|_{\mathcal{B}^{c}\times\mathbb{T}^{d}}$ are smooth.  Our convex integration scheme implies a quantitative bound on this singular set (cf.\ \cite{derosaDimensionSingularSet2021,cheskidovSharpNonuniquenessNavierStokes2020,bulutEpochsRegularityWild2022}):
\eqn{
\dim_{\mathcal H}\left(\mathcal{B}\right)\leq\left(\frac{1}{2(1-\beta)}\right)^{+}.
}
\end{rem}

\begin{rem}
Due to the favorable estimates obeyed by the material derivative of the Reynolds stress during our convex integration procedure, one should expect that the force is regular in time as well, and in particular that one has $F\in C_{t,x}^{2\beta-}$; however we do not pursue this question here.
\end{rem}

\subsection{Alternating convex integration strategy}\label{subsec:strategy}

We now summarize the new ideas required to execute convex integration up to regularity $C_x^{1/2-}$.  We take the proof of Onsager's conjecture in \cite{buckmasterOnsagerConjectureAdmissible2017} as our point of comparison and make use of the now-standard notation and terminology therein.  We distinguish two types of errors that appear in $R_{q+1}$: the oscillation error in which the perturbation $w_{q+1}$ interacts with itself, and linear errors in which the perturbation interacts with the coarse flow $\overline v_q$.

\subsubsection{Oscillation error}\label{oscillation}

A careful reading of the proof of Onsager's conjecture reveals that the oscillation error is already suitably small all the way up to $\frac12-$.  Indeed, $R_\ell$ is mainly supported on frequencies up to $\lambda_q$ which leads (roughly) to $\|R_\ell\|_1\lesssim\lambda_q\delta_{q+1}$ (cf.\ the bound $\|R_\ell\|_1\lesssim\ell^{-1}\delta_{q+1}$ used in \cite{buckmasterOnsagerConjectureAdmissible2017}).  Employing this tighter bound, the required error estimates on $\mathcal \div(w_o\otimes w_o)$, $w_o\otimes w_c$, etc.\ follow from the parameter constraint\footnote{Recall $\alpha$ is a positive constant that is chosen to be small depending on $\beta$.}
\eq{\label{onehalfconstraint}
\lambda_{q+1}^{-1+O(\alpha)}\lambda_q\delta_{q+1}\lesssim\delta_{q+2}
}
which can be satisfied for all $\beta<\frac12$.  As a result, no modification is necessary except to carefully track the optimal estimates on the first several derivatives of $v_q$ and $R_q$.  A similar strategy is encapsulated in the ``frequency-energy levels'' used by Isett \cite{isettProofOnsagerConjecture2018}.

Let us also remark that (\ref{onehalfconstraint}) appears to be an inescapable requirement for any convex integration scheme for a system with a quadratic nonlinearity containing one spatial derivative.  We take this as further evidence of our conjecture in Subsection~\ref{comparison} that $\alpha_N=\frac12$ for the forced Euler equations, or at least that ideas substantially different from convex integration would be required to exceed this threshold.

\subsubsection{Linear errors}

It is well understood that the linear errors restrict the Onsager scheme to $\alpha<\frac13$.  Indeed, the ``Nash error'' takes the form $\mathcal R(w_{q+1}\cdot\grad\overline v_q)$ which is uniformly bounded by $\lambda_{q+1}^{-1}\delta_{q+1}^\frac12\|\overline v_q\|_{C^{1,\alpha}}$.  Local theory for the Euler equations implies that the glued solution $\overline v_q$ obeys $\|\overline v_q\|_{C^{1,\alpha}}\lesssim\|v_q\|_{C^{1,\alpha}}$; thus to close the iteration estimates it is required that
\eq{\label{onsagerconstraint}
\lambda_{q+1}^{-1+O(\alpha)}\delta_{q+1}^\frac12\|v_q\|_{C^{1,\alpha}}\lesssim\delta_{q+2}.
}
In fact, the other linear error demands the same constraint (\ref{onsagerconstraint}) on the parameters.  In the Onsager scheme, one cannot expect a better bound than $\|v_q\|_{C^{1,\alpha}}\lesssim\lambda_q^{1+\alpha}\delta_q^\frac12$ which, along with (\ref{onsagerconstraint}), leads to $\beta<\frac13$.  This motivates our objective to design a scheme for which we have a substantially better estimate for $\|v_q\|_{C^{1,\alpha}}$.

The strategy is as follows: we simultaneously consider forced Euler systems for the two velocity fields $u$ and $v$, one of which has a Reynolds stress error:
\begin{equation}
\begin{cases}
\dd_{t}u_{q}+\divop u_{q}\otimes u_{q}+\grad\pi_{q}=\divop F_{q}\\
\dd_{t}v_{q}+\divop v_{q}\otimes v_{q}+\grad p_{q}=\divop F_{q}+\divop R_{q}\\
\divop u_{q}=\divop v_{q}=0.
\end{cases}\label{eq:forced_Euler_dual_system}
\end{equation}
For the moment, let us ignore complications related to mollification and gluing.  In this oversimplified scenario, only $v_q$ needs to be perturbed by convex integration because it possesses the Reynolds stress.  Thus, constructing a perturbation $w_{q+1}$ to cancel $R_q$ as in \cite{buckmasterOnsagerConjectureAdmissible2017}, we can set $u_{q+1}{\,\coloneqq\,}u_q$ and $v_{q+1}{\,\coloneqq\,}v_q+w_{q+1}$ which solve the same system with force $\div F_q$ and a smaller Reynolds stress $\widetilde{R}_{q+1}$.  The key idea is to modify the force to move the Reynolds stress onto the $u_{q+1}$ equation---by setting $F_{q+1}{\,\coloneqq\,}F_q+\widetilde{R}_{q+1}$ and $R_{q+1}{\,\coloneqq\,}-\widetilde{R}_{q+1}$, we have the new system
\begin{equation*}
\begin{cases}
\dd_{t}u_{q+1}+\divop u_{q+1}\otimes u_{q+1}+\grad\pi_{q+1}=\divop F_{q+1}+\divop R_{q+1}\\
\dd_{t}v_{q+1}+\divop v_{q+1}\otimes v_{q+1}+\grad p_{q+1}=\divop F_{q+1}\\
\divop u_{q+1}=\divop v_{q+1}=0.
\end{cases}
\end{equation*}
The point is that since $u_q$ was not perturbed in the last step, we have the improved bound
\eqn{
\|u_{q+1}\|_{C^{1,\alpha}}\lesssim\|u_q\|_{C^{1,\alpha}}
}
which will weaken the requirement (\ref{onsagerconstraint}).  Indeed, by perturbing each velocity field only on \emph{every other} step, we should\footnote{Due to the effect of the force on the local existence theory (see Lemma~\ref{lem:local_existence}), the glued field has slightly worse bounds than $v_{q-1}$.} have the improved bound $\|v_q\|_{C^{1,\alpha}}\lesssim\lambda_{q-1}^{1+\alpha}\delta_{q-1}^\frac12$.  As a result, when we perform a convex integration step on, say, (\ref{eq:forced_Euler_dual_system}), the constraint (\ref{onsagerconstraint}) becomes $\lambda_{q+1}^{-1+O(\alpha)}\lambda_{q-1}\delta_{q-1}^\frac12\delta_{q+1}^\frac12\lesssim\delta_{q+2}$ which can be satisfied for all $\beta<\frac12$.

This strategy for improving the bound on $\|v_q\|_{C^{1,\alpha}}$ using an alternating convex integration scheme has application as well for 2D Euler and inviscid SQG.  In forthcoming work \cite{buluthuynhpalasekSQG}, we prove non-uniqueness of forced weak solutions in stronger spaces than those that appear in \cite{buckmasterNonuniquenessWeakSolutions2019b}.

\subsubsection{Epochs of regularity and gluing}

In order to specify the initial and final data and to obtain
almost-everywhere smoothness, we use the so-called ``epochs of regularity''
approach, as previously seen in \cite{derosaDimensionSingularSet2021,bulutEpochsRegularityWild2022},
where we optimize the gluing interval and try to glue newer local
solutions with older approximate solutions.  The presence of the force
in this case, however, actually allows us complete control over initial
data and ultimate data.  A new complication arises in this approach,
as we have to perform this modified gluing process for both systems
at once.  It therefore becomes necessary to isolate the resulting errors,
so that they do not ruin the material derivative estimates for the
convex integration in the active system.  The optimized derivative estimates mentioned in Subsubsection~\ref{oscillation} play a crucial role in estimating the gluing errors, as the only lower bound on the mollification length scale $\ell_q$ is the scale of the perturbation.

\subsection{Comparison to previous results}\label{comparison}

In groundbreaking work, Vishik \cite{vishikInstabilityNonuniquenessCauchy2018,vishikInstabilityNonuniquenessCauchy2018b} (see also the notes \cite{albrittonInstabilityNonuniqueness2d2022}) constructed an unstable two-dimensional vortex which, considered in self-similar coordinates, leads to non-unique solutions of (\ref{eq:Euler_start}) with vorticity in $L_t^\infty L_x^p$ for any large $p<\infty$.  In particular this implies that for the forced Euler equations on $\mathbb R^2$ (and, by a trivial extension, $\mathbb R^2\times\mathbb T$), there are non-unique solutions in $L_t^\infty C_x^{1-}$; thus $\alpha_U=1$.  We remark that recently, Bru\'e and De~Lellis \cite{bruedelellisAnomDissip} have also studied anomolous dissipation results for the forced Navier-Stokes equation in the vanishing viscosity limit (see also \cite{bruecolomboOnsagerCritical}).

While our convex integration approach is only able to show $\alpha_U\geq\frac12$, it has the advantage that the solutions constructed are genuinely $d$-dimensional ($d\geq3$) and non-uniqueness is exhibited from \emph{any} smooth data.\footnote{By an elementary gluing argument, non-uniqueness from a particular initial datum implies non-uniqueness from any initial data with an appropriately chosen force.  However, unlike in Theorem~\ref{thm:thm1}, the force cannot be expected to be continuous in time.  Moreover, this trivial gluing argument does not allow one to freely specify $u(T)$ and $v(T)$ as in the theorem.} We believe the failure of this method above $\alpha=\frac12$ is interesting in itself as a possible indication of the value of $\alpha_N$ for the forced equation.  While the non-unique solutions from \cite{vishikInstabilityNonuniquenessCauchy2018,vishikInstabilityNonuniquenessCauchy2018b} enjoy stronger regularity properties than those from Theorem~\ref{thm:thm1}, they do not appear to suggest any flexibility or genericity of the space of solutions.  Indeed, the solutions constructed there are restricted to the vicinity of a particular unstable manifold of a family of vortices.

On the other hand, Theorem~\ref{thm:thm1} is proved using convex integration which is well-known as a tool to prove h-principles for various problems.  From the success of convex integration for $\alpha<\frac12$ (Proposition~\ref{prop:iterative}) and some apparently serious issues when $\alpha>\frac12$ (see Subsection~\ref{subsec:strategy}), one is led to conjecture $\alpha_N=\frac12$ as the exact threshold for the forced Euler equations.  This hypothesis is bolstered by De~Lellis and Inauen \cite{delellisInauen} and Cao and Inauen \cite{caoRigidityFlexibilityIsometric2020} who identified (in a sense) $\alpha_N=\frac12$ for the related problem of isometric extension.  We remark that the precise definition of the h-principle for a forced system is not clear and perhaps not unique---dramatically different outcomes are possible depending on whether the force is allowed to vary in the weak approximation.  Finding a natural formulation of the h-principle for (\ref{eq:Euler_start}) and determining $\alpha_N$ is interesting and will be the subject of future work.

\subsection{Organization of the paper}

The paper is structured as follows: in Section~\ref{sec:notation} we introduce our notational conventions and formulate the iterative proposition for the alternating convex integration scheme.  In Section~\ref{sec:Proof-of-main-thm} we then show how this iterative proposition is used to prove Theorem~\ref{thm:thm1}.  In Section~\ref{sec:proofProp} we begin the proof of the iterative proposition by implementing the mollification and gluing steps (this includes a delicate part of the argument, proving suitable estimates on the glued fields at the ``good-bad'' and ``bad-bad'' interfaces).  In Section~\ref{sec:Convex-integration-and} we construct the perturbation and prove estimates on the resulting fields in order to close the iterative proposition.  The proof of the iterative proposition is then completed in Section~\ref{sec:proofProp2}.  A brief Appendix~\ref{app:parameters} records several comparison estimates for the parameters used in the convex integration construction.  For the reader's convenience we prove the local theory needed to execute gluing for the forced Euler system in Appendix~\ref{app:Local-existence-of}, and include a proof that the weak solutions we construct preserve the energy balance in Appendix~\ref{app:Onsager-exponent}.

\subsection*{Acknowledgements}

We are grateful to Dallas Albritton and Terence Tao for useful discussions.  The third author acknowledges support from a UCLA Dissertation Year Fellowship.

\section{Notation and formulation of the main iterative scheme}
\label{sec:notation}

\setlength{\parskip}{0pt plus1pt}

We now establish some basic notational conventions.  As usual, we write $A\lesssim_{x,\neg y}B$ to mean $A\leq CB$ where $C>0$
may depend on $x$ but not $y$.  Similarly, $A\sim_{x,\neg y}B$ denotes that we have both $A\lesssim_{x,\neg y}B$ and 
$B\lesssim_{x,\neg y}A$.  For $x\in\mathbb R$, we write $x+$ (or, analogously, $x-$) to mean that a given expression holds 
for all $y\in\left(x,x+\varepsilon\right)$, with $\varepsilon>0$ taken sufficiently small.

We will leave some dependence on parameters implicit when it is inessential for the argument.

\subsection{Function spaces and geometric preliminaries}
\label{subsec:geo}

For each $N\in\mathbb{N}_{0}$ and $\alpha\in\left(0,1\right)$, we consider the norms and semi-norms
\begin{align*}
\left\Vert f\right\Vert _{N} & =\left\Vert f\right\Vert _{C^{N}},\quad\left[f\right]_{N}=\left\Vert \nabla^{N}f\right\Vert _{0},\quad\left[f\right]_{N+\alpha}=\left[\nabla^{N}f\right]_{C^{0,\alpha}},
\end{align*}
and
\begin{align*}
\left\Vert f\right\Vert _{N+\alpha} & =\left\Vert f\right\Vert _{C^{N,\alpha}}{\,\coloneqq\,}\left\Vert f\right\Vert _{N}+\left[f\right]_{N+\alpha},
\end{align*}
where $\left[\,\cdot\,\right]_{C^{0,\alpha}}$ is the Holder seminorm.
In this context, we record the elementary inequality
\[
\left\Vert fg\right\Vert _{r}\lesssim\left\Vert f\right\Vert _{0}\left[g\right]_{r}+\left[f\right]_{r}\left\Vert g\right\Vert _{0}\quad\text{for any }r>0.
\]

\vspace{0.2in}

As in \cite{bulutEpochsRegularityWild2022}, we recall the Hodge decomposition
\[
\mathrm{Id}=\mathcal{P}_{1}+\mathcal{P}_{2}+\mathcal{P}_{3}
\]
where $\mathcal{P}_{1}{\,\coloneqq\,}d\left(-\Delta\right)^{-1}\delta$ and $\mathcal{P}_{2}{\,\coloneqq\,}\delta\left(-\Delta\right)^{-1}d$
and $\mathcal{P}_{3}$ maps to harmonic forms (cf. \cite[Section 5.8]{Taylor_PDE1}).
We note that $\mathcal{P}_{1}$ and $\mathcal{P}_{2}$ are Calder\'on-Zygmund
operators.  Note that $\delta=-\Div$, where $\left(\Div T\right)^{i_{1}...i_{k}}=\nabla_{j}T^{ji_{1}...i_{k}}$
for any tensor $T$.  Due to the musical isomorphism, the Hodge projections $\mathcal{P}_{i}$
are also defined on vector fields, and we also write $\sharp\mathcal{P}_{i}\flat$
as $\mathcal{P}_{i}$ for convenience (unless ambiguity arises).  Moreover, because the torus is flat, we have the identities 
\begin{align}
\delta\flat\left(X\cdot\nabla Y\right) & =\delta\flat\left(Y\cdot\nabla X\right)\label{eq:ident_P10}
\end{align}
and
\begin{align}
\mathcal{P}_{1}\left(X\cdot\nabla Y\right) & =\mathcal{P}_{1}\left(Y\cdot\nabla X\right)\label{eq:ident_P1}
\end{align}
for any pair of divergence-free vector fields $X,Y$.  We also recall that, on the torus, harmonic 1-forms (or vector fields) are precisely those 
which have mean zero.  

%
%

\subsection{Leray projection, anti-divergence and Biot-Savart operators}

We define the usual Leray projection
\[
\mathbb{P}{\,\coloneqq\,}\mathcal{P}_{2}+\mathcal{P}_{3},
\]
and note that velocity fields of incompressible fluids are in the image of $\mathbb{P}$.

We will frequently make use of the antidivergence operator $\mathcal{R}:C^{\infty}\left(\mathbb{T}^{d},\mathbb{R}^{d}\right)\to C^{\infty}\left(\mathbb{T}^{d},\mathcal{S}_{0}^{d\times d}\right)$, defined by
\begin{align}
\left(\mathcal{R}v\right)_{ij} & =\mathcal{R}_{ijk}v^{k},\label{eq:antidiv}
\end{align}
\begin{align}
\mathcal{R}_{ijk} & {\,\coloneqq\,}-\frac{d-2}{d-1}\Delta^{-2}\partial_{i}\partial_{j}\partial_{k}-\frac{1}{d-1}\Delta^{-1}\partial_{k}\delta_{ij}+\Delta^{-1}\partial_{i}\delta_{jk}+\Delta^{-1}\partial_{j}\delta_{ik}.\label{eq:antidiv2} 
\end{align}

We will also frequently use the fact that $\div\mathcal{R}v=v-\dashint_{\mathbb{T}^{d}}v=\left(1-\mathcal{P}_{3}\right)v$
for any vector field $v$.  Moreover, via the musical isomorphism, $\mathcal{R}$
is also defined on 1-forms, and we also write $\mathcal{R}\sharp$
as $\mathcal{R}$ for convenience.

We define the higher-dimensional analogue of the Biot-Savart operator as 
\begin{align}
\mathcal{B}{\,\coloneqq\,}\left(-\Delta\right)^{-1}d\flat,\label{def-B}
\end{align}
mapping from vector fields to 2-forms.  Note that with this definition, we have 
\[
\sharp\delta\mathcal{B}=\mathcal{P}_{2},
\]
which implies that $\sharp\delta\mathcal{B}v=v-\dashint_{\mathbb{T}^{d}}v=\mathcal{P}_{2}v$ for any divergence-free vector field $v$.

\subsection{Formulation of the Main Iterative Scheme}
\label{subsec:formulation_iter}

As noted in the introduction, to prove Theorem~\ref{thm:thm1} we introduce a novel iterative construction, involving alternating 
applications of convex integration techniques.  To specify this further, we now introduce the main iteration lemma.  We begin by 
specifying several parameters.

Fix $\beta<\frac{1}{2}$ and $T\geq1$.  For any $q\in\mathbb{Z}_{\geq-1}$,
we set 
\begin{align}
\lambda_{q} & {\,\coloneqq\,}\left\lceil a^{(b^{q})}\right\rceil \\
\delta_{q} & {\,\coloneqq\,}\lambda_{q}^{-2\beta},
\end{align}
with $a\gg1$, $0<b-1\ll1$ (to be chosen later).  The quantity $\lambda_{q}$ will be the frequency parameter (made an integer for phase functions), while $\delta_{q}$ will be the pointwise size of the Reynolds stress.

For sufficiently small $\alpha>0$ (to be specified later in the argument) and any $q\in\mathbb{N}_{0}$, we set 
\begin{align}
\epsilon_{q} & {\,\coloneqq\,}\lambda_{q}^{-\sigma},\\
\ell_{q} & {\,\coloneqq\,}\lambda_{q}^{-\frac{1}{4}}\lambda_{q+1}^{-\frac{3}{4}},\label{eq:length_scale}
\end{align}
and
\begin{align}
\tau_{q} & {\,\coloneqq\,}C_{q}\delta_{q+1}^{-\frac{1}{2}}\lambda_{q}^{-1-3\alpha},\label{eq:tau_q}
\end{align}
where $C_{q}>0$ is an unremarkable constant which is chosen to make $\epsilon_{q-1}\tau_{q-1}\tau_{q}^{-1}$ an integer.  Note that, since $\epsilon_{q-1}\tau_{q-1}\tau_q^{-1}\gg1$, we can choose $C_q$ so that it is comparable to $1$ (independently of $q$); it follows that $C_q$ has no impact on the estimates and will be omitted in the sequel.

Here, $\ell_{q}$ is the mollification length scale which, unless otherwise noted, we refer to as $\ell$ for brevity, $\tau_{q}$ is the time scale for the local existence and gluing step, and $\epsilon_{q}\tau_{q}$ is the smaller time scale of the overlapping epoch between adjacent temporal cutoffs. 

We also set $\epsilon_{-1}=\lambda_{-1}^{-\sigma}$ and $\tau_{-1}$
to be any positive number such that 
\[
1\leq 15\epsilon_{-1}\tau_{-1}\leq T.
\]
We note that $a,b,\alpha,\sigma$ do not depend on $q$.

To formulate the main inductive hypothesis for our applications of convex integration, suppose that for some $q\in\mathbb{N}_{0}$ we have smooth fields $\left(u_{q},v_{q},F_{q},R_{q}\right)$
such that, for 
\begin{itemize}
\item a geometric constant $M>1$ depending only on $d$ (and not $a,\beta,b,\sigma,\alpha,q$) to be chosen later in \Secref{Convex-integration-and}, and
\item a positive sequence $\mathcal{A}=\left(A_{\kappa}\right)_{\kappa\in\mathbb{N}_{0}}$ and a sequence $\left(B_{\kappa}\right)_{\kappa\in\mathbb{N}_{0}}$ completely determined by $M$ (and therefore by $d$), with $(A_{\kappa})$ and $(B_{\kappa})$ both independent of $q$,
\end{itemize}
the following criteria hold:
\begin{enumerate}
\item there exist smooth pressures $p_{q}$ and $\pi_{q}$ solving the dual Euler-Reynolds systems (\ref{eq:forced_Euler_dual_system}) on $[0,T]\times\mathbb{T}^{d}$,
\item we have the estimates 
\begin{align}
\|v_{q}\|_{0},\|u_{q}\|_{0},\left\Vert F_{q}\right\Vert _{0} & \leq1-\delta_{q}^{1/2},\label{eq:uvF_sup}
\end{align}
and, for $1\leq \jmath\leq 12$,
\begin{align}
\|\grad^{\jmath}u_{q}\|_{0} & \leq M\lambda_{q}^{\jmath}\delta_{q}^{1/2},\label{eq:grad_uq_sup}\\
\|\nabla^{\jmath}v_{q}\|_{0} & \leq M\lambda_{q-1}^{\jmath}\delta_{q-1}^{1/2},\label{eq:grad_vq_sup}\\
\|\nabla^{\jmath}F_{q}\|_{0} & \leq M\epsilon_{q}\lambda_{q}^{\jmath-3\alpha}\delta_{q+1},\label{eq:Fq_sup}
\end{align}
as well as, for $0\leq \jmath\leq 12$,
\begin{align}
\|\nabla^{\jmath}R_{q}\|_{0} & \leq M\epsilon_{q}\lambda_{q}^{\jmath-3\alpha}\delta_{q+1},\label{eq:Rq_sup}
\end{align}
\item there exists a set of ``bad'' times $\mathcal{B}_{q}=\bigcup_{i}I_{i}^{b,q}$ which is a union of disjoint closed intervals $I_{i}^{b,q}$
of length $5\epsilon_{q-1}\tau_{q-1}$, and, defining the ``good'' times
\[
\mathcal{G}_{q}=[0,T]\setminus\mathcal{B}_{q}=\bigcup_{i}I_{i}^{g,q}
\]
consisting of disjoint open intervals $I_{i}^{g,q}$, we have 
\begin{equation}
R_{q}\big|_{\mathcal{G}_{q}+B(0,\epsilon_{q-1}\tau_{q-1})}\equiv0,\label{eq:Rq_zerowhere}
\end{equation}
where $\mathcal{G}_{q} +  B(0,\epsilon_{q-1}\tau_{q-1})$ denotes the $\epsilon_{q -1} \tau_{q-1}$-neighborhood of $\mathcal{G}_{q}$, and
\item for $t\in\mathcal{G}_{q}+B(0,\epsilon_{q-1}\tau_{q-1})$, we have, for $1\leq \jmath\leq 8$ and $\kappa\geq 0$,
\begin{align}
\|\grad^{\jmath+\kappa}v_{q}\|_{0}+\|\grad^{\jmath+\kappa}u_{q}\|_{0} & \leq\left(A_{\kappa}+B_{\kappa}\right)\ell_{q-1}^{-\kappa}\lambda_{q-1}^{\jmath}\delta_{q-1}^{1/2}, \label{eq:uq_vq_goodbounds}
\end{align}
and, for $1\leq \jmath\leq 10$ and $\kappa\geq 0$,
\begin{align}
\|\grad^{\jmath+\kappa}F_{q}\|_{0} & \leq\left(A_{\kappa}+B_{\kappa}\right)\ell_{q-1}^{-\kappa}\lambda_{q-1}^{\jmath-3\alpha}\delta_{q}.\label{eq:Fq_goodbounds}
\end{align}
\end{enumerate}

We are now ready to state our main iterative proposition.

\begin{prop}[Iteration scheme]
\label{prop:iterative}Fix $\beta<\frac{1}{2}$ and $T\geq1$, and suppose
that 
\begin{gather}
0<b-1\ll_{\beta}1,\label{eq:smallness_b}\\
0<\sigma<\frac{(b-1)(1-2b\beta)}{b+1},\label{eq:sigma_bound}\\
0<\alpha\ll_{\sigma,b,\beta}1,\label{eq:alpha_bound}
\end{gather}
and
\begin{align}
a\gg_{\mathcal{A},\alpha,\sigma,b,\beta}1.\nonumber 
\end{align}
Fixing $q\in\mathbb{N}_{0}$, if $(u_{q},v_{q},F_{q},R_{q},\mathcal{B}_{q})$
satisfy the assumptions (1)--(4) above, then there exist $(u_{q+1},v_{q+1},\allowbreak F_{q+1},\allowbreak R_{q+1},\mathcal{B}_{q+1})$
satisfying 
\begin{align}
\|v_{q+1}\|_{0},\|u_{q+1}\|_{0},\left\Vert F_{q+1}\right\Vert _{0} & \leq1-\delta_{q+1}^{1/2},\label{eq:vq_sup-1}
\end{align}
and, for $1\leq \jmath \leq 12$,
\begin{align}
\|\nabla^{\jmath}u_{q+1}\|_{0} & \leq M\lambda_{q}^{\jmath}\delta_{q}^{1/2},\label{eq:grad_uq_sup-1}\\
\|\nabla^{\jmath}v_{q+1}\|_{0} & \leq M\lambda_{q+1}^{\jmath}\delta_{q+1}^{1/2},\label{eq:grad_vq_sup-1}\\
\|\nabla^{\jmath}F_{q+1}\|_{0} & \leq M\epsilon_{q+1}\lambda_{q+1}^{\jmath-3\alpha}\delta_{q+2},\label{eq:Fq_sup-1}
\end{align}
as well as, for $0\leq \jmath\leq 12$,
\begin{align}
\|\nabla^{\jmath}R_{q+1}\|_{0} & \leq M\epsilon_{q+1}\lambda_{q+1}^{\jmath-3\alpha}\delta_{q+2}.\label{eq:Rq_sup-1}
\end{align}

Moreover, the estimates (\ref{eq:Rq_zerowhere})-(\ref{eq:Fq_goodbounds}) remain true with $q$ replaced by $q+1$, and we have
\begin{align}
\|v_{q}-v_{q+1}\|_{0}+\lambda{}_{q+1}^{-1}\|v_{q}-v_{q+1}\|_{1} & \leq M\delta_{q+1}^{1/2},\label{eq:iter_prop_vq_est}\\
\|u_{q}-u_{q+1}\|_{0}+\lambda{}_{q+1}^{-1}\|u_{q}-u_{q+1}\|_{1} & \leq M\delta_{q+1}^{1/2},\label{eq:iter_prop_uq_est}\\
\|F_{q}-F_{q+1}\|_{0}+\lambda{}_{q+1}^{-1}\|F_{q}-F_{q+1}\|_{1} & \leq M\delta_{q+1},\label{eq:iter_prop_Fq_est}
\end{align}
and
\begin{align}
v_{q+1}=v_{q},\;u_{q+1}=u_{q},\;F_{q+1}=F_{q} & \text{ on }\mathcal{G}_{q}\times\mathbb{T}^{d}\label{eq:uvF_on_Gq},
\end{align}
with
\begin{align}
\mathcal{G}_{q}\subset\mathcal{G}_{q+1},\left|\mathcal{B}_{q+1}\right| & \leq\epsilon_{q}\left|\mathcal{B}_{q}\right|.\label{eq:Gq_Bq}
\end{align}
\end{prop}

It is important to note that in the statement of \Propref{iterative}, the parameters $b$, $\sigma$, $\alpha$, and $a$ depend only on $\beta$ and $d$.  In particular, they do not depend on $q$ or $T$ (having assumed $T\geq1$).  We will prove \Propref{iterative} in Section \ref{sec:proofProp} through Section \ref{sec:proofProp2} below.  

We conclude this section with a few comments on the inductive assumptions (1)--(4).  The parameters $\epsilon_{q}$ in (\ref{eq:Fq_sup}) and (\ref{eq:Rq_sup}) serve to compensate for the sharp time cutoffs in our gluing construction; this strategy was previously used in \cite{derosaDimensionSingularSet2021,bulutEpochsRegularityWild2022} to obtain convex integration constructions with epochs of regularity.

The role of $A_{\kappa}$ and $B_{\kappa}$ in (\ref{eq:uq_vq_goodbounds})-(\ref{eq:Fq_goodbounds}) is subtle.  The bound by $A_{\kappa}$ is included to ensure that the bounds are satisfied in the initial iteration (when $q=0$).  On the other hand, $B_{\kappa}$ is determined by $M$ within the construction.  In particular, $B_{\kappa}$ must not depend on $B_{\kappa+1}$ or $B_{\kappa+2}$ (which is a loss of derivatives issue), and we have the following diagram: 

\begin{tabular}{ccV{\linewidth}}
\noalign{\vskip0.2cm}
\begin{cellvarwidth}[t]
\centering
$M$ from (\ref{eq:grad_uq_sup})-(\ref{eq:Rq_sup})

(finitely many estimates)
\end{cellvarwidth} & $\leadsto$ & Determine every $B_{\kappa}$ in (\ref{eq:uq_vq_goodbounds})-(\ref{eq:Fq_goodbounds})

with $q$ replaced by $q+1$.\tabularnewline
\noalign{\vskip0.5cm}
\begin{cellvarwidth}[t]
\centering
Finitely many $A_{\kappa}+B_{\kappa}$

from (\ref{eq:uq_vq_goodbounds})-(\ref{eq:Fq_goodbounds})
\end{cellvarwidth} & $\leadsto$ & Recover $M$ in (\ref{eq:grad_uq_sup-1})-(\ref{eq:iter_prop_Fq_est})

with $q$ replaced by $q+1$ (giving finitely 

many lower bounds for $a$).\tabularnewline
\end{tabular}

\vspace{.13in}

Note that we need to keep track of constants when determining the sequence $\left(B_{\kappa}\right)$ (in (\ref{eq:Fl_estimate}), (\ref{eq:Bk_2nd_choice}), (\ref{eq:vj_bound})), to avoid the loss of derivatives.  Apart from this issue, we will suppress the dependence on $A_{\kappa}$ and $B_{\kappa}$ within the notation $\ls_{\kappa}$.

\section{Proof of Theorem~\ref{thm:thm1}} \label{sec:Proof-of-main-thm}


In this section, we give the proof of Theorem~\ref{thm:thm1}, using the main iterative proposition, \Propref{iterative}.  Indeed, the first observation in this direction is that \Propref{iterative} has the following immediate consequence.

\begin{cor}
\label{cor:iteration_corollary} Let $(u_q,v_q,F_q,R_q,\mathcal{B}_q)$ be as in \Propref{iterative}, and let $(u_{q+1},v_{q+1},\allowbreak F_{q+1},R_{q+1},\mathcal{B}_{q+1})$
be as constructed in \Propref{iterative}.  Then, setting $$\widetilde{F}_{q+1}:=F_{q+1}+R_{q+1}, \quad \widetilde{R}_{q+1}:=-R_{q+1},$$
let $(v_{q+2},u_{q+2},\allowbreak\widetilde{F}_{q+2},\allowbreak\widetilde{R}_{q+2},\mathcal{B}_{q+2})$ be the result of applying
\Propref{iterative} to $$(v_{q+1},u_{q+1},\allowbreak\widetilde{F}_{q+1},\allowbreak\widetilde{R}_{q+1},\mathcal{B}_{q+1}).$$

Then, setting $$F_{q+2}:=\widetilde{F}_{q+2}+\widetilde{R}_{q+2}\quad\textrm{and}\quad R_{q+2}:=-\widetilde{R}_{q+2},$$ we have that (\ref{eq:forced_Euler_dual_system}), (\ref{eq:uvF_sup})-(\ref{eq:Rq_sup}), and (\ref{eq:Rq_zerowhere})-(\ref{eq:uq_vq_goodbounds}) all hold with $q$ replaced by $q+2$.  Moreover, (\ref{eq:iter_prop_vq_est})-(\ref{eq:iter_prop_Fq_est}) also holds with $q$ replaced by $q+1$.
\end{cor}

By combining \Propref{iterative} and \Corref{iteration_corollary}, we create a closed iteration loop, in what we call \emph{alternating convex integration}.

\begin{proof}[Proof of Theorem~\ref{thm:thm1}]
Define $V_{1}(t,x){\,\coloneqq\,}V_{1}(x)$ for all $t\in[0,T]$, and note that by the usual local existence of smooth (unforced) Euler solutions, we obtain exact Euler solutions $V_{2}(t,x)$ and $V_{3}(t,x)$ for $t\in[T-\varepsilon,T]$ with $V_2(T,x)=V_2(x)$ and $V_3(T,x)=V_3(x)$, where $\varepsilon\in\left(0,\frac{T}{4}\right)$ depends on $V_{2}$ and $V_{3}$. 

Observe that 
\[
\int_{\mathbb{T}^{d}}\partial_{t}V_{i}=0,\quad i\in \{2,3\},
\]
so that $\int_{\mathbb{T}^{d}}V_{1}(t)=\int_{\mathbb{T}^{d}}V_{2}(t)=\int_{\mathbb{T}^{d}}V_{3}(t)$
for all $t$.

Let $\eta$ be a smooth temporal cutoff on $[0,T]$ such that $\mathbf{1}_{[0,T-\frac{3}{5}\varepsilon]}\geq\eta\geq\mathbf{1}_{[0,T-\frac{2}{5}\varepsilon]}$, and set
\begin{align*}
v_{0}{\,\coloneqq\,}\eta V_{1}+\left(1-\eta\right)V_{2},\quad u_{0}{\,\coloneqq\,}\eta V_{1}+\left(1-\eta\right)V_{3}.
\end{align*}

We observe that 
\begin{align*}
\partial_{t}u_{0}+\divop(u_{0}\otimes u_{0}) & =\partial_{t}\eta\left(V_{1}-V_{3}\right)+\eta\left(\partial_{t}V_{1}+\divop(V_{1}\otimes V_{1})\right)\\
&\hspace{0.2in} +\left(1-\eta\right)\left(\partial_{t}V_{3}+\divop(V_{3}\otimes V_{3})\right)\\
&\hspace{0.2in} +\left(\eta^{2}-\eta\right)\divop\left(\left(V_{1}-V_{3}\right)\otimes\left(V_{1}-V_{3}\right)\right).
\end{align*}
Note that $\partial_{t}V_{1}=0$, and as $V_{3}$ is an exact Euler
solution there is a smooth pressure $P_{3}$ such that $\partial_{t}V_{3}+\divop(V_{3}\otimes V_{3})=-\nabla P_{3}$. 

Let $\mathcal{R}$ denote the antidivergence operator defined in (\ref{eq:antidiv})--(\ref{eq:antidiv2}),
and define
\begin{align*}
F_{0} & {\,\coloneqq\,}\partial_{t}\eta\mathcal{R}\left(V_{1}-V_{3}\right)+\eta V_{1}\otimes V_{1}-\eta\left(1-\eta\right)\left(V_{1}-V_{3}\right)\otimes\left(V_{1}-V_{3}\right),\\
R_{0} & {\,\coloneqq\,}\partial_{t}\eta\mathcal{R}\left(V_{1}-V_{2}\right)+\eta V_{1}\otimes V_{1}-\eta\left(1-\eta\right)\left(V_{1}-V_{2}\right)\otimes\left(V_{1}-V_{2}\right)-F_{0}.
\end{align*}
Then $\left(u_{0},v_{0},F_{0},R_{0}\right)$ is a smooth solution of (\ref{eq:forced_Euler_dual_system}).  We note that $R_{0}=0$ on $[0,T-\frac{3}{5}\varepsilon]\cup[T-\frac{2}{5}\varepsilon,T]$.  Thus we can set $\mathcal{B}_{0}=[T-\frac{2}{3}\varepsilon,T-\frac{1}{3}\varepsilon].$

We now rescale in time, setting, for $\zeta>0$,
\begin{align*}
u_{0}^{\zeta}\left(t,x\right) & {\,\coloneqq\,}\zeta u_{0}\left(\zeta t,x\right),\quad F_{0}^{\zeta}{\,\coloneqq\,}\zeta^{2}F_{0}\left(\zeta t,x\right)\\
v_{0}^{\zeta}\left(t,x\right) & {\,\coloneqq\,}\zeta v_{0}\left(\zeta t,x\right),\quad R_{0}^{\zeta}{\,\coloneqq\,}\zeta^{2}R_{0}\left(\zeta t,x\right)\\
\mathcal{B}_{0}^{\zeta} & {\,\coloneqq\,}\zeta^{-1}\mathcal{B}_{0}
\end{align*}

For $\zeta$ small enough (depending on $V_{1}$, $V_{2}$,
$V_{3}$, and $T$), the five-tuple $$(u_{0}^{\zeta},v_{0}^{\zeta},F_{0}^{\zeta},R_{0}^{\zeta},\mathcal{B}_{0}^{\zeta})$$
satisfies the conditions (\ref{eq:uvF_sup})-(\ref{eq:Rq_sup}) for \Propref{iterative}
on the interval $[0,\zeta^{-1}T]$ where $\zeta^{-1}T\geq\zeta^{-1}\varepsilon\geq1$.
For the first step of the induction, we pick any positive $\tau_{-1}$
such that $5\epsilon_{-1}\tau_{-1}=\frac{1}{3}\zeta^{-1}\varepsilon$.
Then (\ref{eq:Rq_zerowhere}) is satisfied.  We also pick $\left(A_{\kappa}\right)_{\kappa\in\mathbb{N}_{0}}$
so that (\ref{eq:uq_vq_goodbounds})-(\ref{eq:Fq_goodbounds}) are
satisfied.

By iteratively applying \Propref{iterative} and \Corref{iteration_corollary}, we obtain
a sequence $$\left(u_{q}^{\zeta},v_{q}^{\zeta},F_{q}^{\zeta},R_{q}^{\zeta},\mathcal{B}_{q}^{\zeta}\right)$$ 
such that $\left(u_{q}^{\zeta}\right)_{q\in\mathbb{N}_{0}}$ converges in $C_{t}^{0}C_{x}^{\beta-}$ to some $u^{\zeta}$,
$\left(v_{q}^{\zeta}\right)_{q\in\mathbb{N}_{0}}$ converges in $C_{t}^{0}C_{x}^{\beta-}$ to some $v^{\zeta}$, and
$\left(F_{q}^{\zeta}\right)_{q\in\mathbb{N}_{0}}$ converges in $C_{t}^{0}C_{x}^{2\beta-}$ to some $F^{\zeta}$, with
$$R_{q}^{\zeta}\rightarrow 0\quad\textrm{in} C_{t,x}^0,$$
$$\mathcal{B}_{q+1}^{\zeta}\subset\mathcal{B}_{q}^{\zeta},$$
and such that one has $v^{\zeta}=v_{q}^{\zeta}$, $u^{\zeta}=u_{q}^{\zeta}$, and $F^{\zeta}=F_{q}^{\zeta}$ on $\mathcal{G}_{q}^{\zeta}\times\mathbb{T}^{d}$, and are thus smooth.

In particular, $u^{\zeta}$ and $v^{\zeta}$ are weak solutions of
(\ref{eq:Euler_start}).  We can then revert the time-rescaling by setting
\begin{align*}
u_{q}\left(t,x\right) & {\,\coloneqq\,}\zeta^{-1}u_{q}^{\zeta}\left(\zeta^{-1}t,x\right), & u\left(t,x\right) & {\,\coloneqq\,}\zeta^{-1}u^{\zeta}\left(\zeta^{-1}t,x\right),\\
v_{q}\left(t,x\right) & {\,\coloneqq\,}\zeta^{-1}v_{q}^{\zeta}\left(\zeta^{-1}t,x\right), & v\left(t,x\right) & {\,\coloneqq\,}\zeta^{-1}v^{\zeta}\left(\zeta^{-1}t,x\right),\\
F_{q}\left(t,x\right) & {\,\coloneqq\,}\zeta^{-2}F_{q}^{\zeta}\left(\zeta^{-1}t,x\right), & F\left(t,x\right) & {\,\coloneqq\,}\zeta^{-2}F^{\zeta}\left(\zeta^{-1}t,x\right),
\end{align*}
\begin{align*}
\mathcal{B}_{q} & {\,\coloneqq\,}\zeta\mathcal{B}_{q}^{\zeta}.
\end{align*}
 The bad set of times in the limit is then
\[
\mathcal{B}{\,\coloneqq\,}\bigcap_{q}\mathcal{B}_{q}.
\]
By the same calculations as in \cite[Proof of Theorem 1]{bulutEpochsRegularityWild2022}, we have
\begin{align*}
\dim_{\mathcal H}(\mathcal{B})\leq1-\frac{\sigma b}{\left(b-1\right)\left(1+3\alpha+\sigma-\beta\right)}.
\end{align*}
Upon choosing $\alpha$ sufficiently small, then $\sigma$
sufficiently close to $\frac{(b-1)\left(1-2b\beta\right)}{b+1}$, and finally
$b>1$ sufficiently close to $1$, we obtain 
\[
\dim_{\mathcal H}(\mathcal{B})\leq\left(\frac{1}{2(1-\beta)}\right)^{+}
\]
as desired.

Now we show the time-regularity of $u$ (and similarly for $v$).
Let $\beta>\beta_{1}>\beta_{2}\gg\mu>0$.  Then $u\in C_{t}^{0}C_{x}^{\beta_{1}}$.
Let $\psi_{\ell}$ be a smooth standard radial mollifier in space
of length $\ell$.  For any $\varepsilon>0$ small, we write
\begin{align*}
u^{\varepsilon} & =u*\psi_{\varepsilon} & \left(u\otimes u\right)^{\varepsilon} & =\left(u\otimes u\right)*\psi_{\varepsilon} & F^{\varepsilon}=F*\psi_{\varepsilon}
\end{align*}
Observe that $\partial_{t}u^{\varepsilon}+\mathbb{P}\divop\left(u\otimes u\right)^{\varepsilon}=\mathbb{P}\divop F^{\varepsilon}$,
so 
\begin{align*}
\left\Vert \partial_{t}u^{\varepsilon}\right\Vert _{\mu} & \ls_{\mu}\left\Vert \mathbb{P}\left(\divop\left(u\otimes u\right)^{\varepsilon}\right)\right\Vert _{\mu}+\left\Vert \mathbb{P}\divop F^{\varepsilon}\right\Vert _{\mu}\\
 & \ls_{\beta_{1},\mu}\varepsilon^{\beta_{1}-1-\mu}\left\Vert u\right\Vert _{\beta_{1}}^{2}+\varepsilon^{2\beta_{1}-1-\mu}\left\Vert F\right\Vert _{2\beta_{1}}
\end{align*}

and
\begin{align*}
\left\Vert u^{2\varepsilon}-u^{\varepsilon}\right\Vert _{C_{x}^{0}C_{t}^{\beta_{2}}} & \ls_{\beta_{2},\mu}\left\Vert u^{2\varepsilon}-u^{\varepsilon}\right\Vert _{0}^{1-\beta_{2}}\left\Vert \partial_{t}u^{2\varepsilon}-\partial_{t}u^{\varepsilon}\right\Vert _{\mu}^{\beta_{2}}\\
 & \ls_{\beta_{1},\beta_{2},\mu}\left(\varepsilon^{\beta_{1}}\left\Vert u\right\Vert _{\beta_{1}}\right)^{1-\beta_{2}}\left(\varepsilon^{\beta_{1}-1-\mu}\left(\left\Vert u\right\Vert _{\beta_{1}}^{2}+\left\Vert F\right\Vert _{2\beta_{1}}\right)\right)^{\beta_{2}}
\end{align*}
The power of $\varepsilon$ is $\beta_{1}\left(1-\beta_{2}\right)+\left(\beta_{1}-1-\mu\right)\beta_{2}=\beta_{1}-\beta_{2}-\mu\beta_{2}$.
It follows that if we choose $\mu=\mu\left(\beta_{1},\beta_{2}\right)$ small enough,
then $\beta_{1}-\beta_{2}-\mu\beta_{2}>0$ and $\left(u^{2^{-n}}\right)_{n\in\mathbb{N}_{1}}$
converges in $C_{x}^{0}C_{t}^{\beta_{2}}$ by geometric series.  Then
$u\in C_{t,x}^{\beta_{2}}$ and so is $v$.
\end{proof}

\begin{rem*}
The time-regularity of $u$ and $v$ just requires $F\in C_{t}^{0}C_{x}^{\beta-}$.
\end{rem*}

\section{Beginning of the proof of \Propref{iterative}: mollification and gluing estimates}
\label{sec:proofProp}

In this section, we begin the proof of \Propref{iterative}.  As described in the introduction, the proof is based on convex integration techniques, and
consists of several steps: an initial mollification procedure, followed by a delicate balance of gluing estimates (between good and bad intervals, as described
below) and perturbation estimates, which allows one to close the iteration in the proof of \Propref{iterative}.  We perform the mollification procedure and derive 
the relevant gluing estimates in this section.  The perturbation estimates are then established in \Secref{Convex-integration-and} below, while the proof of 
\Propref{iterative} is completed in \Secref{proofProp2}.

\vspace{0.2in}

We begin by recalling that there are several natural relationships inherent in the choice of parameters described in \Subsecref{formulation_iter}.  We record these
in Appendix \ref{app:parameters}, and use them freely in the sequel.

\subsection{Mollification}
\label{subsec:mollifier}

We now introduce the mollification scheme.  With $\psi_{\ell}$ a smooth radial mollifier in space at the length scale $\ell$ defined in (\ref{eq:length_scale}), set
\begin{align*}
u_{\ell}{\,\coloneqq\,}\psi_{\ell}*u_{q},\quad v_{\ell}{\,\coloneqq\,}\psi_{\ell}*v_{q}.
\end{align*}

By (\ref{eq:grad_uq_sup}) and (\ref{eq:grad_vq_sup}) we then have 
\begin{align}
\|v_{\ell}\|_{\jmath+\kappa} & \lesssim_{\kappa}\lambda_{q-1}^{\jmath}\delta_{q-1}^{\frac{1}{2}}\ell^{-\kappa}\label{eq:grad_v_l}\\
\|u_{\ell}\|_{\jmath+\kappa} & \lesssim_{\kappa}\lambda_{q}^{\jmath}\delta_{q}^{\frac{1}{2}}\ell^{-\kappa}\label{eq:grad_u_l}
\end{align}
for any $1\leq\jmath\leq12,0\leq\kappa$.  Moreover, setting
\begin{align}
F_{\ell} & {\,\coloneqq\,}\psi_{\ell}*F_{q}-\psi_{\ell}*\left(u_{q}\otimes u_{q}\right)+u_{\ell}\otimes u_{\ell}\label{eq:fl_definition}\\
R_{\ell} & {\,\coloneqq\,}\psi_{\ell}*R_{q}+\psi_{\ell}*\left(u_{q}\otimes u_{q}\right)-u_{\ell}\otimes u_{\ell}-\psi_{\ell}*\left(v_{q}\otimes v_{q}\right)+v_{\ell}\otimes v_{\ell},
\end{align}
we observe that $(u_{\ell},v_{\ell},F_{\ell},R_{\ell})$ solves (\ref{eq:forced_Euler_dual_system})
for suitable choices of the pressures $\pi_{\ell},p_{\ell}$. 

We recall the usual commutator estimates; see, e.g., \cite{conti2012h}.  For any $f,g\in C^{\infty}\left(\mathbb{T}^{d}\right)$ and $l>0,r\geq0$, one has
\begin{align}
\left\Vert f-f*\psi_{l}\right\Vert _{r} & \ls_{r}l^{2}\left\Vert f\right\Vert _{r+2}\nonumber 
\end{align}
and
\begin{align}
\left\Vert \left(f*\psi_{l}\right)\left(g*\psi_{l}\right)-\left(fg\right)*\psi_{l}\right\Vert _{r} & \lesssim_{r}l^{2-r}\left\Vert f\right\Vert _{1}\left\Vert g\right\Vert _{1}.\label{comm1}
\end{align}
Moreover, applying the product rule, one has
\begin{align*}
\left\Vert \left(f*\psi_{l}\right)\left(g*\psi_{l}\right)-\left(fg\right)*\psi_{l}\right\Vert _{\jmath+\kappa} & \lesssim_{\jmath,\kappa}l^{2-\kappa}\sum_{i=0}^{\jmath}\left\Vert f\right\Vert _{1+i}\left\Vert g\right\Vert _{1+\jmath-i}\\
 & \ls_{\jmath}l^{2-\kappa}\left(\left\Vert f\right\Vert _{1}\left\Vert g\right\Vert _{1+\jmath}+\left\Vert f\right\Vert _{1+\jmath}\left\Vert g\right\Vert _{1}\right)
\end{align*}
for any $\jmath,\kappa\in\mathbb{N}_{0}$.

\subsection{Preliminary estimates}

Before proceeding, we establish several preliminary estimates for these mollified quantities, which will be used frequently in the sequel.  For $0\leq \jmath\leq 10$, by (\ref{eq:grad_uq_sup}) and (\ref{eq:double_skip_iteration})
we have 
\begin{equation}
\|u_{\ell}-u_{q}\|_{\jmath}\ls\ell^{2}\left\Vert u_{q}\right\Vert _{\jmath+2}\ls\ell^{2}\lambda_{q}^{\jmath+2}\delta_{q}^{\frac{1}{2}}\ls\lambda_{q}^{\jmath}\frac{\lambda_{q}}{\lambda_{q+1}}\delta_{q}^{\frac{1}{2}}\ll\epsilon_{q+1}\lambda_{q+1}^{\jmath}\delta_{q+2}^{\frac{1}{2}}\label{eq:ul-uq}
\end{equation}

Next, for $\kappa\geq0$ and $1\leq \jmath\leq 10$, we have the estimates
\begin{align}
\|F_{\ell}\|_{\jmath+\kappa} & \leq C(\kappa)\ell^{-\kappa}\left\Vert F_{q}\right\Vert _{\jmath}+C(\kappa)\ell^{2-\kappa}\left\Vert u_{q}\right\Vert _{1+\jmath}\left\Vert u_{q}\right\Vert _{1}\nonumber \\
 & \leq C(\kappa)\ell^{-\kappa}M\epsilon_{q}\lambda_{q}^{\jmath-3\alpha}\delta_{q+1}+C(\kappa)\ell^{2-\kappa}M^{2}\lambda_{q}^{2+\jmath}\delta_{q}\nonumber \\
 & \leq B_{\kappa,1}\ell^{-\kappa}\epsilon_{q}\lambda_{q}^{\jmath-3\alpha}\delta_{q+1}\label{eq:Fl_estimate}
\end{align}
for a sufficiently large choice of $B_{\kappa,1}$, because of (\ref{eq:grad_uq_sup}),
(\ref{eq:Fq_sup}), and (\ref{eq:ell_smaller}).  In particular, the
constant $B_{\kappa,1}$ in (\ref{eq:Fl_estimate}) does not depend
on any $A_{\kappa}+B_{\kappa}$, and will help determine $B_{\kappa}$
in (\ref{eq:uq_vq_goodbounds})-(\ref{eq:Fq_goodbounds}) later.

On the other hand, in order to avoid loss of derivatives, for $\kappa\geq0$ and $1\leq \jmath\leq 12$, we have 
\begin{align}
\|F_{\ell}\|_{\jmath+\kappa} & \lesssim_{\kappa}\ell^{-\kappa}\left\Vert F_{q}\right\Vert _{\jmath}+\ell^{2-\left(\jmath-1\right)-\kappa}\left\Vert u_{q}\right\Vert _{2}\left\Vert u_{q}\right\Vert _{1}\nonumber \\
 & \ls\ell^{-\kappa}\epsilon_{q}\lambda_{q}^{\jmath-3\alpha}\delta_{q+1}+\ell^{-\jmath-\kappa}\ell^{3}\lambda_{q}^{3}\delta_{q}\nonumber \\
 & \ls\epsilon_{q+1}\lambda_{q+1}^{\jmath+\kappa-4\alpha}\delta_{q+2}+\ell^{-\jmath-\kappa}\left(\frac{\lambda_{q}}{\lambda_{q+1}}\right)^{\frac{9}{4}}\delta_{q}\nonumber \\
 & \ls\epsilon_{q+1}\lambda_{q+1}^{\jmath+\kappa-4\alpha}\delta_{q+2}\label{eq:Fl_est_no_loss}
\end{align}
where we used (\ref{eq:stress_size_ind1}) to pass to the third line,
and (\ref{eq:double_skip_iteration}) to pass to the last line.

For $0\leq \jmath\leq 10$, by (\ref{eq:Fq_sup}), (\ref{eq:grad_uq_sup}),
(\ref{eq:ell_smaller}) we have
\begin{align}
\left\Vert \left(F_{\ell}-F_{q}\right)(t)\right\Vert _{\jmath} & \ls\ell_{q}^{2}\left\Vert F_{q}(t)\right\Vert _{\jmath+2}+\ell_{q}^{2}\left\Vert u_{q}(t)\right\Vert _{1+\jmath}\left\Vert u_{q}(t)\right\Vert _{1}\nonumber \\
 & \ls\ell_{q}^{2}\epsilon_{q}\lambda_{q}^{\jmath+2-3\alpha}\delta_{q+1}+\ell_{q}^{2}\lambda_{q}^{2+\jmath}\delta_{q}\nonumber \\
 & \ls\epsilon_{q}\lambda_{q}^{\jmath-3\alpha}\delta_{q+1}\label{eq:Fl-fq_est-1}
\end{align}

For $0\leq \jmath\leq 10$, $\kappa\geq0$ and $t\in\mathcal{G}_{q}+B(0,\epsilon_{q-1}\tau_{q-1})$,
by (\ref{eq:Fq_goodbounds}), (\ref{eq:uq_vq_goodbounds}), (\ref{eq:double_skip_iteration}),
(\ref{eq:stress_size_ind1}) we have
\begin{align}
\left\Vert \left(F_{\ell}-F_{q}\right)(t)\right\Vert _{\jmath+\kappa} & \ls_{\kappa}\ell_{q}^{2}\left\Vert F_{q}(t)\right\Vert _{\jmath+\kappa+2}+\ell_{q}^{2-\kappa}\left\Vert u_{q}(t)\right\Vert _{1+\jmath}\left\Vert u_{q}(t)\right\Vert _{1}\nonumber \\
 & \ls_{\kappa}\ell_{q}^{2}\ell_{q-1}^{-\kappa}\lambda_{q-1}^{2+\jmath-3\alpha}\delta_{q}+\ell_{q}^{2-\kappa}\lambda_{q-1}^{2+\jmath}\delta_{q-1}\nonumber \\
 & \ls\ell_{q-1}^{-\kappa}\lambda_{q}^{\jmath-3\alpha}\epsilon_{q}\frac{\lambda_{q}}{\lambda_{q+1}}\delta_{q+1}+\ell_{q}^{-\kappa}\lambda_{q}^{\jmath}\frac{\lambda_{q}}{\lambda_{q+1}}\delta_{q+1}\nonumber \\
 & \ls\epsilon_{q}\ell_{q}^{-\kappa}\lambda_{q}^{\jmath-3\alpha}\delta_{q+1}\label{eq:Fl-fq_est-2}
\end{align}

For $0\leq \jmath\leq 12$ (no loss of derivatives) and $t\in\mathcal{G}_{q}+B(0,\epsilon_{q-1}\tau_{q-1})$:
\begin{align}
\left\Vert \left(F_{\ell}-F_{q}\right)(t)\right\Vert _{\jmath} & \ls\ell_{q}^{2}\left\Vert F_{q}(t)\right\Vert _{\jmath+2}+\ell_{q}^{2-\jmath}\left\Vert u_{q}(t)\right\Vert _{1}^{2}\nonumber \\
 & \ls\ell_{q}^{2}\ell_{q-1}^{-\jmath}\lambda_{q-1}^{2-3\alpha}\delta_{q}+\ell_{q}^{2-\jmath}\lambda_{q-1}^{2}\delta_{q-1}\nonumber \\
 & \ls\ell_{q-1}^{-\jmath}\frac{1}{\lambda_{q}\lambda_{q+1}}\lambda_{q}^{2-3\alpha}\delta_{q+1}+\lambda_{q+1}^{\jmath}\frac{1}{\lambda_{q}\lambda_{q+1}}\lambda_{q}^{2}\delta_{q+1}\nonumber \\
 & \ls\epsilon_{q+1}\lambda_{q+1}^{\jmath-4\alpha}\delta_{q+2}\label{eq:Fl-fq_est}
\end{align}
where we used (\ref{eq:uq_vq_goodbounds}), (\ref{eq:Fq_goodbounds}),
(\ref{eq:ell_smaller}) and (\ref{eq:stress_size_ind3}).

Next, by (\ref{eq:grad_uq_sup}), (\ref{eq:uvF_sup}) we have 
\begin{align}
\left\Vert F_{\ell}\right\Vert _{0} & \leq\left\Vert F_{q}\right\Vert _{0}+C(d)\ell^{2}\left\Vert u_{q}(t)\right\Vert _{1}^{2}\leq1-\delta_{q}^{\frac{1}{2}}+C(d)\ell_{q}^{2}\lambda_{q}^{2}\delta_{q}\nonumber \\
& \leq1-\delta_{q}^{\frac{1}{2}}+C(d)\epsilon_{q}\delta_{q+1}\\
& \ll 1-\frac{3}{2}\delta_{q+1}^{\frac{1}{2}}\label{eq:fl_0}
\end{align}
because of (\ref{eq:ell_smaller}).

For $\kappa\geq0$ and $0\leq \jmath\leq 10$, we have 
\begin{align}
\|R_{\ell}\|_{\jmath+\kappa} & \lesssim_{\kappa}\ell^{-\kappa}\left\Vert R_{q}\right\Vert _{\jmath}+\ell^{2-\kappa}\left(\left\Vert u_{q}\right\Vert _{1+\jmath}\left\Vert u_{q}\right\Vert _{1}+\left\Vert v_{q}\right\Vert _{1+\jmath}\left\Vert v_{q}\right\Vert _{1}\right)\nonumber \\
 & \lesssim\ell^{-\kappa}\epsilon_{q}\lambda_{q}^{\jmath-3\alpha}\delta_{q+1}+\ell^{2-k}\lambda_{q}^{2+\jmath}\delta_{q}+\ell^{2-k}\lambda_{q-1}^{2+\jmath}\delta_{q-1}\nonumber \\
 & \lesssim\ell^{-\kappa}\epsilon_{q}\lambda_{q}^{\jmath-3\alpha}\delta_{q+1}\label{eq:Rl_estimate}
\end{align}
because of (\ref{eq:grad_uq_sup}), (\ref{eq:grad_vq_sup}), (\ref{eq:Fq_sup}),
(\ref{eq:ell_smaller}) and (\ref{eq:double_skip_iteration}).


\subsection{Temporal cutoffs}
\label{subsec:temporal_cutoff}

We now resume the main construction.  With $\tau_{q}$ defined in (\ref{eq:tau_q}), we let $t_{j}{\,\coloneqq\,}j\tau_{q}$, and let
$\mathcal{J}$ be the set of all indices $j$ such that 
\[
\left[t_{j}-2\epsilon_{q}\tau_{q},t_{j}+3\epsilon_{q}\tau_{q}\right]\subset \mathcal B_{q}.
\]
The set $\mathcal J$ contains the ``bad'' indices whose corresponding time intervals will be part of $\mathcal{B}_{q+1}$.  Clearly we have $\#(\mathcal{J})\sim\tau_{q}^{-1}\prod_{p=1}^{q-1}\epsilon_{p}$. 

We next define $$\mathcal{J}^{*}\,\coloneqq\,\{j\in\mathcal{J}:j+1\in\mathcal{J}\}.$$  This set consists of the indices where we will apply the local existence estimates for forced Euler equations.

Recall that we have the decompositions $$\mathcal B_q=\bigcup_iI_i^{b,q},\quad \mathcal G_q=\bigcup_iI_i^{g,q}$$ as unions of disjoint intervals.  We consider a partition of unity $\{\chi_{j}^{b}\}_{j}\cup\{\chi_{i}^{g}\}_{i}$ of $[0,T]$ such that 
\begin{itemize}
\item for any $j\in\mathcal{J}^{*}$, $\mathbf{1}_{[t_{j},t_{j+1}+\epsilon_{q}\tau_{q}]}\geq\chi_{j}^{b}\geq\mathbf{1}_{[t_{j}+\epsilon_{q}\tau_{q},t_{j+1}]}$, 
\item $\chi_{i}^{g}$ is supported in $I_{i}^{g,q}+B\left(0,\tau_{q}+6\epsilon_{q}\tau_{q}\right)$, and
\item we have the bounds 
\begin{equation}
\|\dd_{t}^{N}\chi_{i}^{g}\|_{0}+\|\dd_{t}^{N}\chi_{j}^{b}\|_{0}\lesssim_{N}(\epsilon_{q}\tau_{q})^{-N}\quad\text{for all }N\geq1.\label{eq:cutoff_time_deri}
\end{equation}
\end{itemize}

We remark that (\ref{eq:Rq_zerowhere}) and the fact that $\epsilon_{q}\tau_{q}\ll\tau_{q}\ll\epsilon_{q-1}\tau_{q-1}$ imply that $R_{q}=0$ on $\supp\chi_{i}^{g}$.  We refer the reader to Figure 3.1 in \cite{bulutEpochsRegularityWild2022} for an illustration of this time cutoff scheme.

We now set $\chi^{g}{\,\coloneqq\,}\sum_{i}\chi_{i}^{g}$ and $\chi^{b}{\,\coloneqq\,}\sum_{j\in\mathcal{J^{*}}}\chi_{j}^{b}$, and
define
\begin{align}
u_{q+1} & =\overline{u}_{q}=\chi^{g}u_{q}+\chi^{b}u_{\ell}\nonumber \\
F_{q+1} & =\overline{F}_{q}=\dd_{t}\chi^{g}\mathcal{R}(u_{q}-u_{\ell})-\chi^{g}(1-\chi^{g})(u_{q}-u_{\ell})\otimes(u_{q}-u_{\ell})\nonumber \\
 & \phantom{=\overline{F}_{q}=}+\chi^{g}F_{q}+\left(1-\chi^{g}\right)F_{\ell}\label{eq:Fbarq_formula}
\end{align}
where $\mathcal{R}$ is as defined in (\ref{eq:antidiv}).  We then have 
\[
\dd_{t}\overline{u}_{q}+\divop\overline{u}_{q}\otimes\overline{u}_{q}+\grad\overline{\pi}_{q}=\divop\overline{F}_{q}
\]
for a suitable pressure $\overline{\pi}_{q}$.  Moreover, for $1\leq \jmath\leq 12$ and $\kappa\geq0$, by (\ref{eq:grad_uq_sup}) we have 
\begin{equation}
\left\Vert u_{\ell}\right\Vert _{\jmath+\kappa}\leq C(\kappa)\ell_{q}^{-k}\left\Vert u_{q}\right\Vert _{\jmath}\leq C(\kappa)\ell_{q}^{-k}M\lambda_{q}^{\jmath}\delta_{q}^{\frac{1}{2}}\leq B_{\kappa,2}\ell_{q}^{-k}\lambda_{q}^{\jmath}\delta_{q}^{\frac{1}{2}}\label{eq:Bk_2nd_choice}
\end{equation}
for a sufficiently large choice of $B_{\kappa,2}$.  In particular,
$B_{\kappa,2}$ does not depend on any $A_{\kappa}+B_{\kappa}$ and
will help determine $B_{\kappa}$ in (\ref{eq:uq_vq_goodbounds})-(\ref{eq:Fq_goodbounds})
later.

We now state the relevant local existence estimates.
\begin{lem}
\label{lem:local_existence}Suppose we are given $\alpha\in\left(0,1\right)$, a
smooth divergence-free datum $v_{0}$, and a smooth force $f$.  Then for any $\tau\lesssim_{\alpha}\min\left(\left\Vert v_{0}\right\Vert _{C_{x}^{1+\alpha}}^{-1},\left\Vert f\right\Vert _{C_{t}^{0}C_{x}^{1+\alpha}}^{-1/2}\right)$,
there exists a unique smooth solution $v$ to (\ref{eq:Euler_start}) on
$[0,\tau]\times\mathbb{T}^{d}$ such that $v\left(0,\cdot\right)=v_{0}$
and 
\[
\left\Vert v\right\Vert _{N+\alpha}\lesssim_{N,\alpha}\left\Vert v_{0}\right\Vert _{N+\alpha}+\tau\left\Vert f\right\Vert _{C_{t}^{0}C_{x}^{N+\alpha}}\quad\text{for all }N\geq1.
\]
\end{lem}

The proof of Lemma \ref{lem:local_existence} follows standard techniques; we include the details in Appendix \ref{app:Local-existence-of}.  Invoking \Lemref{local_existence}, for any $j\in\mathcal{J}^{*}$, let $v_{j}$ to be the solution of the forced Euler equations
\begin{align*}
\dd_{t}v_{j}+\divop v_{j}\otimes v_{j}+\grad p_{j} & =\divop F_{\ell}\\
\divop v_{j} & =0\\
v_{j}(t_{j}) & =v_{\ell}(t_{j})
\end{align*}
on $[t_{j},t_{j+2}]\times\T$.  Indeed, the definition of $v_j$ on this time scale is permissible because
\begin{align}
 & \tau_{q}\|v_{\ell}(t_{j})\|_{1+\alpha}+\tau_{q}\left\Vert \divop F_{\ell}\right\Vert _{C_{t}^{0}C_{x}^{1+\alpha}}^{1/2}\label{eq:tau_l}\\
 & \lesssim\tau_{q}\lambda_{q-1}^{1+\alpha}\delta_{q-1}^{\frac{1}{2}}+\tau_{q}\left(\lambda_{q}^{2-2\alpha}\delta_{q+1}\right)^{1/2}\nonumber \\
 & \lesssim\tau_{q}\lambda_{q}^{1+\alpha}\delta_{q+1}^{\frac{1}{2}}+\tau_{q}\lambda_{q}^{1-\alpha}\delta_{q+1}^{\frac{1}{2}}\nonumber \\
 & \ll1\nonumber 
\end{align}
where we have used (\ref{eq:tau_q}), (\ref{eq:Fl_estimate}), and (\ref{eq:double_skip_iteration}).\footnote{We remark that the local well-posedness is allowed to continue much longer than allowable in the proof of Onsager's conjecture.  This will be essential in Section~\ref{errorestimates}; see the discussion in Subsection~\ref{subsec:strategy}.}

For $1\leq\jmath\leq 8$ and $\kappa\geq0$ we then have 
\begin{align}
\|v_{j}\|_{\jmath+\kappa+\alpha} & \leq C(\kappa)\|v_{\ell}(t_{j})\|_{\jmath+\kappa+\alpha}+C(\kappa)\tau\left\Vert \divop F_{\ell}\right\Vert _{C_{t}^{0}C_{x}^{\jmath+\kappa+\alpha}}\nonumber \\
 & \leq C(\kappa)M\ell^{-\kappa}\lambda_{q-1}^{\jmath+\alpha}\delta_{q-1}^{1/2}+C(\kappa)\tau B_{\kappa,1}\tau\ell^{-\kappa}\lambda_{q}^{\jmath+1-2\alpha}\delta_{q+1}\nonumber \\
 & \leq B_{\kappa,3}\ell^{-\kappa}\lambda_{q}^{\jmath-\alpha}\delta_{q+1}^{\frac{1}{2}}\label{eq:vj_bound}
\end{align}
 for a sufficiently large choice of $B_{\kappa,3}$, because of (\ref{eq:tau_q}),
(\ref{eq:Fl_estimate}), (\ref{eq:grad_vq_sup}) and (\ref{eq:double_skip_iteration}).
In particular, $B_{\kappa,3}$ does not depend on any $A_{\kappa}+B_{\kappa}$.
Note that this implies 
\begin{equation}
\|v_{j}\|_{\jmath+\kappa+\alpha}+\|v_{\ell}\|_{\jmath+\kappa+\alpha}\lesssim_{\kappa}\ell^{-\kappa}\lambda_{q}^{\jmath+\alpha}\delta_{q+1}^{1/2}\label{eq:vj_and_vl_bound}.
\end{equation}

We now define 
\begin{align}
\overline{v}_{q} & {\,\coloneqq\,}\sum_{i}\chi_{i}^gv_q+\sum_{j\in\mathcal{J^{*}}}{}\chi_{j}^{b}v_{j}\label{eq:vbar_q},
\end{align}
and let $\mathcal{B}_{q+1}$ be the union of the intervals $\left[t_{j}-2\epsilon_{q}\tau_{q},t_{j}+3\epsilon_{q}\tau_{q}\right]$ lying in $\mathcal{B}_{q}$.

It immediately follows that (\ref{eq:Gq_Bq}) is satisfied.  Moreover, by choosing $$B_{\kappa}=\max\left\{ B_{\kappa,1},B_{\kappa,2},B_{\kappa,3}\right\},$$
(from (\ref{eq:Fl_estimate}), (\ref{eq:Bk_2nd_choice}), (\ref{eq:vj_bound})),
we have, for $1\leq\jmath\leq8$ and $\kappa\geq0$,
\begin{align}
\left\Vert \overline{u}_{q}\right\Vert _{\jmath+\kappa} & \leq\chi^{g}\left\Vert u_{q}\right\Vert _{\jmath+\kappa}+\chi^{b}\left\Vert u_{\ell}\right\Vert _{\jmath+\kappa}\nonumber \\
 & \leq\left(A_{\kappa}+B_{\kappa}\right)\ell_{q-1}^{-\kappa}\lambda_{q-1}^{\jmath}\delta_{q-1}^{\frac{1}{2}}\chi^{g}(t)+B_{\kappa,2}\ell_{q}^{-k}\lambda_{q}^{\jmath}\delta_{q}^{\frac{1}{2}}\chi^{b}(t)\nonumber \\
 & \leq\left(A_{\kappa}+B_{\kappa}\right)\ell_{q}^{-\kappa}\lambda_{q}^{\jmath}\delta_{q}^{\frac{1}{2}}\label{eq:glue2-u}
\end{align}
because of (\ref{eq:uq_vq_goodbounds}), (\ref{eq:Bk_2nd_choice}),
and (\ref{eq:double_skip_iteration}).  Similarly, for $1\leq\jmath\leq8$ and $\kappa\geq0$,
\begin{align}
\left\Vert \overline{v}_{q}\right\Vert _{\jmath+\kappa} & \leq\chi^{g}\left\Vert v_{q}\right\Vert _{\jmath+\kappa}+\chi^{b}\|v_{j}\|_{\jmath+\kappa+\alpha}\nonumber \\
 & \leq\left(A_{\kappa}+B_{\kappa}\right)\ell_{q-1}^{-\kappa}\lambda_{q-1}^{\jmath}\delta_{q-1}^{\frac{1}{2}}\chi^{g}(t)+B_{\kappa,3}\ell_{q}^{-k}\lambda_{q}^{\jmath-\alpha}\delta_{q+1}^{\frac{1}{2}}\chi^{b}(t)\nonumber \\
 & \leq\left(A_{\kappa}+B_{\kappa}\right)\ell_{q}^{-\kappa}\lambda_{q}^{\jmath-\alpha}\delta_{q+1}^{\frac{1}{2}}\label{eq:glue2}
\end{align}
because of (\ref{eq:uq_vq_goodbounds}), (\ref{eq:Bk_2nd_choice}), and (\ref{eq:vj_bound}).  It follows that (\ref{eq:Fl_estimate}), (\ref{eq:glue2-u}), and (\ref{eq:glue2}) imply (\ref{eq:uq_vq_goodbounds})-(\ref{eq:Fq_goodbounds}) with $q$ changed to $q+1$.

In \Subsecref{gluing-v} below we will construct a favorable smooth tensor field $\overline{R}_{q}$ such that 
\begin{equation}
\dd_{t}\overline{v}_{q}+\divop\left(\overline{v}_{q}\otimes\overline{v}_{q}\right)+\grad\overline{p}_{q}=\divop\left(\chi^{g}F_{q}+\chi^{b}F_{\ell}\right)+\divop\overline{R}_{q}\label{eq:Rq_bar_def}
\end{equation}
for some pressure $\overline{p}_{q}$. 

Note that we will use $\chi^{g}F_{q}+\chi^{b}F_{\ell}$ (instead of $\overline{F}_{q}$) for constructing the stress in the remainder of this section and in \Secref{Convex-integration-and}. 

\subsection{Gluing estimates for $u$}  We now establish our gluing estimates for $u$.

\begin{prop}
\label{prop:glue_est_for_u}For any $0\leq \jmath\leq 12$, and $\kappa\geq0$,
we have 
\begin{align}
\left\Vert \overline{u}_{q}-u_{\ell}\right\Vert _{\jmath+\kappa+\alpha} & \leq\mathbf{1}_{\supp\chi^{g}}\left\Vert u_{q}-u_{\ell}\right\Vert _{\jmath+\kappa+\alpha}\nonumber \\
&\lesssim_{\kappa}\epsilon_{q}\epsilon_{q+1}\lambda_{q+1}^{\jmath+\kappa-4\alpha}\delta_{q+2}^{\frac{1}{2}},\label{eq:glue1-u}
\end{align}
and
\begin{align}
\left\Vert \mathcal{B}\overline{u}_{q}-\mathcal{B}u_{\ell}\right\Vert _{\jmath+\kappa+\alpha} & \leq\mathbf{1}_{\supp\chi^{g}}\left\Vert \mathcal{B}u_{q}-\mathcal{B}u_{\ell}\right\Vert _{\jmath+\kappa+\alpha}\nonumber\\
& \lesssim_{\kappa}\left(\epsilon_{q}\epsilon_{q+1}\right)^{\frac{1}{2}}\lambda_{q+1}^{\jmath+\kappa-1-4\alpha}\delta_{q+2}^{\frac{1}{2}},\label{eq:glue1_bu}
\end{align}
where $\mathcal{B}$ is the Biot-Savart operator defined in (\ref{def-B}).  Moreover, for $0\leq \jmath\leq 12$, we have
\begin{align}
&\left\Vert \dd_{t}\chi^{g}\mathcal{R}(u_{q}-u_{\ell})-\chi^{g}(1-\chi^{g})(u_{q}-u_{\ell})\otimes(u_{q}-u_{\ell})\right\Vert _{\jmath+\alpha}\nonumber \\
&\hspace{0.2in}\leq\frac{M}{2} \epsilon_{q+1}\lambda_{q+1}^{\jmath-4\alpha}\delta_{q+2}\label{eq:commutator_term_in_Fq}.
\end{align}
\end{prop}

\begin{proof}[Proof of \Propref{glue_est_for_u}]
For $0\leq \jmath\leq 12$, $\kappa\geq0$ and $t\in\mathcal{G}_{q}+B\left(0,\epsilon_{q-1}\tau_{q-1}\right)$, observe that
\begin{align}
&\|u_{q}-u_{\ell}\|_{\jmath+\kappa+\alpha}+\|v_{q}-v_{\ell}\|_{\jmath+\kappa+\alpha}\nonumber \\
&\hspace{0.2in} \lesssim_{\kappa}\ell_{q}^{2}\left(\|u_{q}\|_{\jmath+\kappa+2+\alpha}+\|v_{q}\|_{\jmath+\kappa+2+\alpha}\right)\nonumber \\
&\hspace{0.2in} \lesssim_{\kappa}\ell_{q-1}^{-\kappa-\jmath}\frac{1}{\lambda_{q}\lambda_{q+1}}\lambda_{q-1}^{2+\alpha}\delta_{q-1}^{\frac{1}{2}}\nonumber \\
&\hspace{0.2in} \lesssim\ell_{q-1}^{-\kappa-\jmath}\frac{\lambda_{q}^{1+\alpha}}{\lambda_{q+1}}\epsilon_{q}\delta_{q+1}^{\frac{1}{2}}\nonumber \\
&\hspace{0.2in} \ll\min\left\{ \epsilon_{q}\lambda_{q}^{\jmath+\kappa-5\alpha}\delta_{q+1}^{\frac{1}{2}},\epsilon_{q}^{2}\epsilon_{q+1}\lambda_{q+1}^{\jmath+\kappa-4\alpha}\delta_{q+2}^{\frac{1}{2}}\right\} \label{eq:uq-ul_goodtimes}
\end{align}
where we have used (\ref{eq:uq_vq_goodbounds}) and (\ref{eq:ell_smaller})
to pass to the third line, (\ref{eq:double_skip_iteration}) to pass
to the fourth line, and (\ref{eq:stress_size_ind2}) in the final
inequality.  This proves (\ref{eq:glue1-u}).  Similarly, by (\ref{eq:uq_vq_goodbounds}),
(\ref{eq:ell_smaller}), (\ref{eq:double_skip_iteration}), and (\ref{eq:stress_size_ind1})
we have
\begin{align}
&\|\mathcal{B}u_{q}-\mathcal{B}u_{\ell}\|_{\jmath+\kappa+\alpha}+\|\mathcal{B}v_{q}-\mathcal{B}v_{\ell}\|_{\jmath+\kappa+\alpha}\nonumber \\
&\hspace{0.2in} \lesssim_{\kappa}\ell_{q}^{2}\left(\|\mathcal{B}u_{q}\|_{\jmath+\kappa+2+\alpha}+\|\mathcal{B}v_{q}\|_{\jmath+\kappa+2+\alpha}\right)\nonumber \\
&\hspace{0.2in} \lesssim_{\kappa}\ell_{q}^{2}\left(\|u_{q}\|_{\jmath+\kappa+1+\alpha}+\|v_{q}\|_{\jmath+\kappa+1+\alpha}\right)\nonumber \\
&\hspace{0.2in} \lesssim_{\kappa}\ell_{q-1}^{-\kappa-\jmath}\lambda_{q}^{-\frac{1}{2}}\lambda_{q+1}^{-\frac{3}{2}}\lambda_{q-1}^{1+\alpha}\delta_{q-1}^{\frac{1}{2}}\nonumber \\
&\hspace{0.2in} \lesssim\ell_{q-1}^{-\kappa-\jmath}\lambda_{q+1}^{-\frac{3}{2}}\epsilon_{q}\lambda_{q}^{\frac{1}{2}+\alpha}\delta_{q+1}^{\frac{1}{2}}\nonumber \\
&\hspace{0.2in} \ll\epsilon_{q}^{\frac{3}{2}}\epsilon_{q+1}^{\frac{1}{2}}\lambda_{q+1}^{\jmath+\kappa-1-4\alpha}\delta_{q+2}^{\frac{1}{2}}\label{eq:buq-bul_goodtimes}
\end{align}
for $0\leq \jmath\leq 12$, $\kappa\geq0$ and $t\in\mathcal{G}_{q}+B\left(0,\epsilon_{q-1}\tau_{q-1}\right)$.  
We have thus proven (\ref{eq:glue1-u}) and (\ref{eq:glue1_bu}).

For $0\leq \jmath\leq 12$ and $\kappa\geq0$, because of (\ref{eq:uq-ul_goodtimes}),
(\ref{eq:buq-bul_goodtimes}), and (\ref{eq:stress_size_ind3}) we have
\begin{align*}
&\left\Vert \dd_{t}\chi^{g}\mathcal{R}(u_{q}-u_{\ell})-\chi^{g}(1-\chi^{g})(u_{q}-u_{\ell})\otimes(u_{q}-u_{\ell})\right\Vert _{\jmath+\alpha}\\
&\hspace{0.2in} \lesssim_{\kappa}\left(\epsilon_{q}\tau_{q}\right)^{-1}\left\Vert \mathcal{B}u_{q}-\mathcal{B}u_{\ell}\right\Vert _{\jmath+\alpha}+\left\Vert u_{q}-u_{\ell}\right\Vert _{\jmath+\alpha}\left\Vert u_{q}-u_{\ell}\right\Vert _{\alpha}\\
&\hspace{0.2in} \ls\left(\frac{\epsilon_{q+1}}{\epsilon_{q}}\right)^{\frac{1}{2}}\left(\lambda_{q}^{1+3\alpha}\delta_{q+1}^{\frac{1}{2}}\right)\lambda_{q+1}^{\jmath-1-4\alpha}\delta_{q+2}^{\frac{1}{2}}+\epsilon_{q}^{2}\epsilon_{q+1}^{2}\lambda_{q+1}^{\jmath-8\alpha}\delta_{q+2}\\
&\hspace{0.2in} \ls\epsilon_{q+1}\lambda_{q+1}^{\jmath-5\alpha}\delta_{q+2}
\end{align*}
Choosing $a$ sufficiently large, this completes the proof of (\ref{eq:commutator_term_in_Fq}).
\end{proof}

\begin{rem}
The estimate (\ref{eq:buq-bul_goodtimes}) is the point in the argument where the strictest bound on $\ell_{q}$ is required; in particular, it is this estimate which requires $\ell_q$ to be as small as defined in (\ref{eq:length_scale}). 
\end{rem}

We now record some straightforward consequences of \Propref{glue_est_for_u}.  Observe that for $1\leq \jmath\leq 12$, we have
\begin{equation}
\left\Vert F_{q+1}\right\Vert _{\jmath}=\left\Vert \overline{F}_{q}\right\Vert _{\jmath}\ls\epsilon_{q+1}\lambda_{q+1}^{\jmath-4\alpha}\delta_{q+2}\label{eq:F_q+1_bound_noloss}
\end{equation}
because of (\ref{eq:commutator_term_in_Fq}), (\ref{eq:Fl_est_no_loss}),
and (\ref{eq:Fq_sup}).  The force increment obeys, for $0\leq \jmath\leq 12$,
\begin{equation}
\left\Vert \overline{F}_{q}-F_{\ell}\right\Vert _{\jmath}\ls\chi^{g}\left\Vert F_{q}-F_{\ell}\right\Vert _{\jmath}+\epsilon_{q+1}\lambda_{q+1}^{\jmath-4\alpha}\delta_{q+2}\ls\epsilon_{q+1}\lambda_{q+1}^{\jmath-4\alpha}\delta_{q+2}\label{eq:Fbar-Fl_noloss}
\end{equation}
because of (\ref{eq:commutator_term_in_Fq}) and (\ref{eq:Fl-fq_est}).  Finally,
\begin{align}
\left\Vert F_{q+1}\right\Vert _{0}=\left\Vert \overline{F}_{q}\right\Vert _{0} & \leq C(d)\epsilon_{q+1}\lambda_{q+1}^{\jmath-4\alpha}\delta_{q+2}+\chi^{g}\left\Vert F_{q}\right\Vert _{0}+\left(1-\chi^{g}\right)\left\Vert F_{\ell}\right\Vert _{0}\nonumber \\
 & \ll\delta_{q+1}+1-\frac{3}{2}\delta_{q+1}^{\frac{1}{2}}\ll1-\delta_{q+1}^{\frac{1}{2}}\label{eq:Fq+1_0}
\end{align}
for large enough $a$, because of (\ref{eq:uvF_sup}), (\ref{eq:commutator_term_in_Fq}),
and (\ref{eq:fl_0}).

\subsection{Gluing estimates for $v$}
\label{subsec:gluing-v}

We next turn to the gluing estimates for $v$.  To simplify notation, we will use 
\begin{align*}
D_{t,\ell} & {\,\coloneqq\,}\dd_{t}+v_{\ell}\cdot\grad,\\
D_{t,q} & {\,\coloneqq\,}\dd_{t}+\overline{v}_{q}\cdot\grad.
\end{align*}
to denote the material derivatives along the respective coarse flows.

\begin{prop}[Gluing estimates for $v$]
\label{prop:glue_est} For any $0\leq \jmath\leq 7$, and $\kappa\geq0$, we
have 
\begin{align}
\left\Vert \overline{v}_{q}-v_{\ell}\right\Vert _{\jmath+\kappa+\alpha} & \lesssim_{\kappa}\epsilon_{q}\tau_{q}\delta_{q+1}\lambda_{q}^{1+\jmath-2\alpha}\ell^{-\kappa},\label{eq:glue1}\\
\|\overline{R}_{q}\|_{\jmath+\kappa+\alpha} & \lesssim_{\kappa}\ell^{-\kappa}\lambda_{q}^{\jmath-2\alpha}\delta_{q+1},\label{eq:glue3}
\end{align}
and, for $\jmath\leq 6$,
\begin{align}
\|D_{t,q}\overline{R}_{q}\|_{\jmath+\kappa+\alpha} & \lesssim_{\kappa}\left(\epsilon_{q}\tau_{q}\right)^{-1}\ell^{-\kappa}\lambda_{q}^{\jmath-2\alpha}\delta_{q+1}.\label{eq:glue4}
\end{align}
\end{prop}

We remark that in subsequent convex integration steps, the solution will only be perturbed for $t\in\bigcup_{j\in\mathcal J}\left[t_{j}-\epsilon_{q}\tau_{q},t_{j}+2\epsilon_{q}\tau_{q}\right]$.  Thus,
\begin{align*}
v_{q+1} & =\overline{v}_{q}\\
\overline{R}_{q} & =0
\end{align*}
for all other times. 

The proof of \Propref{glue_est} is broken into three steps, treating (I) the region near the good sets (away from the gluing intervals), (II) the bad sets, and (III) the good-bad interface.

\subsection*{Region (I): near the good sets}

Consider the temporal region $\{\chi^{g}=1\}$.  In this region, we have
\begin{align*}
\overline{v}_{q} & =v_{q}, & \overline{F}_{q} & =F_{q}, & \overline{R}_{q} & {\,\coloneqq\,}0
\end{align*}
Then (\ref{eq:uq-ul_goodtimes}) implies (\ref{eq:glue1}) and therefore the conclusion of \Propref{glue_est} holds in this region.

\subsection*{Region (II): in the bad sets\protect\label{subsec:Bad-bad-interface}}

Consider the temporal region $\left[t_{j},t_{j}+2\tau_{q}\right]$
where $j\in\mathcal{J}^{*}$ such that $j+1\in\mathcal{J}^{*}$.  Note
that $\supp(\chi_{j}^{b}\chi_{j+1}^{b})$ lies in $\left[t_{j+1},t_{j+1}+\epsilon_{q}\tau_{q}\right]$.  Furthermore, we have
\[
\dd_{t}\overline{v}_{q}+\divop\overline{v}_{q}\otimes\overline{v}_{q}+\grad\overline{p}_{q}=\divop F_{\ell}+\divop\overline{R}_{q},
\]
where 
\begin{equation}
\overline{R}_{q}=\dd_{t}\chi_{j}^{b}\mathcal{R}(v_{j}-v_{j+1})-\chi_{j}^{b}(1-\chi_{j}^{b})(v_{j}-v_{j+1})\otimes(v_{j}-v_{j+1})\label{eq:glued_stress}
\end{equation}
and $\mathcal{R}$ is the standard inverse divergence; see (\ref{eq:antidiv})--(\ref{eq:antidiv2}).  To estimate $v_j-v_{j+1}$, thanks to the identity $v_{j}-v_{j+1}=(v_{j}-v_{\ell})-(v_{j+1}-v_{\ell})$, it suffices to prove bounds on  $v_{j}-v_{\ell}$.

Let us recall the transport estimate as in \cite[Proposition 3.3]{buckmasterOnsagerConjectureAdmissible2017}, which asserts that for 
$\alpha\in\left(0,1\right)$, if $v$ is a smooth vector field,
and $t\left\Vert v\right\Vert _{1}\leq1$, then  one has
\begin{align}
\left\Vert f(t)\right\Vert _{\alpha}\leq e^{\alpha}\left(\left\Vert f(0)\right\Vert _{\alpha}+\int_{0}^{t}\mathrm{d}s\;\left\Vert \left(\partial_{t}+\nabla_{v}\right)f(s)\right\Vert _{\alpha}\right).\label{eq-transport}
\end{align}

\begin{prop}
\label{prop:vglue}For $0\leq \jmath\leq 7$, $\kappa\geq0$, and $t\in\left[t_{j},t_{j}+2\tau_{q}\right]$,
\begin{align}
\|v_{j}-v_{\ell}\|_{\jmath+\kappa+\alpha} & \lesssim_{\kappa}\tau_{q}\epsilon_{q}\ell^{-\kappa}\lambda_{q}^{\jmath+1-2\alpha}\delta_{q+1}\label{eq:vj-vl_est}\\
\|D_{t,\ell}\left(v_{j}-v_{\ell}\right)\|_{\jmath+\kappa+\alpha} & \lesssim_{\kappa}\epsilon_{q}\ell^{-\kappa}\lambda_{q}^{\jmath+1-2\alpha}\delta_{q+1}.\label{eq:vj-vl_transport_est}
\end{align}
\end{prop}

\begin{proof}
Note that
\begin{align}
\left(\partial_{t}+v_{\ell}\cdot\nabla\right)\left(v_{\ell}-v_{j}\right) & =-\left(v_{\ell}-v_{j}\right)\cdot\nabla v_{j}-\nabla\left(p_{\ell}-p_{j}\right)+\Div R_{\ell}\label{eq:vl_vj_subtract}
\end{align}
and
\begin{align}
\nabla\left(p_{\ell}-p_{j}\right) & =\mathcal{P}_{1}\left(-\left(v_{\ell}-v_{j}\right)\cdot\nabla v_{\ell}-\left(v_{\ell}-v_{j}\right)\cdot\nabla v_{j}+\Div R_{\ell}\right)\label{eq:grad_pl-pj},
\end{align}
where $\mathcal{P}_{1}$ is as defined in \Subsecref{geo}, and where we have implicitly used the identity (\ref{eq:ident_P1}).  Then, as usual, by the transport estimate (\ref{eq-transport}),
\begin{align*}
\left\Vert v_{\ell}-v_{\jmath}\right\Vert _{\alpha} & \lesssim\int_{t_{\jmath}}^{t}\mathrm{d}s\;\left\Vert \left(v_{\ell}-v_{j}\right)\cdot\nabla v_{j}(s)\right\Vert _{\alpha}+\left\Vert \nabla\left(p_{\ell}-p_{j}\right)(s)\right\Vert _{\alpha}+\left\Vert R_{\ell}(s)\right\Vert _{1+\alpha}\\
 & \lesssim\tau_{q}\left\Vert R_{\ell}\right\Vert _{C_{t}^{0}C_{x}^{1+\alpha}}+\int_{t_{j}}^{t}\mathrm{d}s\;\left\Vert \left(v_{\ell}-v_{j}\right)(s)\right\Vert _{\alpha}\left(\left\Vert v_{j}\right\Vert _{1+\alpha}+\left\Vert v_{\ell}\right\Vert _{1+\alpha}\right)\\
 & \lesssim\tau_{q}\epsilon_{q}\lambda_{q}^{1-2\alpha}\delta_{q+1}+\int_{t_{j}}^{t}\mathrm{d}s\;\left\Vert \left(v_{\ell}-v_{j}\right)(s)\right\Vert _{\alpha}\lambda_{q}^{1+\alpha}\delta_{q+1}^{\frac{1}{2}}
\end{align*}
where we used (\ref{eq:Rl_estimate}), (\ref{eq:vj_bound}), (\ref{eq:grad_v_l})
and (\ref{eq:double_skip_iteration}).  By Gr\"onwall and (\ref{eq:tau_q}), (\ref{eq:vj_bound}), (\ref{eq:grad_v_l})
and (\ref{eq:Rl_estimate}) we conclude 
\begin{align*}
\left\Vert v_{\ell}-v_{j}\right\Vert _{\alpha} & \lesssim\tau_{q}\epsilon_{q}\lambda_{q}^{1-2\alpha}\delta_{q+1}\exp\left(\tau_{q}\lambda_{q}^{1+\alpha}\delta_{q+1}^{\frac{1}{2}}\right)\lesssim\tau_{q}\epsilon_{q}\lambda_{q}^{1-2\alpha}\delta_{q+1},\\
\left\Vert \nabla\left(p_{\ell}-p_{j}\right)\right\Vert _{\alpha} & \lesssim\tau_{q}\epsilon_{q}\lambda_{q}^{1-2\alpha}\delta_{q+1}\left(\lambda_{q}^{1+\alpha}\delta_{q+1}^{\frac{1}{2}}\right)+\epsilon_{q}\lambda_{q}^{1-2\alpha}\delta_{q+1}\lesssim\epsilon_{q}\lambda_{q}^{1-2\alpha}\delta_{q+1},
\end{align*}
and
\begin{align*}
\left\Vert D_{t,\ell}\left(v_{\ell}-v_{j}\right)\right\Vert _{\alpha} & \lesssim\tau_{q}\epsilon_{q}\lambda_{q}^{1-2\alpha}\delta_{q+1}\left(\lambda_{q}^{1+\alpha}\delta_{q+1}^{\frac{1}{2}}\right)+\epsilon_{q}\lambda_{q}^{1-2\alpha}\delta_{q+1}\lesssim\epsilon_{q}\lambda_{q}^{1-2\alpha}\delta_{q+1}.
\end{align*}

Let $\theta$ be a multi-index with $\left|\theta\right|=\jmath+\kappa$
for any $1\leq\jmath\leq7,\kappa\geq0$; then by (\ref{eq:Rl_estimate}),
(\ref{eq:tau_q}), and (\ref{eq:grad_pl-pj}) we have
\begin{align*}
\left\Vert \partial^{\theta}\nabla\left(p_{\ell}-p_{j}\right)\right\Vert _{\alpha} & \lesssim\left\Vert R_{\ell}\right\Vert _{1+\jmath+\kappa+\alpha}+\left\Vert v_{\ell}-v_{j}\right\Vert _{\alpha}\left(\left\Vert v_{j}\right\Vert _{1+\jmath+\kappa+\alpha}+\left\Vert v_{\ell}\right\Vert _{1+\jmath+\kappa+\alpha}\right)\\
 & \phantom{\lesssim\left\Vert R_{\ell}\right\Vert _{1+\jmath+\kappa+\alpha}}+\left\Vert v_{\ell}-v_{j}\right\Vert _{\jmath+\kappa+\alpha}\left(\left\Vert v_{j}\right\Vert _{1+\alpha}+\left\Vert v_{\ell}\right\Vert _{1+\alpha}\right)\\
 & \lesssim\ell^{-\kappa}\epsilon_{q}\lambda_{q}^{\jmath+1-2\alpha}\delta_{q+1}+\tau_{q}\epsilon_{q}\lambda_{q}^{1-2\alpha}\delta_{q+1}\left(\ell^{-\kappa}\lambda_{q}^{\jmath+1+\alpha}\delta_{q+1}^{\frac{1}{2}}\right)\\
 & \phantom{\lesssim\ell^{-\kappa}\epsilon_{q}\lambda_{q}^{\jmath+1-2\alpha}\delta_{q+1}}+\left\Vert v_{\ell}-v_{j}\right\Vert _{\jmath+\kappa+\alpha}\left(\lambda_{q}^{1+\alpha}\delta_{q+1}^{\frac{1}{2}}\right)\\
 & \lesssim\ell^{-\kappa}\epsilon_{q}\lambda_{q}^{\jmath+1-2\alpha}\delta_{q+1}+\left\Vert v_{\ell}-v_{j}\right\Vert _{\jmath+\kappa+\alpha}\left(\lambda_{q}^{1+\alpha}\delta_{q+1}^{\frac{1}{2}}\right).
\end{align*}
Therefore, by (\ref{eq:vl_vj_subtract}) we have 
\[
\left\Vert \partial^{\theta}D_{t,\ell}\left(v_{\ell}-v_{j}\right)\right\Vert _{\alpha}\lesssim\ell^{-\kappa}\epsilon_{q}\lambda_{q}^{\jmath+1-2\alpha}\delta_{q+1}+\left\Vert v_{\ell}-v_{j}\right\Vert _{\jmath+\kappa+\alpha}\left(\lambda_{q}^{1+\alpha}\delta_{q+1}^{\frac{1}{2}}\right).
\]
Invoking the transport estimate once again, we have 
\begin{align*}
\left\Vert \partial^{\theta}\left(v_{\ell}-v_{j}\right)\right\Vert _{\alpha} & \lesssim\int_{t_{0}}^{t}\mathrm{d}s\;\left\Vert D_{t,\ell}\partial^{\theta}\left(v_{\ell}-v_{j}\right)(s)\right\Vert _{\alpha}\\
 & \lesssim\int_{t_{0}}^{t}\mathrm{d}s\;\left\Vert \left[D_{t,\ell},\partial^{\theta}\right]\left(v_{\ell}-v_{j}\right)(s)\right\Vert _{\alpha}+\ell^{-\kappa}\epsilon_{q}\lambda_{q}^{\jmath+1-2\alpha}\delta_{q+1}\\
 & \phantom{\lessapprox}+\left\Vert \left(v_{\ell}-v_{j}\right)(s)\right\Vert _{\jmath+\kappa+\alpha}\left(\lambda_{q}^{1+\alpha}\delta_{q+1}^{\frac{1}{2}}\right).
\end{align*}
By interpolation, (\ref{eq:tau_q}) and (\ref{eq:vj_bound}), we have
\begin{align*}
\left\Vert \left[D_{t,\ell},\partial^{\theta}\right]\left(v_{\ell}-v_{j}\right)(s)\right\Vert _{\alpha} & \lesssim\left\Vert v_{\ell}\right\Vert _{1+\alpha}\left\Vert \left(v_{\ell}-v_{j}\right)(s)\right\Vert _{\jmath+\kappa+\alpha}\\
&\hspace{0.2in}+\left\Vert v_{\ell}\right\Vert _{\jmath+\kappa+\alpha}\left\Vert v_{\ell}-v_{j}\right\Vert _{1+\alpha}\\
 & \lesssim\left\Vert v_{\ell}\right\Vert _{1+\alpha}\left\Vert \left(v_{\ell}-v_{j}\right)(s)\right\Vert _{\jmath+\kappa+\alpha}\\
&\hspace{0.2in}+\left\Vert v_{\ell}\right\Vert _{1+\jmath+\kappa+\alpha}\left\Vert v_{\ell}-v_{j}\right\Vert _{\alpha}\\
 & \lesssim\left(\lambda_{q}^{1+\alpha}\delta_{q+1}^{\frac{1}{2}}\right)\left\Vert \left(v_{\ell}-v_{j}\right)(s)\right\Vert _{\jmath+\kappa+\alpha}\\
 &\quad+\left(\ell^{-\kappa}\lambda_{q}^{1+\jmath+\alpha}\delta_{q+1}^{\frac{1}{2}}\right)\tau_{q}\epsilon_{q}\lambda_{q}^{1-2\alpha}\delta_{q+1}\\
 & \lesssim\ell^{-\kappa}\epsilon_{q}\lambda_{q}^{\jmath+1-2\alpha}\delta_{q+1}\\
&\hspace{0.2in}+\left\Vert \left(v_{\ell}-v_{j}\right)(s)\right\Vert _{\jmath+\kappa+\alpha}\left(\lambda_{q}^{1+\alpha}\delta_{q+1}^{\frac{1}{2}}\right).
\end{align*}

Combining these estimates and using Gr\"onwall's inequality, we have 
\begin{align*}
\left\Vert v_{\ell}-v_{j}\right\Vert _{\jmath+\kappa+\alpha} & \lesssim\tau_{q}\epsilon_{q}\ell^{-\kappa}\lambda_{q}^{\jmath+1-2\alpha}\delta_{q+1}\exp\left(\tau\lambda_{q}^{1+\alpha}\delta_{q+1}^{\frac{1}{2}}\right)\ls\tau_{q}\epsilon_{q}\ell^{-\kappa}\lambda_{q}^{\jmath+1-2\alpha}\delta_{q+1}
\end{align*}
and
\begin{align*}
\left\Vert \partial^{\theta}D_{t,\ell}\left(v_{\ell}-v_{j}\right)\right\Vert _{\alpha} & \ls\epsilon_{q}\ell^{-\kappa}\lambda_{q}^{\jmath+1-2\alpha}\delta_{q+1}.
\end{align*}
This completes the proof of \Propref{vglue}.
\end{proof}

The above proposition proves (\ref{eq:glue1}) for any $t\in\left[t_{j},t_{j}+2\tau_{q}\right]$.  To proceed, we next define the potentials $z_{j}{\,\coloneqq\,}\mathcal{B}v_{j}$ and $z_{\ell}{\,\coloneqq\,}\mathcal{B}v_{\ell}$, where $\mathcal{B}$ is the Biot-Savart operator defined in (\ref{def-B}), and establish analogous estimates.
\begin{prop}
\label{prop:zglue}For $0\leq \jmath\leq 7$, $\kappa\geq0$ and $t\in\left[t_{j},t_{j}+2\tau_{q}\right]$:
\begin{align}
\|z_{j}-z_{\ell}\|_{\jmath+\kappa+\alpha} & \lesssim_{\kappa}\tau_{q}\epsilon_{q}\ell^{-\kappa}\lambda_{q}^{\jmath-2\alpha}\delta_{q+1}\label{eq:z_est}\\
\|D_{t,\ell}\left(z_{j}-z_{\ell}\right)\|_{\jmath+\kappa+\alpha} & \lesssim_{\kappa}\epsilon_{q}\ell^{-\kappa}\lambda_{q}^{\jmath-2\alpha}\delta_{q+1}\label{eq:z_transport_est}
\end{align}
\end{prop}

\begin{proof}
Let $\widetilde{z}{\,\coloneqq\,}z_{\ell}-z_{j}$.  From (\ref{eq:grad_pl-pj}) we deduce
\begin{align*}
\partial_{t}\widetilde{z}+\nabla_{v_{\ell}}\widetilde{z} & =\left(-\Delta\right)^{-1}d\circ\Div\left(\nabla v_{j,\ell}*\widetilde{z}+R_{\ell}\right)+\left(-\Delta\right)^{-1}\delta\circ\Div\left(\nabla v_{\ell}*\widetilde{z}\right)
\end{align*}
where $v_{j,\ell}$ could be $v_{j}$ or $v_{\ell}$ and $*$ represents a tensor contraction whose details are not important; see \cite[Proposition~11]{bulutEpochsRegularityWild2022} for the calculation.

Let $0\leq \jmath\leq 7$ and $\kappa\geq0$ be given.  Since $\left(-\Delta\right)^{-1}d\circ\Div$
and $\left(-\Delta\right)^{-1}\delta\circ\Div$ are Calderón-Zygmund operators, we estimate 
\begin{align}
\begin{array}{c}
\left\Vert D_{t,\ell}\widetilde{z}(s)\right\Vert _{\jmath+\kappa+\alpha}\end{array} & \lesssim\left\Vert \nabla v_{j,\ell}\right\Vert _{\jmath+\kappa+\alpha}\left\Vert \widetilde{z}(s)\right\Vert _{\alpha}+\left\Vert \nabla v_{j,\ell}\right\Vert _{\alpha}\left\Vert \widetilde{z}(s)\right\Vert _{\jmath+\kappa+\alpha}+\left\Vert R_{\ell}\right\Vert _{\jmath+\kappa+\alpha}\nonumber \\
 & \lesssim\ell^{-\kappa}\lambda_{q}^{\jmath+1+\alpha}\delta_{q+1}^{\frac{1}{2}}\left\Vert \widetilde{z}(s)\right\Vert _{\alpha}+\lambda_{q}^{1+\alpha}\delta_{q+1}^{\frac{1}{2}}\left\Vert \widetilde{z}(s)\right\Vert _{\jmath+\kappa+\alpha}\nonumber\\
 &\quad+\ell^{-\kappa}\epsilon_{q}\lambda_{q}^{\jmath-2\alpha}\delta_{q+1}.\label{eq:dtl_inter}
\end{align}
Once again by the transport estimate we have
\begin{align}
\left\Vert \widetilde{z}\left(t\right)\right\Vert _{\alpha} & \lesssim\int_{t_{j}}^{t}\mathrm{d}s\;\left\Vert D_{t,\ell}\widetilde{z}\left(s\right)\right\Vert _{\alpha}\;\label{eq:modified_transpot}\\
 & \lesssim\int_{t_{j}}^{t}\mathrm{d}s\;\lambda_{q}^{1+\alpha}\delta_{q}^{\frac{1}{2}}\left\Vert \widetilde{z}(s)\right\Vert _{\alpha}+\epsilon_{q}\lambda_{q}^{-2\alpha}\delta_{q+1}.\nonumber 
\end{align}
By Gr\"onwall's inequality, we obtain 
\[
\left\Vert \widetilde{z}\left(t\right)\right\Vert _{\alpha}\ls\epsilon_{q}\tau_{q}\lambda_{q}^{-2\alpha}\delta_{q+1}.
\]
For $1\leq \jmath\leq 7$ and $\kappa\geq0$, as $\nabla\mathcal{B}$ is Calderón-Zygmund,
we have 
\begin{align*}
\left\Vert z_{j}-z_{\ell}\right\Vert _{\jmath+\kappa+\alpha} & \lesssim\left\Vert \nabla\left(z_{j}-z_{\ell}\right)\right\Vert _{\jmath-1+k+\alpha}=\left\Vert \nabla\mathcal{B}\left(v_{j}-v_{\ell}\right)\right\Vert _{\jmath-1+k+\alpha}\\
 & \lesssim\left\Vert v_{j}-v_{\ell}\right\Vert _{\jmath-1+k+\alpha}\ls\tau_{q}\epsilon_{q}\ell^{-\kappa}\lambda_{q}^{\jmath-2\alpha}\delta_{q+1}.
\end{align*}
We conclude
\[
\left\Vert D_{t,\ell}\widetilde{z}(s)\right\Vert _{\jmath+\kappa+\alpha}\ls\ell^{-\kappa}\epsilon_{q}\lambda_{q}^{\jmath-2\alpha}\delta_{q+1}
\]
for $1\leq \jmath\leq 7$ and $\kappa\geq0$.
\end{proof}

It remains to establish the estimates for $\overline{R}_q$.  This is the content of the next proposition.

\begin{prop}
\label{prop:glued_stress_est}With $\overline{R}_{q}$ defined in (\ref{eq:glued_stress}), we have the bounds 
\begin{align}
\|\overline{R}_{q}\|_{\jmath+\kappa+\alpha} & \lesssim_{\kappa}\ell^{-\kappa}\lambda_{q}^{\jmath-2\alpha}\delta_{q+1}\label{eq:gluedR}\\
\|(\dd_{t}+\overline{v}_{q}\cdot\grad)\overline{R}_{q}\|_{\jmath+\kappa+\alpha} & \lesssim_{\kappa}\left(\epsilon_{q}\tau_{q}\right)^{-1}\ell^{-\kappa}\lambda_{q}^{\jmath-2\alpha}\delta_{q+1}\label{eq:gluedR_transport}
\end{align}
for $0\leq \jmath\leq 7$, $\kappa\geq0$ and $t\in\left[t_{j+1},t_{j+1}+\epsilon_{q}\tau_{q}\right]$.
\end{prop}

In the proof of \Propref{glued_stress_est}, we will find it useful to recall a commutator inequality due to \cite[Lemma 1]{constantin2015lagrangian} and \cite[Proposition D.1]{buckmasterOnsagerConjectureAdmissible2017}.
\begin{lem}
\label{lem:singular_comm}Suppose $\alpha\in\left(0,1\right)$, $N\in\mathbb{N}_{0}$,
$\mathcal{T}$ is a Calderón-Zygmund operator, and $b\in C^{N+1,\alpha}$
is a divergence-free vector field on $\mathbb{T}^{d}$.  Then for any
$f\in C^{N+\alpha}\left(\mathbb{T}^{d}\right)$, we have
\[
\left\Vert \left[\mathcal{T},b\cdot\nabla\right]f\right\Vert _{N+\alpha}\lesssim_{N,\alpha,\mathcal{T}}\left\Vert b\right\Vert _{1+\alpha}\left\Vert f\right\Vert _{N+\alpha}+\left\Vert b\right\Vert _{N+1+\alpha}\left\Vert f\right\Vert _{\alpha}.
\]
\end{lem}

With this lemma in hand, we establish the proposition.

\begin{proof}[Proof of \Propref{glued_stress_est}]
Note that, combining (\ref{eq:cutoff_time_deri}), (\ref{eq:tau_q}), (\ref{eq:z_est}), and (\ref{eq:vj-vl_est}), as well as the boundedness of the Calderón-Zygmund operator $\mathcal{R}\delta$, we obtain
\begin{align}
\|\dd_{t}\chi_{j}^{b}\mathcal{R}(v_{j}-v_{j+1})\|_{\jmath+\kappa+\alpha}\lesssim_{\kappa}\epsilon_{q}^{-1}\tau_{q}^{-1}\|z_{j}-z_{j+1}\|_{\jmath+\kappa+\alpha}\lesssim_{\kappa}\ell^{-\kappa}\lambda_{q}^{\jmath-2\alpha}\delta_{q+1}\label{eq:1stRv}
\end{align}
and
\begin{align}
\|\chi_{j}^{b}(1-\chi_{j}^{b})(v_{j}-v_{j+1})\otimes(v_{j}-v_{j+1})\|_{\jmath+\kappa+\alpha}&\lesssim_{\kappa}(\epsilon_{q}\tau_{q}\lambda_{q}\delta_{q+1})^{2}\lambda_{q}^{\jmath}\ell^{-\kappa}\nonumber \\
&\ls\epsilon_{q}^{2}\ell^{-\kappa}\lambda_{q}^{\jmath-5\alpha}\delta_{q+1}\label{eq:2ndRV}
\end{align}
for $0\leq \jmath\leq 7$, $\kappa\geq0$ and $t\in\left[t_{j+1},t_{j+1}+\epsilon_{q}\tau_{q}\right]$.

Now, observe that (\ref{eq:glued_stress}), (\ref{eq:1stRv}), and (\ref{eq:2ndRV})
imply
\[
\|\overline{R}_{q}\|_{\jmath+\kappa+\alpha}\lesssim_{\kappa}\ell^{-\kappa}\lambda_{q}^{\jmath-2\alpha}\delta_{q+1}
\]
for $0\leq \jmath\leq 7$ and $\kappa\geq0$.  Then (\ref{eq:gluedR}) follows from (\ref{eq:time-length}).  For the material derivative, we have 
\begin{align*}
\left\Vert \left(\partial_{t}+\nabla_{\overline{v}_{q}}\right)\overline{R}_{q}\right\Vert _{\jmath+\kappa+\alpha} & \leq\left\Vert D_{t,\ell}\overline{R}_{q}\right\Vert _{\jmath+\kappa+\alpha}+\left\Vert \nabla_{\overline{v}_{q}-v_{\ell}}\overline{R}_{q}\right\Vert _{\jmath+\kappa+\alpha}.
\end{align*}
One can compute
\begin{align*}
D_{t,\ell}\overline{R}_{q} & =\left(\partial_{t}^{2}\chi_{j}^{b}\right)\mathcal{R}\delta\left(z_{j}-z_{j+1}\right)\\
 & +\left(\partial_{t}\chi_{j}^{b}\right)\mathcal{R}\delta D_{t,\ell}\left(z_{j}-z_{j+1}\right)+\left(\partial_{t}\chi_{j}^{b}\right)\left[v_{\ell}\cdot\nabla,\mathcal{R}\delta\right]\left(z_{j}-z_{j+1}\right)\\
 & +\partial_{t}\left(\left(\chi_{j}^{b}\right)^{2}-\chi_{j}^{b}\right)\left(v_{j}-v_{j+1}\right)\otimes\left(v_{j}-v_{j+1}\right)\\
 & +\left(\left(\chi_{j}^{b}\right)^{2}-\chi_{j}^{b}\right)\Big(D_{t,\ell}\left(v_{j}-v_{j+1}\right)\otimes\left(v_{j}-v_{j+1}\right)\\
&\hspace{2.2in}+\left(v_{j}-v_{j+1}\right)\otimes D_{t,\ell}\left(v_{j}-v_{j+1}\right)\Big).
\end{align*}

The term involving $\left[v_{\ell}\cdot\nabla,\mathcal{R}\delta\right]$
can be handled by \Lemref{singular_comm}.  Then by (\ref{eq:cutoff_time_deri}),
(\ref{eq:gluedR}), (\ref{eq:tau_q}), \Propref[s]{vglue} and \ref{prop:zglue},
we conclude 
\begin{align*}
\left\Vert \left[v_{\ell}\cdot\nabla,\mathcal{R}\delta\right]\left(z_{j}-z_{j+1}\right)\right\Vert _{\jmath+\kappa+\alpha} & \ls_{\kappa}\left\Vert v_{\ell}\right\Vert _{1+\alpha}\left\Vert z_{j}-z_{j+1}\right\Vert _{\jmath+\kappa+\alpha}\\
&\quad\quad+\left\Vert v_{\ell}\right\Vert _{\jmath+\kappa+1+\alpha}\left\Vert z_{j}-z_{j+1}\right\Vert _{\alpha}\\
 & \ls\lambda_{q}^{1+\alpha}\delta_{q+1}^{\frac{1}{2}}\left(\epsilon_{q}\tau_{q}\ell^{-\kappa}\lambda_{q}^{\jmath-2\alpha}\delta_{q+1}\right)\ls\epsilon_{q}\ell^{-\kappa}\lambda_{q}^{\jmath-4\alpha}\delta_{q+1},
\end{align*}
and
\begin{align*}
\|(\dd_{t}+\overline{v}_{q}\cdot\grad)\overline{R}_{q}\|_{\jmath+\kappa+\alpha} & \lesssim_{\kappa}\left(\epsilon_{q}\tau_{q}\right)^{-1}\ell^{-\kappa}\lambda_{q}^{\jmath-2\alpha}\delta_{q+1}+\tau_{q}^{-1}\ell^{-\kappa}\lambda_{q}^{\jmath-2\alpha}\delta_{q+1}\\
 & \phantom{\lesssim_{\kappa}}+\epsilon_{q}\tau_{q}\ell^{-\kappa}\lambda_{q}^{2+\jmath-4\alpha}\delta_{q+1}^{2}+\epsilon_{q}^{2}\tau_{q}\ell^{-\kappa}\lambda_{q}^{2+\jmath-4\alpha}\delta_{q+1}^{2}\\
 & \ls\left(\epsilon_{q}\tau_{q}\right)^{-1}\ell^{-\kappa}\lambda_{q}^{\jmath-2\alpha}\delta_{q+1}.
\end{align*}
\end{proof}

\subsection*{Region (III): good-bad interface}

It remains to establish \Propref{glue_est} in the third region, covering the interface between good and bad intervals, for which we consider pairs of indices $i$ and $j$ satisfying $\chi_{i}^{g}\chi_{j}^{b}\not\equiv0$.
Recall that $\supp(\chi_{i}^{g}\chi_{j}^{b})$
is an interval of length $\sim\epsilon_{q}\tau_{q}$ in which $R_{q}\equiv0$.  Within such an interval, the glued solution obeys
\begin{align*}
\overline{v}_{q}-v_{\ell} & =\chi_{i}^{g}\left(v_{q}-v_{\ell}\right)+\chi_{j}^{b}\left(v_{j}-v_{\ell}\right)\\
\dd_{t}\overline{v}_{q}+\divop\overline{v}_{q}\otimes\overline{v}_{q}+\grad\overline{p}_{q} & =\divop\left(\chi^{g}F_{q}+\chi^{b}F_{\ell}\right)+\divop\overline{R}_{q}\\
\overline{R}_{q} & {\,\coloneqq\,}\dd_{t}\chi_{i}^{g}\mathcal{R}(v_{q}-v_{j})-\chi_{i}^{g}(1-\chi_{i}^{g})(v_{q}-v_{j})\otimes(v_{q}-v_{j}).
\end{align*}
Estimating $v_{j}-v_{\ell}$ is precisely the same as in Region (II), treated above.  We thus focus on $v_{q}-v_{\ell}$.
\begin{prop}
\label{prop:vglue-1}For $\jmath\ge0$, $\kappa\geq0$ and $t\in\mathcal{G}_{q}+B\left(0,\epsilon_{q-1}\tau_{q-1}\right)$, we have
\begin{align}
\|v_{q}-v_{\ell}\|_{\jmath+\kappa+\alpha} & \lesssim\tau_{q}\epsilon_{q}\ell_{q}^{-\kappa}\lambda_{q}^{\jmath+1-2\alpha}\delta_{q+1}\label{eq:vestimate-1-1}
\end{align}
for $\jmath\leq 12$, and
\begin{align}
\|D_{t,\ell}\left(v_{q}-v_{\ell}\right)\|_{\jmath+\kappa+\alpha} & \lesssim\epsilon_{q}\ell_{q}^{-\kappa}\lambda_{q}^{\jmath+1-2\alpha}\delta_{q+1}\label{eq:vtransportestimate-1-1}
\end{align}
for $\jmath\leq 6$.
\end{prop}

\begin{proof}
From (\ref{eq:uq-ul_goodtimes}) we have 
\[
\|v_{q}-v_{\ell}\|_{\jmath+\kappa+\alpha}\lesssim\epsilon_{q}\ell_{q-1}^{-\kappa-\jmath}\frac{\lambda_{q}^{1+\alpha}}{\lambda_{q+1}}\delta_{q+1}^{\frac{1}{2}}\ll\epsilon_{q}\ell_{q}^{-\kappa}\lambda_{q}^{\jmath-5\alpha}\delta_{q+1}^{\frac{1}{2}}
\]
So (\ref{eq:vestimate-1-1}) is proven.  Then as $R_{q}=0$ on this temporal region, we have (in analogy to (\ref{eq:vl_vj_subtract})--(\ref{eq:grad_pl-pj})), 
\begin{align}
\left(\partial_{t}+v_{\ell}\cdot\nabla\right)\left(v_{\ell}-v_{q}\right) & =-\left(v_{\ell}-v_{q}\right)\cdot\nabla v_{q}-\nabla\left(p_{\ell}-p_{q}\right)\nonumber \\
&\hspace{0.2in}+\Div\left(F_{\ell}-F_{q}+R_{\ell}\right)\label{eq:vl_vj_subtract-1}
\end{align}
and
\begin{align}
\nabla\left(p_{\ell}-p_{q}\right) & =\mathcal{P}_{1}\left(-\left(v_{\ell}-v_{q}\right)\cdot\nabla v_{\ell}-\left(v_{\ell}-v_{q}\right)\cdot\nabla v_{q}+\Div R_{\ell}\right).
\end{align}

For $0\leq \jmath\leq 6$, $\kappa\geq0$, and $t\in\mathcal{G}_{q}+B\left(0,\epsilon_{q-1}\tau_{q-1}\right)$, it then follows by (\ref{eq:Fl-fq_est-2}), (\ref{eq:vestimate-1-1}), and (\ref{eq:Rl_estimate}), that
\begin{align*}
\|D_{t,\ell}\left(v_{q}-v_{\ell}\right)\|_{\jmath+\kappa+\alpha} & \ls\|v_{q}-v_{\ell}\|_{\jmath+\kappa+\alpha}\left\Vert \nabla v_{q}\right\Vert _{\alpha}+\|v_{q}-v_{\ell}\|_{\alpha}\left\Vert \nabla v_{q}\right\Vert _{\jmath+\kappa+\alpha}\\
 & \phantom{\ls}+\left\Vert F_{\ell}-F_{q}\right\Vert _{1+\jmath+\kappa+\alpha}+\left\Vert R_{\ell}\right\Vert _{1+\jmath+\kappa+\alpha}\\
 & \ls\epsilon_{q}\ell_{q}^{-\kappa}\lambda_{q}^{\jmath+1-2\alpha}\delta_{q+1},
\end{align*}
as desired.
\end{proof}

The above proposition completes the proof of (\ref{eq:glue1}) in this region.  Next let us define the potentials $z_{q}{\,\coloneqq\,}\mathcal{B}v_{q},z_{\ell}{\,\coloneqq\,}\mathcal{B}v_{\ell}$ and $\widetilde{z}=z_{\ell}-z_{q}$.  Then (\ref{eq:buq-bul_goodtimes}) implies
\begin{align*}
\|z_{\ell}-z_{q}\|_{\jmath+\kappa+\alpha} & \ls\ell_{q-1}^{-\kappa-\jmath}\lambda_{q+1}^{-\frac{3}{2}}\epsilon_{q}\lambda_{q}^{\frac{1}{2}+\alpha}\delta_{q+1}^{\frac{1}{2}}\\
&\ll\epsilon_{q}\ell_{q}^{-\kappa}\lambda_{q}^{\jmath-1-5\alpha}\delta_{q+1}^{\frac{1}{2}}\\
&\sim\tau_{q}\epsilon_{q}\ell_{q}^{-\kappa}\lambda_{q}^{\jmath-2\alpha}\delta_{q+1}
\end{align*}
for $0\leq \jmath\leq 12$, $\kappa\geq0$, and $t\in\mathcal{G}_{q}+B\left(0,\epsilon_{q-1}\tau_{q-1}\right)$.

As in the proof of \Propref{zglue}, we observe that 
\[
\partial_{t}\widetilde{z}+\nabla_{v_{\ell}}\widetilde{z}=\left(-\Delta\right)^{-1}d\circ\Div\left(\nabla v_{j,\ell}*\widetilde{z}+F_{\ell}-F_{q}+R_{\ell}\right)+\left(-\Delta\right)^{-1}\delta\circ\Div\left(\nabla v_{\ell}*\widetilde{z}\right)
\]
so that
\begin{align}
\begin{array}{c}
\left\Vert D_{t,\ell}\widetilde{z}(s)\right\Vert _{\jmath+\kappa+\alpha}\end{array} & \lesssim\left\Vert \nabla v_{q,\ell}\right\Vert _{\jmath+\kappa+\alpha}\left\Vert \widetilde{z}(s)\right\Vert _{\alpha}+\left\Vert \nabla v_{q,\ell}\right\Vert _{\alpha}\left\Vert \widetilde{z}(s)\right\Vert _{\jmath+\kappa+\alpha}\nonumber \\
 & \phantom{\ls}+\left\Vert F_{\ell}-F_{q}\right\Vert _{\jmath+\kappa+\alpha}+\left\Vert R_{\ell}\right\Vert _{\jmath+\kappa+\alpha}\nonumber \\
 & \lesssim\ell^{-\kappa}\lambda_{q}^{\jmath+1+\alpha}\delta_{q+1}^{\frac{1}{2}}\left(\epsilon_{q}\lambda_{q}^{-1-5\alpha}\delta_{q+1}^{\frac{1}{2}}\right)+\epsilon_{q}\ell_{q}^{-\kappa}\lambda_{q}^{\jmath-2\alpha}\delta_{q+1}\nonumber\\ 
 & \ls\epsilon_{q}\ell_{q}^{-\kappa}\lambda_{q}^{\jmath-2\alpha}\delta_{q+1}\nonumber 
\end{align}
for $0\leq \jmath\leq 6$, $\kappa\geq0$, and $t\in\mathcal{G}_{q}+B\left(0,\epsilon_{q-1}\tau_{q-1}\right)$.

Then, as with (\ref{eq:1stRv}) and (\ref{eq:2ndRV}), we have 
\begin{align}
\|\dd_{t}\chi_{i}^{g}\mathcal{R}(v_{q}-v_{j+1})\|_{\jmath+\kappa+\alpha} & \lesssim_{\kappa}\ell^{-\kappa}\lambda_{q}^{\jmath-2\alpha}\delta_{q+1}\nonumber\\ 
\|\chi_{i}^{g}(1-\chi_{i}^{g})(v_{q}-v_{j+1})\otimes(v_{q}-v_{j+1})\|_{\jmath+\kappa+\alpha} & \lesssim_{\kappa}\epsilon_{q}^{2}\ell^{-\kappa}\lambda_{q}^{\jmath-5\alpha}\delta_{q+1}\nonumber 
\end{align}
 for any $0\leq \jmath\leq 7$, $\kappa\geq0$ and $t\in\supp(\chi_{i}^{g}\chi_{j}^{b})$.

From here, (\ref{eq:glue1}), (\ref{eq:glue3}), and (\ref{eq:glue4}) are immediate.  This establishes an anlogue of \Propref{glued_stress_est} for $t\in \supp(\chi_i^g\chi_j^b)$, which in turn completes the proof of \Propref{glue_est}.

\section{\protect\label{sec:Convex-integration-and}Perturbation estimates: the convex integration construction}

In this section, we establish the perturbation estimates which will be used to
complete the proof of \Propref{iterative}.  In particular, having localized the
Reynolds stress in time with gluing, the next step is high-frequency
perturbation of $\overline v_q$ in order to cancel $\overline R_q$.  The result
of this convex integration is as follows.

\begin{prop}
\label{prop:convex_int} There is a smooth pair $\left(v_{q+1},\widetilde{R}_{q+1}\right)$
such that 
\[
\dd_{t}v_{q+1}+\divop\left(v_{q+1}\otimes v_{q+1}\right)+\grad p_{q+1}=\divop\left(\chi^{g}F_{q}+\chi^{b}F_{\ell}\right)+\divop\widetilde{R}_{q}
\]
for some pressure $p_{q+1}.$ Moreover, $v_{q+1}=\overline{v}_{q},\;\widetilde{R}_{q+1}=0$
outside the temporal regions $\left[t_{j}-\epsilon_{q}\tau_{q},t_{j}+2\epsilon_{q}\tau_{q}\right]$
($j\in\mathcal{J}$), and we have the estimates
\begin{align}
\left\Vert v_{q+1}-\overline{v}_{q}\right\Vert _{\jmath} & \leq\frac{M}{2}\lambda_{q+1}^{\jmath}\delta_{q+1}^{\frac{1}{2}}\label{eq:perturb_1}\\
\left\Vert \widetilde{R}_{q+1}\right\Vert _{\jmath} & \leq\frac{M}{2}\epsilon_{q+1}\lambda_{q+1}^{\jmath-3\alpha}\delta_{q+2}\label{eq:finalR}
\end{align}
for $0\leq \jmath\leq 12$, where $M>1$ is a geometric constant depending on $d$ but not on $a,\beta,b,\sigma,\alpha$, and $q$.
\end{prop}

The proof of this proposition includes a delicate iteration to avoid loss of derivatives issues at the threshold between the ``good'' and ``bad'' epochs. 


The geometric construction for our convex integration procedure is largely the same as in \cite{bulutEpochsRegularityWild2022}, where more details 
and explanations can be found.  We recall the definition of the Mikado flows from \cite[Lemma 5.1]{buckmasterOnsagerConjectureAdmissible2017},
which is valid for any dimension $d\geq3$ (see also \cite[Section 4.1]{cheskidovSharpNonuniquenessNavierStokes2020}).
Indeed, for any compact subset $\mathcal{N}\subset\subset\mathcal{S}_{+}^{d\times d}$
there is a smooth vector field $W:\mathcal{N}\times\mathbb{T}^{d}\to\mathbb{R}^{d}$
such that 
\begin{align}
\Div_{\xi}W(R,\xi)\otimes W(R,\xi)  =0,\label{eq:divf1}
\end{align}
\begin{align*}
\Div_{\xi}W(R,\xi)=0,
\end{align*}
\begin{align*}
\dashint_{\mathbb{T}^{d}}W\left(R,\xi\right)\;\mathrm{d}\xi=0,
\end{align*}
and
\begin{align*}
\dashint_{\mathbb{T}^{d}}W(R,\xi)\otimes W(R,\xi)\;\mathrm{d}\xi=R.
\end{align*}
Unless otherwise noted, we set $\mathcal{N}=\overline{B_{1/2}(\textrm{Id)}}$.  One can define the smooth coefficient functions $a_k:\mathcal N\to\mathbb C$, $C_k:\mathcal N\to\mathbb C^{d\times d}$ such that we have the Fourier series
\begin{align*}
W\left(R,\xi\right) & =\sum_{k\in\mathbb{Z}^{d}\backslash\{0\}}a_{k}\left(R\right)e^{
i2\pi\left\langle k,\xi\right\rangle }\\
W(R,\xi)\otimes W(R,\xi) & =R+\sum_{k\in\mathbb{Z}^{d}\backslash\{0\}}C_{k}\left(R\right)e^{i2\pi\left\langle k,\xi\right\rangle }
\end{align*}
and such that
\eqn{
\|\grad_R^Na_k\|_0+\|\grad_R^Nb_k\|_0\lesssim_{N,M}\langle k\rangle^{-M}.
}

In the standard manner, we introduce local-in-time Lagrangian
coordinates.  We define the backwards transport flow $\Phi_{i}$ as the solution
to 
\begin{align}
\left(\partial_{t}+\overline{v}_{q}\cdot\nabla\right)\Phi_{i} & =0\\
\Phi_{i}\left(t_{i},\cdot\right) & =\mathrm{Id}_{\mathbb{T}^{d}}\label{phi}
\end{align}
as well as the forward characteristic flow $X_{i}$ as the the
flow generated by $\overline{v}_{q}$: 
\begin{align*}
\partial_{\underline{t}}X_{i}\left(\underline{t},\underline{x}\right) & =\overline{v}_{q}\left(\underline{t},X_{i}\left(\underline{t},\underline{x}\right)\right)\\
X_{i}\left(t_{i},\cdot\right) & =\mathrm{Id}_{\mathbb{T}^{d}}.
\end{align*}
By defining their spacetime versions
\begin{align*}
\mathbf{\Phi}_{i}\left(t,x\right) & {\,\coloneqq\,}\left(t,\Phi_{i}\left(t,x\right)\right)\\
\mathbf{X}_{i}\left(\underline{t},\underline{x}\right) & {\,\coloneqq\,}\left(\underline{t},X_{i}\left(\underline{t},\underline{x}\right)\right)
\end{align*}
we can conclude $\mathbf{X}_{i}=\left(\mathbf{\Phi}_{i}\right)^{-1}$,
and that $\mathbf{X}_{i}$ maps from the Lagrangian spacetime $\left(\underline{t},\underline{x}\right)$
to the Eulerian spacetime $\left(t,x\right)$.

As in \cite[Proposition 3.1]{buckmasterOnsagerConjectureAdmissible2017},
for any $1\leq \jmath\leq 6$, $\kappa\geq0$, and $\left|t-t_{i}\right|\lesssim\tau_{q}$, we have
\begin{align}
\left\Vert \nabla\Phi_{i}\left(t\right)-\mathrm{Id}\right\Vert _{0} & \lesssim\left|t-t_{i}\right|\left\Vert \nabla\overline{v}_{q}\right\Vert _{0}\lesssim\tau_{q}\lambda_{q}\delta_{q+1}^{\frac{1}{2}}\ls\lambda_{q}^{-3\alpha}\ll1\label{eq:phi_id}\\
\left\Vert \nabla\Phi_{i}\left(t\right)\right\Vert _{\jmath+\kappa} & \lesssim\left|t-t_{i}\right|\left\Vert \nabla\overline{v}_{q}\right\Vert _{\jmath+\kappa}\lesssim\tau_{q}\ell^{-\kappa}\lambda_{q}^{\jmath+1}\delta_{q+1}^{\frac{1}{2}}\ls\ell^{-\kappa}\lambda_{q}^{\jmath-3\alpha}\quad\label{eq:phi_id_2}
\end{align}
where we have used (\ref{eq:glue2}).

We now define
\[
\underline{R_{i}}{\,\coloneqq\,}\mathbf{X}_{i}^{*}\left(\textrm{Id}-\frac{\overline{R}_{q}}{\delta_{q+1}}\right)
\]
where we treat $\overline{R}_{q}$ as a $(2,0)$-tensor (more explicitly, we remark that we have the identity 
\begin{align}
\underline{R_{i}}\circ\mathbf{\Phi}_{i}=\nabla\Phi_{i}\left(\textrm{Id}-\frac{\overline{R}_{q}}{\delta_{q+1}}\right)\nabla\Phi_{i}^{T}\label{eq:R_phiii}
\end{align}
and we note that, for $\left|t-t_{i}\right|\lesssim\tau_{q}$, we have $\underline{R_{i}}\in B_{1/2}(\textrm{Id)}$, since $\grad\Phi_{i}$ is close to $\textrm{Id}$ and $\left\Vert \frac{\overline{R}_{q}}{\delta_{q+1}}\right\Vert _{0}\lesssim\lambda_{q}^{-2\alpha}$ by (\ref{eq:glue3})).

For each $i$ let $\rho_{i}$ be a smooth cutoff such that
\eqn{
\mathbf{1}_{\left[t_{i},t_{i}+\epsilon_{q}\tau_{q}\right]}\leq\rho_{i}\leq\mathbf{1}_{\left[t_{i}-\epsilon_{q}\tau_{q},t_{i}+2\epsilon_{q}\tau_{q}\right]}
}
obeying the bounds
\[
\left\Vert \partial_{t}^{N}\rho_{i}\right\Vert _{0}\lesssim\left(\epsilon_{q}\tau_{q}\right)^{-N}\;\forall N\in\mathbb{N}_{0}.
\]

We now define the perturbation 
\begin{align*}
w^{(o)} & {\,\coloneqq\,}\sum_{i}\delta_{q+1}^{1/2}\rho_{i}(t)\grad\Phi_{i}^{-1}W(\underline{R_{i}}\circ\mathbf{\Phi}_{i},\lambda_{q+1}\Phi_{i}).
\end{align*}
For $t\in\left[t_{i}-\epsilon_{q}\tau_{q},t_{i}+2\epsilon_{q}\tau_{q}\right]$,
in local-in-time Lagrangian coordinates with $$\underline{w^{(o)}}{\,\coloneqq\,}\mathbf{X}_{i}^{*}w^{(o)},$$
we have 
\begin{align*}
\underline{w^{(o)}} & =\delta_{q+1}^{1/2}\rho_{i}(\underline{t})W(\underline{R_{i}},\lambda_{q+1}\underline{x})\\
 & =\sum_{k\neq0}\underbrace{\delta_{q+1}^{1/2}\rho_{i}(\underline{t})a_{k}(\underline{R_{i}})}_{{\,\coloneqq\,}\underline{b_{i,k}}}e^{i2\pi\left\langle \lambda_{q+1}k,\underline{x}\right\rangle }=\sum_{k\neq0}\underline{b_{i,k}}e^{i2\pi\left\langle \lambda_{q+1}k,\underline{x}\right\rangle }
\end{align*}
and therefore, by defining $b_{i,k}{\,\coloneqq\,}\mathbf{\Phi}_{i}^{*}\underline{b_{i,k}}$
(extended by zero outside $\supp\rho_{i}$), we have
\[
w^{\left(o\right)}=\sum_{i}\sum_{k\neq0}b_{i,k}e^{i2\pi\left\langle \lambda_{q+1}k,\Phi_{i}\right\rangle }.
\]

To address the fact that $\div w^{(o)}\neq0$, we introduce an incompressibility corrector.  In particular, for
for $t\in\left[t_{i}-\epsilon_{q}\tau_{q},t_{i}+2\epsilon_{q}\tau_{q}\right]$, in local-in-time Lagrangian coordinates, we define 
\[
\underline{w^{\left(c\right)}}{\,\coloneqq\,}\sum_{k\neq0}\underbrace{\delta_{q+1}^{1/2}\rho_{i}(\underline{t})\Div_{\underline{x}}\left(\frac{k\wedge a_{k}\left(\underline{R_{i}}\right)}{i2\pi\lambda_{q+1}\left|k\right|^{2}}\right)}_{{\,\coloneqq\,}\underline{c_{i,k}}}e^{i2\pi\left\langle \lambda_{q+1}k,\underline{x}\right\rangle }=\sum_{k\neq0}\underline{c_{i,k}}e^{i2\pi\left\langle \lambda_{q+1}k,\underline{x}\right\rangle }.
\]
In Eulerian coordinates, we define $c_{i,k}{\,\coloneqq\,}\mathbf{\Phi}_{i}^{*}\underline{c_{i,k}}$
(again extended by zero outside $\supp\rho_{i}$), as well as
\begin{align*}
w^{(c)} & {\,\coloneqq\,}\sum_{i}\sum_{k\neq0}c_{i,k}e^{i2\pi\left\langle \lambda_{q+1}k,\Phi_{i}\right\rangle }
\end{align*}
to obtain $\underline{w^{(c)}}=\mathbf{X}_{i}^{*}w^{(c)}$ for $t\in\left[t_{i}-\epsilon_{q}\tau_{q},t_{i}+2\epsilon_{q}\tau_{q}\right]$.  Equipped with this corrector, we can now finally define 
\[
v_{q+1}{\,\coloneqq\,}\overline{v}_{q}+w_{q+1}.
\]

Note that the full perturbation
\[
w_{q+1}{\,\coloneqq\,}w^{(o)}+w^{(c)}
\]
is divergence-free.  Moreover,
\begin{align*}
\partial_{t}v_{q+1}+\Div\left(v_{q+1}\otimes v_{q+1}\right)&= \left(\partial_{t}\overline{v}_{q}+\Div\left(\overline{v}_{q}\otimes\overline{v}_{q}\right)\right)+\Div\left(w_{q+1}\otimes w_{q+1}\right)\\
 & \quad+\partial_{t}w_{q+1}+\Div\left(\overline{v}_{q}\otimes w_{q+1}\right)+\Div\left(w_{q+1}\otimes\overline{v}_{q}\right)\\
&=-\nabla\overline{p}_{q}+\Div\left(\overline{R}_{q}+w_{q+1}\otimes w_{q+1}\right)\\
 & \quad+D_{t,q}w_{q+1}+w_{q+1}\cdot\nabla\overline{v}_{q}+\Div\left(\chi^{g}F_{q}+\chi^{b}F_{\ell}\right)
\end{align*}
so we can define the stress as 
\begin{align*}
\widetilde{R}_{q+1} & {\,\coloneqq\,}R_{\mathrm{osc}}+R_{\mathrm{trans}}+R_{\mathrm{Nash}}
\end{align*}
where
\begin{align*}
R_{\mathrm{osc}} & {\,\coloneqq\,}\mathcal{R}\Div\left(\overline{R}_{q}+w_{q+1}\otimes w_{q+1}\right)\\
R_{\mathrm{trans}} & {\,\coloneqq\,}\mathcal{R}D_{t,q}w_{q+1}\\
R_{\mathrm{Nash}} & {\,\coloneqq\,}\mathcal{R}\left(w_{q+1}\cdot\nabla\overline{v}_{q}\right).
\end{align*}

We are now ready to proceed with the proof of the estimates (\ref{eq:perturb_1}) and (\ref{eq:finalR}) asserted in the statement of \Propref{convex_int}.

\subsection{Proof of (\ref{eq:perturb_1}): Perturbation estimates}

We now establish our main collection of perturbation estimates.

\begin{prop}
\label{prop:pertub_est}Throughout this subsection, we assume $t\in\left[t_{i}-\epsilon_{q}\tau_{q},t_{i}+2\epsilon_{q}\tau_{q}\right]$
and $N\in\mathbb{N}_{0}$.  Then for $\kappa\geq0$ and $k\in\mathbb{Z}^{d}\backslash\{0\}$,
we have, 
\begin{align}
\left\Vert \nabla\Phi_{i}\right\Vert _{N}+\left\Vert \nabla\Phi_{i}^{-1}\right\Vert _{\jmath+\kappa} & \lesssim_{\kappa}\ell^{-\kappa}\lambda_{q}^{\jmath}\label{eq:per_est1}\\
\left\Vert \underline{R_{i}}\circ\mathbf{\Phi}_{i}\right\Vert _{\jmath+\kappa} & \lesssim_{\kappa}\ell^{-\kappa}\lambda_{q}^{\jmath}\label{eq:per_est2}\\
\left\Vert b_{i,k}\right\Vert _{\jmath+\kappa} & \lesssim_{\kappa}\delta_{q+1}^{\frac{1}{2}}\ell^{-\kappa}\lambda_{q}^{\jmath}\left|k\right|^{-2d}\label{eq:per_est3}\\
\left\Vert c_{i,k}\right\Vert _{\jmath+\kappa} & \lesssim_{\kappa}\frac{\lambda_{q}}{\lambda_{q+1}}\delta_{q+1}^{\frac{1}{2}}\ell^{-\kappa}\lambda_{q}^{\jmath}\left|k\right|^{-2d}\label{eq:per_est4}\\
\left\Vert D_{t,q}\left(\nabla\Phi_{i}\right)\right\Vert _{\jmath+\kappa} & \lesssim_{\kappa}\lambda_{q}\delta_{q+1}^{\frac{1}{2}}\ell^{-\kappa}\lambda_{q}^{\jmath}\label{eq:per_est5}\\
\left\Vert D_{t,q}\left(\underline{R_{i}}\circ\mathbf{\Phi}_{i}\right)\right\Vert _{\jmath+\kappa} & \lesssim_{\kappa}\left(\epsilon_{q}\tau_{q}\right)^{-1}\ell^{-\kappa}\lambda_{q}^{\jmath}\label{eq:per_est6}\\
\left\Vert D_{t,q}b_{i,k}\right\Vert _{\jmath+\kappa} & \lesssim_{\kappa}\left(\epsilon_{q}\tau_{q}\right)^{-1}\delta_{q+1}^{\frac{1}{2}}\ell^{-\kappa}\lambda_{q}^{\jmath}\left|k\right|^{-2d}\label{eq:per_est7}
\end{align}
for $0\leq \jmath\leq 6$, and
\begin{align}
\left\Vert D_{t,q}c_{i,k}\right\Vert _{\jmath+\kappa} & \lesssim_{\kappa}\left(\epsilon_{q}\tau_{q}\right)^{-1}\frac{\lambda_{q}}{\lambda_{q+1}}\delta_{q+1}^{\frac{1}{2}}\ell^{-\kappa}\lambda_{q}^{\jmath}\left|k\right|^{-2d}\label{eq:per_est8}
\end{align}
for $0\leq \jmath\leq 5$.
\end{prop}

\begin{proof}
The estimate (\ref{eq:per_est1}) is a standard consequence of (\ref{eq:phi_id}) and (\ref{eq:phi_id_2}).  Next, observe that (\ref{eq:per_est1}) and (\ref{eq:glue3}) imply
(\ref{eq:per_est2}):
\eqn{
\left\Vert \underline{R_{i}}\circ\mathbf{\Phi}_{i}\right\Vert _{\jmath+\kappa}&\lesssim_{\kappa}\left\Vert \nabla\Phi_{i}\right\Vert _{0}^{2}\left\Vert \textrm{Id}-\frac{\overline{R}_{q}}{\delta_{q+1}}\right\Vert _{\jmath+\kappa}+\left\Vert \nabla\Phi_{i}\right\Vert _{\jmath+\kappa}\left\Vert \nabla\Phi_{i}\right\Vert _{0}\left\Vert \textrm{Id}-\frac{\overline{R}_{q}}{\delta_{q+1}}\right\Vert _{0}\\
&\lesssim\ell^{-\kappa}\lambda_{q}^{\jmath}.
}
To prove (\ref{eq:per_est3}), we use (\ref{eq:per_est2}) along with the rapid decay of $a_k$ and its derivatives to find
\begin{align*}
\left\Vert b_{i,k}\right\Vert _{\jmath+\kappa} & =\left\Vert \delta_{q+1}^{\frac{1}{2}}\rho_{i}(t)\nabla\Phi_{i}^{-1}a_{k}\left(\underline{R_{i}}\circ\mathbf{\Phi}_{i}\right)\right\Vert _{\jmath+\kappa}\\
 & \lesssim\delta_{q+1}^{\frac{1}{2}}\left(\left\Vert \nabla\Phi_{i}^{-1}\right\Vert _{\jmath+\kappa}\left\Vert a_{k}\left(\underline{R_{i}}\circ\mathbf{\Phi}_{i}\right)\right\Vert _{0}+\left\Vert \nabla\Phi_{i}^{-1}\right\Vert _{0}\left\Vert a_{k}\left(\underline{R_{i}}\circ\mathbf{\Phi}_{i}\right)\right\Vert _{\jmath+\kappa}\right)\\
 & \lesssim\delta_{q+1}^{\frac{1}{2}}\ell^{-\kappa}\lambda_{q}^{\jmath}\left|k\right|^{-2d}.
\end{align*}

Similarly we obtain the estimate (\ref{eq:per_est4}):
\begin{align*}
\left\Vert c_{i,k}\right\Vert _{\jmath+\kappa} & =\left\Vert \delta_{q+1}^{1/2}\rho_{i}(t)\nabla\Phi_{i}^{-1}\Div_{\underline{x}}\left(\frac{k\wedge a_{k}\left(\underline{R_{i}}\right)}{i2\pi\lambda_{q+1}\left|k\right|^{2}}\right)\circ\mathbf{\Phi}_{i}\right\Vert _{\jmath+\kappa}\\
 & \lesssim\delta_{q+1}^{\frac{1}{2}}\left|k\right|^{-1}\lambda_{q+1}^{-1}\Big(\left\Vert \nabla\Phi_{i}^{-1}\right\Vert _{\jmath+\kappa}\left\Vert \nabla\left(a_{k}\left(\underline{R_{i}}\right)\right)\circ\mathbf{\Phi}_{i}\right\Vert _{0}\\
 &\quad+\left\Vert \nabla\Phi_{i}^{-1}\right\Vert _{0}\left\Vert \nabla\left(a_{k}\left(\underline{R_{i}}\right)\right)\circ\mathbf{\Phi}_{i}\right\Vert z_{\jmath+\kappa}\Big)\\
 & \lesssim\frac{\lambda_{q}}{\lambda_{q+1}}\delta_{q+1}^{\frac{1}{2}}\ell^{-\kappa}\lambda_{q}^{\jmath}\left|k\right|^{-2d}
\end{align*}
having used the chain rule 
\begin{equation}
\nabla\left(a_{k}\left(\underline{R_{i}}\right)\right)\circ\mathbf{\Phi}_{i}=\nabla\left(a_{k}\left(\underline{R_{i}}\circ\mathbf{\Phi}_{i}\right)\right)\left(\nabla\Phi_{i}\right)^{-1}\label{eq:chain_r}.
\end{equation}
Next, we compute
\begin{align*}
\left\Vert D_{t,q}\nabla\Phi_{i}\right\Vert _{\jmath+\kappa} & =\left\Vert \nabla_{\overline{v}_{q}}\left(\nabla\Phi_{i}\right)+\nabla\partial_{t}\Phi_{i}\right\Vert _{\jmath+\kappa}=\left\Vert \left[\nabla_{\overline{v}_{q}},\nabla\right]\Phi_{i}\right\Vert _{\jmath+\kappa}\\
 & \lesssim\left\Vert \nabla\overline{v}_{q}\right\Vert _{\jmath+\kappa}\left\Vert \nabla\Phi_{i}\right\Vert _{0}+\left\Vert \nabla\overline{v}_{q}\right\Vert _{0}\left\Vert \nabla\Phi_{i}\right\Vert _{\jmath+\kappa}\lesssim\lambda_{q}\delta_{q+1}^{\frac{1}{2}}\ell^{-\kappa}\lambda_{q}^{\jmath}
\end{align*}
which proves (\ref{eq:per_est5}), where we have used (\ref{eq:glue2}).

Then (\ref{eq:per_est5}), (\ref{eq:R_phiii}), (\ref{eq:glue3}),
and (\ref{eq:glue4}) imply (\ref{eq:per_est6}), via the estimate
\begin{align*}
&\left\Vert D_{t,q}\left(\underline{R_{i}}\circ\mathbf{\Phi}_{i}\right)\right\Vert _{\jmath+\kappa} \\
&\hspace{0.2in} \lesssim\left\Vert D_{t,q}\left(\nabla\Phi_{i}\right)\left(\textrm{Id}-\frac{\overline{R}_{q}}{\delta_{q+1}}\right)\nabla\Phi_{i}^{T}+\nabla\Phi_{i}\left(\textrm{Id}-\frac{\overline{R}_{q}}{\delta_{q+1}}\right)D_{t,q}\nabla\Phi_{i}^{T}\right\Vert _{\jmath+\kappa}\\
&\hspace{0.4in} +\delta_{q+1}^{-1}\left\Vert \nabla\Phi_{i}\left(D_{t,q}\overline{R}_{q}\right)\nabla\Phi_{i}^{T}\right\Vert _{\jmath+\kappa}\\
&\hspace{0.2in} \lesssim\lambda_{q}\delta_{q+1}^{\frac{1}{2}}\ell^{-\kappa}\lambda_{q}^{\jmath}+\left(\epsilon_{q}\tau_{q}\right)^{-1}\ell^{-\kappa}\lambda_{q}^{\jmath}\lesssim\left(\epsilon_{q}\tau_{q}\right)^{-1}\ell^{-\kappa}\lambda_{q}^{\jmath}.
\end{align*}
We recall the identities 
\begin{align*}
\partial_{\underline{t}}\left(w\circ\mathbf{X}_{i}\right) & =\left(D_{t,q}w\right)\circ\mathbf{X}_{i}\\
\partial_{\underline{t}}\underline{w}\circ\mathbf{\Phi}_{i} & =D_{t,q}\left(\underline{w}\circ\mathbf{\Phi}_{i}\right)
\end{align*}
for any tensors $w$, $\underline{w}$.  We then use (\ref{eq:per_est6}), (\ref{eq:per_est2}), and (\ref{eq:per_est1}) to prove (\ref{eq:per_est7}), via the bounds
\begin{align*}
\left\Vert D_{t,q}b_{i,k}\right\Vert _{\jmath+\kappa}&=\left\Vert \partial_{\underline{t}}\left(\mathbf{\Phi}_{i}^{*}\underline{b_{i,k}}\circ\mathbf{X}_{i}\right)\circ\mathbf{\Phi}_{i}\right\Vert _{\jmath+\kappa}\\
&=\left\Vert \partial_{\underline{t}}\left(\left(\nabla X_{i}\right)\underline{b_{i,k}}\right)\circ\mathbf{\Phi}_{i}\right\Vert _{\jmath+\kappa}\\
&= \delta_{q+1}^{1/2}\left\Vert \partial_{\underline{t}}\left(\left(\nabla X_{i}\right)\rho_{i}\left(\underline{t}\right)a_{k}\left(\underline{R_{i}}\right)\right)\circ\mathbf{\Phi}_{i}\right\Vert _{\jmath+\kappa}\\
&\lesssim \delta_{q+1}^{1/2}\left(\epsilon_{q}\tau_{q}\right)^{-1}\left\Vert \left(\left(\nabla X_{i}\right)a_{k}\left(\underline{R_{i}}\right)\right)\circ\mathbf{\Phi}_{i}\right\Vert _{\jmath+\kappa}\\
&\hspace{0.2in}+\delta_{q+1}^{1/2}\left\Vert \partial_{\underline{t}}\left(\left(\nabla X_{i}\right)a_{k}\left(\underline{R_{i}}\right)\right)\circ\mathbf{\Phi}_{i}\right\Vert _{\jmath+\kappa}\\
&\lesssim \delta_{q+1}^{1/2}\left(\epsilon_{q}\tau_{q}\right)^{-1}\left\Vert \left(\nabla\Phi_{i}\right)^{-1}a_{k}\left(\underline{R_{i}}\circ\mathbf{\Phi}_{i}\right)\right\Vert _{\jmath+\kappa}\\
&\hspace{0.2in}+\delta_{q+1}^{1/2}\left\Vert \left(\nabla\left(\overline{v}_{q}\circ\mathbf{X}_{i}\right)a_{k}\left(\underline{R_{i}}\right)\right)\circ\mathbf{\Phi}_{i}\right\Vert _{\jmath+\kappa}\\
&\hspace{0.2in} +\delta_{q+1}^{1/2}\left\Vert \left(\nabla X_{i}\right)\left(\nabla a_{k}\left(\underline{R_{i}}\right)\partial_{\underline{t}}\left(\underline{R_{i}}\right)\right)\circ\mathbf{\Phi}_{i}\right\Vert _{\jmath+\kappa}\\
&\lesssim \left(\epsilon_{q}\tau_{q}\right)^{-1}\delta_{q+1}^{1/2}\ell^{-\kappa}\lambda_{q}^{\jmath}\left|k\right|^{-2d}\\
&\hspace{0.2in}+\delta_{q+1}^{1/2}\left\Vert \left(\nabla\overline{v}_{q}\right)\left(\nabla\Phi_{i}\right)^{-1}a_{k}\left(\underline{R_{i}}\circ\mathbf{\Phi}_{i}\right)\right\Vert _{\jmath+\kappa}\\
&\hspace{0.2in} +\delta_{q+1}^{1/2}\left\Vert \left(\nabla\Phi_{i}\right)^{-1}\left(\nabla a_{k}\left(\underline{R_{i}}\circ\mathbf{\Phi}_{i}\right)D_{t,q}\left(\underline{R_{i}}\circ\mathbf{\Phi}_{i}\right)\right)\right\Vert _{\jmath+\kappa}\\
&\lesssim \left(\epsilon_{q}\tau_{q}\right)^{-1}\delta_{q+1}^{1/2}\ell^{-\kappa}\lambda_{q}^{\jmath}\left|k\right|^{-2d}\\
&\hspace{0.2in}+\delta_{q+1}^{1/2}\lambda_{q}\delta_{q+1}^{\frac{1}{2}}\ell^{-\kappa}\lambda_{q}^{\jmath}\left|k\right|^{-2d}\\
&\hspace{0.2in}+\left(\epsilon_{q}\tau_{q}\right)^{-1}\delta_{q+1}^{1/2}\ell^{-\kappa}\lambda_{q}^{\jmath}\left|k\right|^{-2d}\\
&\lesssim \left(\epsilon_{q}\tau_{q}\right)^{-1}\delta_{q+1}^{1/2}\ell^{-\kappa}\lambda_{q}^{\jmath}\left|k\right|^{-2d}.
\end{align*}

Finally, once again using (\ref{eq:chain_r}), and letting the symbol $*$ again denote an arbitrary tensor contraction, we estimate
\begin{align*}
&\left\Vert D_{t,q}c_{i,k}\right\Vert _{N}\\
&\hspace{0.2in}\ls \delta_{q+1}^{1/2}\lambda_{q+1}^{-1}\left|k\right|^{-1}\left\Vert \partial_{\underline{t}}\left(\rho_{i}(\underline{t})\left(\nabla X_{i}\right)*\nabla\left(a_{k}\left(\underline{R_{i}}\right)\right)\right)\circ\mathbf{\Phi}_{i}\right\Vert _{\jmath+\kappa}\\
&\hspace{0.2in}\lesssim \delta_{q+1}^{1/2}\left(\epsilon_{q}\tau_{q}\right)^{-1}\lambda_{q+1}^{-1}\left|k\right|^{-1}\left\Vert \left(\nabla\Phi_{i}\right)^{-1}*\nabla\left(a_{k}\left(\underline{R_{i}}\circ\mathbf{\Phi}_{i}\right)\right)*\nabla\Phi_{i}^{-1}\right\Vert _{\jmath+\kappa}\\
&\hspace{0.4in} +\delta_{q+1}^{1/2}\lambda_{q+1}^{-1}\left|k\right|^{-1}\left\Vert \nabla\overline{v}_{q}*\nabla\Phi_{i}^{-1}*\nabla\left(a_{k}\left(\underline{R_{i}}\circ\mathbf{\Phi}_{i}\right)\right)*\nabla\Phi_{i}^{-1}\right\Vert _{\jmath+\kappa}\\
&\hspace{0.4in} +\delta_{q+1}^{1/2}\lambda_{q+1}^{-1}\left|k\right|^{-1}\left\Vert \left(\nabla\Phi_{i}\right)^{-1}*\nabla\left(\nabla a_{k}\left(\underline{R_{i}}\circ\mathbf{\Phi}_{i}\right)D_{t,q}\left(\underline{R_{i}}\circ\mathbf{\Phi}_{i}\right)\right)*\nabla\Phi_{i}^{-1}\right\Vert _{\jmath+\kappa}\\
&\hspace{0.2in}\lesssim \left(\epsilon_{q}\tau_{q}\right)^{-1}\frac{\lambda_{q}}{\lambda_{q+1}}\delta_{q+1}^{1/2}\ell^{-\kappa}\lambda_{q}^{\jmath}\left|k\right|^{-2d}+\left(\lambda_{q}\delta_{q+1}^{\frac{1}{2}}\right)\frac{\lambda_{q}}{\lambda_{q+1}}\delta_{q+1}^{1/2}\ell^{-\kappa}\lambda_{q}^{\jmath}\left|k\right|^{-2d}\\
&\hspace{0.2in}\lesssim \left(\epsilon_{q}\tau_{q}\right)^{-1}\frac{\lambda_{q}}{\lambda_{q+1}}\delta_{q+1}^{1/2}\ell^{-\kappa}\lambda_{q}^{\jmath}\left|k\right|^{-2d},
\end{align*}
which establishes (\ref{eq:per_est8}).
\end{proof}

\begin{cor}
\label{cor:There-is-} There is a universal geometric constant $M=M(d)>1$
(not depending on $a,\beta,b,\sigma,\alpha,q$) such that for any
$0\leq \jmath\leq 12$, 
\begin{align}
\|w^{(c)}\|_{\jmath} & \ls_{d,\neg q}\frac{\lambda_{q}}{\lambda_{q+1}}\lambda_{q+1}^{\jmath}\delta_{q+1}^{1/2}\label{eq:cor1}\\
\|w^{(o)}\|_{\jmath} & \leq\frac{M}{4}\lambda_{q+1}^{\jmath}\delta_{q+1}^{1/2}\label{eq:cor2}\\
\|w_{q+1}\|_{\jmath} & \leq\frac{M}{2}\lambda_{q+1}^{\jmath}\delta_{q+1}^{1/2}\label{eq:cor3}.
\end{align}
\end{cor}

\begin{proof}
By choosing $a$ sufficiently large, we can arrange that $\left\Vert \nabla\Phi_{i}\right\Vert _{0}\leq2$.  From the proof of (\ref{eq:per_est3}), we note that there are $M=M(d)$
and $\overline{M}=\overline{M}(d)$ (not depending on $a,\beta,b,\sigma,\alpha,q$)
such that 
\begin{align*}
\left\Vert b_{i,k}\right\Vert _{0}\leq & \overline{M}\left|k\right|^{-2d}\delta_{q+1}^{\frac{1}{2}}\\
\sum_{k\neq0}\left\Vert b_{i,k}\right\Vert _{0} & \leq\sum_{k\neq0}\overline{M}\left|k\right|^{-2d}\delta_{q+1}^{\frac{1}{2}}\leq\frac{M}{10}\delta_{q+1}^{\frac{1}{2}}.
\end{align*}
Indeed $M$ and $\overline{M}$ only depend on the choice of $W:\mathcal{N}\times\mathbb{T}^{d}\to\mathbb{R}^{d}$.  Thus (\ref{eq:cor2}) holds when $\jmath=0$.

For higher derivatives $\partial^{\theta}w^{(o)}$ where $\left|\theta\right|=\jmath\leq12$,
the only problematic term is
\[
\sum_{k\neq0}b_{i,k}\partial^{\theta}\left(e^{i2\pi\left\langle \lambda_{q+1}k,\Phi_{i}\right\rangle }\right)
\]
as it will yield $\sum_{k\neq0}\lambda_{q+1}^{\jmath}\overline{M}\left|k\right|^{-2d}\delta_{q+1}^{\frac{1}{2}}$
(enlarging $\overline{M}$ and $M$ if necessary).  All other terms
in $\partial^{\theta}w^{(o)}$ will involve a derivative of $b_{i,k}$,
which will yield $C(d)\ell^{-1}$, and we have $C(d)\ell^{-1}\ll\frac{M}{100}\lambda_{q+1}$
for large $a$.  Thus we can absorb all appearances of $C(d)$ by
increasing $a$.  The remaining inequalities are similarly immediate.
\end{proof}

This completes the proof of (\ref{eq:perturb_1}).  

\subsection{Proof of (\ref{eq:finalR}): Estimates on the new Reynolds stress}\label{errorestimates}

Once again, throughout this subsection, we assume $t\in\left[t_{i}-\epsilon_{q}\tau_{q},t_{i}+2\epsilon_{q}\tau_{q}\right]$.
To obtain (\ref{eq:finalR}), we need only to prove
\begin{equation}
\left\Vert \widetilde{R}_{q+1}\right\Vert _{\kappa+\alpha}\lesssim_{\kappa,d,\neg q}\epsilon_{q+1}\lambda_{q+1}^{\kappa-4\alpha}\delta_{q+2}.\label{eq:fin_stress_er}
\end{equation}
We use the following standard stationary phase lemma for the anti-divergence operator; see for instance
\cite[Proposition C.2]{buckmasterOnsagerConjectureAdmissible2017}.
\begin{lem}\label{stationaryphase}
For any $N\geq1$, vector field $u\in\mathfrak{X}\left(\mathbb{T}^{d}\right)$, and phase function
$\phi\in C^{\infty}\left(\mathbb{T}^{d}\to\mathbb{T}^{d}\right)$
such that $\frac{1}{2}\leq\left|\nabla\phi\right|\leq2$,
\begin{equation}
\left\Vert \mathcal{R}\left(u(x)e^{i2\pi\left\langle k,\phi\right\rangle }\right)\right\Vert _{\alpha}\lesssim_{N}\left|k\right|^{\alpha-1}\left\Vert u\right\Vert _{0}+\left|k\right|^{\alpha-N}\left(\left\Vert u\right\Vert _{0}\left\Vert \phi\right\Vert _{N+\alpha}+\left\Vert u\right\Vert _{N+\alpha}\right).\label{eq:antidiv_est}
\end{equation}
\end{lem}

The error terms in (\ref{eq:antidiv_est}) can be suppressed by choosing $N$ sufficiently large (independently of
$q$).  In particular, we can arrange that
\begin{equation}
\ell_{q}^{N+100\alpha}\lambda_{q+1}^{N-1-100\alpha}>1,\label{eq:trick2}
\end{equation}
as long as
\[
-\frac{1}{4}-\frac{3}{4}b+b\left(\frac{N-1-100\alpha}{N+100\alpha}\right)>0
\]
which is true when $N=N\left(b,\beta,\sigma,\alpha\right)$ is large
enough.  Unless otherwise noted, we will be using such a choice of $N$.

Let us record that for any $\kappa\geq0,$ there is a trivial estimate
\begin{equation}
\left\Vert e^{i2\pi\left\langle \lambda_{q+1}k,\Phi_{i}\right\rangle }\right\Vert _{\kappa}\ls_{\kappa}\lambda_{q+1}^{\kappa}.\label{eq:phase_Ck}
\end{equation}

\subsubsection{Nash error}

By using (\ref{eq:antidiv_est}) and \Propref{pertub_est}, we have
\begin{align*}
\left\Vert \mathcal{R}\left(w^{(o)}\cdot\nabla\overline{v}_{q}\right)\right\Vert _{\alpha} & \lesssim\sum_{k\neq0}\left\Vert \mathcal{R}\left(b_{i,k}\cdot\nabla\overline{v}_{q}e^{i2\pi\left\langle \lambda_{q+1}k,\Phi_{i}\right\rangle }\right)\right\Vert _{\alpha}\\
 & \lesssim_{N}\sum_{k\neq0}\left|\lambda_{q+1}k\right|^{\alpha-1}\left|k\right|^{-2d}\delta_{q+1}^{\frac{1}{2}}\left(\lambda_{q}\delta_{q+1}^{\frac{1}{2}}\right)\lambda_{q+1}^{3\alpha}\\
 & \quad+\left|\lambda_{q+1}k\right|^{\alpha-N}\left|k\right|^{-2d}\left(\delta_{q+1}^{\frac{1}{2}}\lambda_{q}\delta_{q+1}^{\frac{1}{2}}\ell^{-N-2\alpha}\right)\\
 & \lesssim\lambda_{q+1}^{4\alpha}\frac{\lambda_{q}}{\lambda_{q+1}}\delta_{q+1}\lesssim\epsilon_{q+1}\delta_{q+2}\lambda_{q+1}^{-4\alpha}
\end{align*}
where we used (\ref{eq:trick2}) to pass to the last line, and (\ref{eq:stress_size_ind1})
in the last inequality.  In addition, for any $\kappa\geq1$,
\begin{align*}
\left\Vert \mathcal{R}\left(w^{(o)}\cdot\nabla\overline{v}_{q}\right)\right\Vert _{\kappa+\alpha} & \ls_{\kappa}\sum_{k\neq0}\left\Vert b_{i,k}\cdot\nabla\overline{v}_{q}e^{i2\pi\left\langle \lambda_{q+1}k,\Phi_{i}\right\rangle }\right\Vert _{\kappa-1+\alpha}\\
&\ls\lambda_{q+1}^{\kappa-1+3\alpha}\delta_{q+1}^{\frac{1}{2}}\left(\lambda_{q}\delta_{q+1}^{\frac{1}{2}}\right)\\
&=\lambda_{q+1}^{\kappa+3\alpha}\frac{\lambda_{q}}{\lambda_{q+1}}\delta_{q+1}\\
&\ls\epsilon_{q+1}\lambda_{q+1}^{\kappa-4\alpha}\delta_{q+2}
\end{align*}
where (\ref{eq:stress_size_ind1}) is used again in the last inequality.  Similarly, for any $\kappa\geq1$,
\begin{align*}
\left\Vert \mathcal{R}\left(w^{(c)}\cdot\nabla\overline{v}_{q}\right)\right\Vert _{\alpha} & \lesssim\sum_{k\neq0}\left\Vert \mathcal{R}\left(c_{i,k}\cdot\nabla\overline{v}_{q}e^{i2\pi\left\langle \lambda_{q+1}k,\Phi_{i}\right\rangle }\right)\right\Vert _{\alpha}\\
&\lesssim\frac{\lambda_{q}}{\lambda_{q+1}}\epsilon_{q+1}\delta_{q+2}\lambda_{q+1}^{-4\alpha}
\end{align*}
and
\begin{align*}
\left\Vert \mathcal{R}\left(w^{(c)}\cdot\nabla\overline{v}_{q}\right)\right\Vert _{\kappa+\alpha} & \ls\sum_{k\neq0}\left\Vert c_{i,k}\cdot\nabla\overline{v}_{q}e^{i2\pi\left\langle \lambda_{q+1}k,\Phi_{i}\right\rangle }\right\Vert _{\kappa-1+\alpha}\\
&\ls\frac{\lambda_{q}}{\lambda_{q+1}}\epsilon_{q+1}\lambda_{q+1}^{\kappa-4\alpha}\delta_{q+2}
\end{align*}
which is an improvement by the small factor  $\frac{\lambda_{q}}{\lambda_{q+1}}$.  Thus we have 
\[
\left\Vert R_{\mathrm{Nash}}\right\Vert _{\kappa+\alpha}\lesssim_{\kappa}\epsilon_{q+1}\lambda_{q+1}^{\kappa-4\alpha}\delta_{q+2}
\]
for any $\kappa\geq0$.

\subsubsection{Transport error}

By construction of $\Phi$ (see (\ref{phi}) and the surrounding discussion), we have the key identity $D_{t,q}\left(e^{i2\pi\left\langle \lambda_{q+1}k,\Phi_{i}\right\rangle }\right)=0$ through which we can avoid factors of $\lambda_{q+1}$ in the estimates.

Compared to $\mathcal{R}\left(w^{(o)}\cdot\nabla\overline{v}_{q}\right)$,
the estimates of $\mathcal{R}D_{t,q}w^{(o)}$ will have an extra factor
$\epsilon_{q}^{-1}\lambda_{q}^{3\alpha}$, as $\nabla\overline{v}_{q}$
costs $\lambda_{q}\delta_{q+1}^{\frac{1}{2}}$ while $D_{t,q}$ costs
$\epsilon_{q}^{-1}\tau_{q}^{-1}$.  Otherwise the calculations are identical and we have, using (\ref{eq:stress_size_ind1}),
\begin{align*}
\left\Vert \mathcal{R}D_{t,q}w^{(o)}\right\Vert _{\alpha} & \lesssim\lambda_{q+1}^{7\alpha}\epsilon_{q}^{-1}\frac{\lambda_{q}}{\lambda_{q+1}}\delta_{q+1}\lesssim\epsilon_{q+1}\lambda_{q+1}^{-4\alpha}\delta_{q+2},\\
\left\Vert \mathcal{R}D_{t,q}w^{(o)}\right\Vert _{\kappa+\alpha} & \ls\lambda_{q+1}^{\kappa+6\alpha}\epsilon_{q}^{-1}\frac{\lambda_{q}}{\lambda_{q+1}}\delta_{q+1}\ls\epsilon_{q+1}\lambda_{q+1}^{\kappa-4\alpha}\delta_{q+2},\\
\left\Vert \mathcal{R}D_{t,q}w^{(c)}\right\Vert _{\alpha} & \lesssim\frac{\lambda_{q}}{\lambda_{q+1}}\epsilon_{q+1}\lambda_{q+1}^{-4\alpha}\delta_{q+2},\\
\end{align*}
and
\begin{align*}
\left\Vert \mathcal{R}D_{t,q}w^{(c)}\right\Vert _{\kappa+\alpha} & \lesssim\frac{\lambda_{q}}{\lambda_{q+1}}\epsilon_{q+1}\lambda_{q+1}^{\kappa-4\alpha}\delta_{q+2}
\end{align*}
for $\kappa\geq1$.  Thus we have 
\[
\left\Vert R_{\mathrm{trans}}\right\Vert _{\kappa+\alpha}\lesssim_{\kappa}\epsilon_{q+1}\lambda_{q+1}^{\kappa-4\alpha}\delta_{q+2}
\]
for any $\kappa\geq0$.

\subsubsection{Oscillation error}

We employ the decomposition
\begin{align*}
R_{\mathrm{osc}}=\mathcal{R}\Div\left(\overline{R}_{q}+w_{q+1}\otimes w_{q+1}\right)=\mathcal O_1+\mathcal O_2
\end{align*}
where
\eqn{
\mathcal O_1&\,\coloneqq\,\mathcal{R}\Div\left(\overline{R}_{q}+w^{(o)}\otimes w^{(o)}\right),\\
\mathcal O_2&\,\coloneqq\,\mathcal{R}\Div\left(w^{(c)}\otimes w^{(o)}+w^{(o)}\otimes w^{(c)}+w^{(c)}\otimes w^{(c)}\right).
}
Using that $\mathcal{R}\Div$
is a Calderón-Zygmund operator along with the bounds in \Corref{There-is-}, for any $\kappa\geq0$,
\begin{align*}
\left\Vert \mathcal{O}_{2}\right\Vert _{\kappa+\alpha} & \lesssim\left\Vert w^{(c)}\right\Vert _{\kappa+\alpha}\left\Vert w^{(o)}\right\Vert _{\alpha}+\left\Vert w^{(c)}\right\Vert _{\alpha}\left\Vert w^{(o)}\right\Vert _{\kappa+\alpha}+\left\Vert w^{(c)}\right\Vert _{\kappa+\alpha}\left\Vert w^{(c)}\right\Vert _{\alpha}\\
 & \lesssim\frac{\lambda_{q}}{\lambda_{q+1}}\lambda_{q+1}^{\kappa+2\alpha}\delta_{q+1}\ls\epsilon_{q+1}\lambda_{q+1}^{\kappa-4\alpha}\delta_{q+2}
\end{align*}
where we once again use the parameter relations (\ref{eq:stress_size_ind1}).  Next, one computes using (\ref{eq:divf1}) that
we have 
\begin{align*}
\mathcal{O}_{1} & =\sum_{k\in\mathbb{Z}^{d}\backslash\{0\}}\delta_{q+1}\rho_{i}^{2}\mathcal{R}\left(\Div\left(\mathbf{\Phi}_{i}^{*}\left(C_{k}\left(\underline{R_{i}}\right)\right)\right)e^{i2\pi\left\langle \lambda_{q+1}k,\Phi_{i}\right\rangle }\right),
\end{align*}
see Section 5.3.3 in \cite{bulutEpochsRegularityWild2022} for the detailed calculation.  Then (\ref{eq:antidiv_est})
and (\ref{eq:stress_size_ind1}) allow us to estimate
\begin{align*}
\left\Vert \mathcal{O}_{1}\right\Vert _{\alpha} & \lesssim\sum_{k\in\mathbb{Z}^{d}\backslash\{0\}}\left\Vert \delta_{q+1}\mathcal{R}\left(\Div\left(\nabla\Phi_{i}^{-1}C_{k}\left(\underline{R_{i}}\circ\mathbf{\Phi}_{i}\right)\nabla\Phi_{i}^{-T}\right)e^{i2\pi\left\langle \lambda_{q+1}k,\Phi_{i}\right\rangle }\right)\right\Vert _{\alpha}\\
 & \lesssim_{N}\sum_{k\neq0}\left|\lambda_{q+1}k\right|^{\alpha-1}\left|k\right|^{-2d}\lambda_{q}\delta_{q+1}+\left|\lambda_{q+1}k\right|^{\alpha-N}\left|k\right|^{-2d}\left(\lambda_{q}\delta_{q+1}\right)\ell^{-N-4\alpha}\\
 & \lesssim\lambda_{q+1}^{\alpha}\frac{\lambda_{q}}{\lambda_{q+1}}\delta_{q+1}\ls\epsilon_{q+1}\lambda_{q+1}^{-4\alpha}\delta_{q+2}
\end{align*}
and similarly, for any $\kappa\geq1$,
\begin{align*}
\left\Vert \mathcal{O}_{1}\right\Vert _{\kappa+\alpha} & \ls\delta_{q+1}\left\Vert \Div\left(\nabla\Phi_{i}^{-1}C_{k}\left(\underline{R_{i}}\circ\mathbf{\Phi}_{i}\right)\nabla\Phi_{i}^{-T}\right)e^{i2\pi\left\langle \lambda_{q+1}k,\Phi_{i}\right\rangle }\right\Vert _{\kappa-1+\alpha}\\
 & \ls\delta_{q+1}\lambda_{q+1}^{\kappa+4\alpha}\frac{\lambda_{q}}{\lambda_{q+1}}\ls\epsilon_{q+1}\lambda_{q+1}^{\kappa-4\alpha}\delta_{q+2}.
\end{align*}
We conclude that for any $\kappa\geq0$,
\[
\left\Vert R_{\mathrm{osc}}\right\Vert _{\kappa+\alpha}\lesssim_{\kappa}\epsilon_{q+1}\lambda_{q+1}^{\kappa-4\alpha}\delta_{q+2}
\]
which completes the proof of (\ref{eq:fin_stress_er}).

We have thus completed the proof of \Propref{convex_int}.

\section{Conclusion of the proof of Proposition \ref{prop:iterative}}
\label{sec:proofProp2}

In this section, we complete the proof of \Propref{iterative}.

\begin{proof}[Proof of Proposition \ref{prop:iterative}]
Note that 
\begin{equation}
\begin{cases}
\dd_{t}\overline{u}_{q}+\divop\overline{u}_{q}\otimes\overline{u}_{q}+\grad\overline{\pi}_{q}=\divop\overline{F}_{q}\\
\dd_{t}v_{q+1}+\divop v_{q+1}\otimes v_{q+1}+\grad p_{q+1}=\divop\overline{F}_{q}+\divop\left(\chi^{g}F_{q}+\chi^{b}F_{\ell}-\overline{F}_{q}+\widetilde{R}_{q+1}\right).
\end{cases}\label{eq:forced_Euler_dual_system-1}
\end{equation}

Defining $u_{q+1}{\,\coloneqq\,}\overline{u}_{q}$, $F_{q+1}{\,\coloneqq\,}\overline{F}_{q}$, and $R_{q+1}{\,\coloneqq\,}\chi^{g}F_{q}+\chi^{b}F_{\ell}-\overline{F}_{q}+\widetilde{R}_{q+1}$, we clearly have (\ref{eq:Rq_zerowhere}) with $q$ changed to $q+1$, as well as (\ref{eq:uvF_on_Gq}).  By using (\ref{eq:uvF_sup}), (\ref{eq:perturb_1}), and (\ref{eq:glue1}), it then follows that we have
\begin{align*}
\left\Vert u_{q+1}\right\Vert _{0} & \leq\left\Vert u_{q}\right\Vert _{0}\leq1-\delta_{q}^{\frac{1}{2}}\leq1-\delta_{q+1}^{\frac{1}{2}}
\end{align*}
and
\begin{align*}
\left\Vert v_{q+1}\right\Vert _{0} & \leq\left\Vert v_{q+1}-\overline{v}_{q}\right\Vert _{0}+\left\Vert \overline{v}_{q}-v_{\ell}\right\Vert _{0}+\left\Vert v_{\ell}\right\Vert _{0}\\
 & \leq\frac{M}{2}\delta_{q+1}^{\frac{1}{2}}+C(d)\epsilon_{q}\delta_{q+1}^{\frac{1}{2}}+1-\delta_{q}^{\frac{1}{2}}\ll1-\delta_{q+1}^{\frac{1}{2}}
\end{align*}
provided $a$ is chosen sufficiently large.  Combining this with (\ref{eq:Fq+1_0}), we have therefore proven (\ref{eq:vq_sup-1}).

Similarly, for any $1\leq \jmath\leq 12$, by (\ref{eq:grad_uq_sup}), (\ref{eq:perturb_1}) and (\ref{eq:glue2}) we have
\begin{align*}
\left\Vert u_{q+1}\right\Vert _{\jmath} & =\left\Vert \chi^{g}u_{q}+\chi^{b}u_{\ell}\right\Vert _{\jmath}\leq\left\Vert u_{q}\right\Vert _{\jmath}\leq M\lambda_{q}^{\jmath}\delta_{q}^{\frac{1}{2}}
\end{align*}
and
\begin{align*}
\left\Vert v_{q+1}\right\Vert _{\jmath} & \leq\left\Vert v_{q+1}-\overline{v}_{q}\right\Vert _{\jmath}+\left\Vert \overline{v}_{q}\right\Vert _{\jmath}\leq\frac{M}{2}\lambda_{q+1}^{\jmath}\delta_{q+1}^{\frac{1}{2}}+C(d)\ell^{-\left(\jmath-1\right)}\lambda_{q}^{1-\alpha}\delta_{q+1}^{\frac{1}{2}}\ll M\lambda_{q+1}^{\jmath}\delta_{q+1}^{\frac{1}{2}}
\end{align*}
for large enough $a$, so that, together with (\ref{eq:F_q+1_bound_noloss}), we have proven (\ref{eq:grad_uq_sup-1}), (\ref{eq:grad_vq_sup-1}), and (\ref{eq:Fq_sup-1}). 

Turning to $R_{q+1}$, we note that for any $0\leq \jmath\leq 12$, we have
\begin{align*}
\left\Vert R_{q+1}\right\Vert _{\jmath} & \leq\left\Vert \widetilde{R}_{q+1}\right\Vert _{\jmath}+\left\Vert \chi^{g}F_{q}+\chi^{b}F_{\ell}-\overline{F}_{q}\right\Vert _{\jmath}\leq M\epsilon_{q+1}\lambda_{q+1}^{\jmath-3\alpha}\delta_{q+2}
\end{align*}
because of (\ref{eq:finalR}) and (\ref{eq:commutator_term_in_Fq}); thus we have proven (\ref{eq:Rq_sup-1}).

Finally, for $\jmath\in \{0,1\}$, by (\ref{eq:ul-uq}), (\ref{eq:perturb_1}),
(\ref{eq:glue1}), (\ref{eq:grad_vq_sup}), (\ref{eq:Fl-fq_est-1}),
 and (\ref{eq:Fbar-Fl_noloss}) we have,
\begin{align*}
\|u_{q+1}-u_{q}\|_{\jmath} & =\chi^{b}\|u_{\ell}-u_{q}\|_{\jmath}\leq C(d)\epsilon_{q+1}\lambda_{q+1}^{\jmath}\delta_{q+2}^{\frac{1}{2}}\ll M\lambda_{q+1}^{\jmath}\delta_{q+1}^{\frac{1}{2}},
\end{align*}
\begin{align*}
\|v_{q+1}-v_{q}\|_{\jmath} & \leq\left\Vert v_{q+1}-\overline{v}_{q}\right\Vert _{\jmath}+\left\Vert \overline{v}_{q}-v_{\ell}\right\Vert _{\jmath}+\left\Vert v_{\ell}-v_{q}\right\Vert _{\jmath}\\
 & \leq\frac{M}{2}\lambda_{q+1}^{\jmath}\delta_{q+1}^{\frac{1}{2}}+C(d)\epsilon_{q}\lambda_{q}^{\jmath}\delta_{q+1}^{\frac{1}{2}}+\ell C(d)\lambda_{q-1}^{1+\jmath}\delta_{q-1}^{\frac{1}{2}}\\
 & \ll M\lambda_{q+1}^{\jmath}\delta_{q+1}^{\frac{1}{2}},
 \end{align*}
 and
 \begin{align*}
\|F_{q+1}-F_{q}\|_{\jmath} & \leq\left\Vert \overline{F}_{q}-F_{\ell}\right\Vert _{\jmath}+\left\Vert F_{\ell}-F_{q}\right\Vert _{\jmath}\\
 & \le C(d)\epsilon_{q+1}\lambda_{q+1}^{\jmath-4\alpha}\delta_{q+2}+C(d)\epsilon_{q}\lambda_{q}^{\jmath-3\alpha}\delta_{q+1}\ll M\lambda_{q+1}^{\jmath}\delta_{q+1}
\end{align*}
which prove (\ref{eq:iter_prop_vq_est}), (\ref{eq:iter_prop_uq_est}), and (\ref{eq:iter_prop_Fq_est}).  We note that the constants $C(d)$ are absorbed by increasing
$a$. 

To conclude, we remark that all the properties regarding $\mathcal{B}_{q+1}$ were proven
in \Subsecref{temporal_cutoff}.
\end{proof}

\appendix

\section{Parameter comparisons}
\label{app:parameters}

In this appendix, we record some useful relations between the parameters used in the convex integration construction.  These are used
routinely in Sections $4$--$7$.  We begin by noting the ``essential conversions''
\begin{align}
\tau_{q}\delta_{q+1}^{\frac{1}{2}}\lambda_{q} & \ll1,\label{eq:time-length}
\end{align}
\begin{align}
\ell_{q}=\lambda_{q}^{-\frac{1}{4}}\lambda_{q+1}^{-\frac{3}{4}}\ll\left(\lambda_{q}\lambda_{q+1}\right)^{-\frac{1}{2}} & \ll\epsilon_{q}^{\frac{1}{2}}\frac{\delta_{q+1}^{\frac{1}{2}}}{\lambda_{q}^{1+\frac{3\alpha}{2}}\delta_{q}^{\frac{1}{2}}},\label{eq:ell_smaller}
\end{align}
and
\begin{align}
\lambda_{q}^{1+4\alpha}\ll\ell_{q}^{-1} & \ll\lambda_{q+1}^{1-4\alpha}\ll\lambda_{q}^{\frac{3}{2}}.\label{eq:length_Freq}
\end{align}
Observe that (\ref{eq:time-length}) comes from $\alpha>0$.  On the other hand, since $\alpha$ can be made arbitrarily small by (\ref{eq:alpha_bound}),
(\ref{eq:ell_smaller}) comes from
\begin{align*}
-\frac{1}{2}-\frac{b}{2} & <-\frac{\sigma}{2}-1-b\beta+\beta\\
\iff\sigma & <\left(b-1\right)\left(1-2\beta\right)
\end{align*}
which is implied by (\ref{eq:sigma_bound}).  Finally, (\ref{eq:length_Freq}) is self-evident.

In addition, we have the ``double-skipping'' iteration
\begin{align}
\lambda_{q-1}^{1+20\alpha}\delta_{q-1}^{\frac{1}{2}} & \ll\epsilon_{q}\lambda_{q}^{1-20\alpha}\delta_{q+1}^{\frac{1}{2}}\ll\epsilon_{q}\lambda_{q}^{1-20\alpha}\delta_{q}^{\frac{1}{2}}\label{eq:double_skip_iteration}
\end{align}
because 
\begin{align*}
1-\beta & <-b\sigma+b-b^{2}\beta\\
\iff & \sigma<\frac{\left(b-1\right)\left(1-\beta b-\beta\right)}{b}
\end{align*}
which is implied by (\ref{eq:sigma_bound}).

Lastly, we record the iteration inequalities 
\begin{align}
\epsilon_{q}^{-1}\lambda_{q}^{1+20\alpha}\delta_{q+1} & \ll\epsilon_{q+1}\delta_{q+2}\lambda_{q+1}^{1-20\alpha},\label{eq:stress_size_ind1}
\end{align}
\begin{align}
\epsilon_{q}^{-1}\lambda_{q}^{1+20\alpha}\delta_{q+1}^{\frac{1}{2}} & \ll\epsilon_{q+1}\delta_{q+2}^{\frac{1}{2}}\lambda_{q+1}^{1-20\alpha},\label{eq:stress_size_ind2}
\end{align}
and
\begin{align}
\lambda_{q-1}^{1+20\alpha}\delta_{q}^{\frac{1}{2}}\ll\lambda_{q}^{1+20\alpha}\delta_{q+1}^{\frac{1}{2}} & \ll\epsilon_{q}\epsilon_{q+1}\lambda_{q+1}^{1-20\alpha}\delta_{q+2}^{\frac{1}{2}}\label{eq:stress_size_ind3}.
\end{align}
Indeed, we observe that (\ref{eq:stress_size_ind1}) comes from
\[
-2b\beta-b+1+\sigma\leq-b\sigma-b^{2}\left(2\beta\right)
\]
which is precisely (\ref{eq:sigma_bound}).  Then (\ref{eq:stress_size_ind2}) is an immediate consequence, as $\frac{\delta_{q+2}}{\delta_{q+1}}<1$.  The bound (\ref{eq:stress_size_ind3}) follows immediately as well.

\section{Local existence estimates\protect\label{app:Local-existence-of}}

In this appendix, we give a proof of \Lemref{local_existence}.  As a first step, we recall a nonlinear Gr\"onwall inequality due to LaSalle \cite{lasalleUniquenessTheoremsSuccessive1949}.
\begin{lem}
\label{lem:lasalle}
Assume $A,C>0$ and $f$ is a continuous non-negative function such
that 
\[
f\left(t\right)\leq A+C\int_{0}^{t}f(s)^{2}\;\mathrm{d}s
\]
Then for $t\in\left(0,\frac{1}{2AC}\right)$ we have 
\[
f\left(t\right)\leq\frac{1}{A^{-1}-Ct}\leq\frac{2}{A^{-1}}=2A
\]
\end{lem}

\begin{proof}
Let $F(t)=\int_{0}^{t}f(s)^{2}\;\mathrm{d}s$.  Then $F'(t)\leq\left(A+CF(t)\right)^{2}$,
or equivalently 
\[
\partial_{t}\left((A+CF)^{-1}\right)\geq-C,
\]
which implies $(A+CF(t))^{-1}-A^{-1}\geq-Ct$.  Then for $t\in(0,\left(2AC\right)^{-1})$
we have
\[
f\leq A+CF\leq\frac{1}{A^{-1}-Ct}.
\]
\end{proof}
We now consider the Euler equations
\[
\begin{cases}
\partial_{t}v+\divop\left(v\otimes v\right)+\nabla p=f\\
\divop v=0\\
v(0)=v_{0}
\end{cases}
\]
where $f$ and $v_{0}$ are smooth.

It is well-known (cf. \cite{temamLocalExistenceInf1976}) that there
exists a smooth solution $v$ on $[0,T^{*})$ where $T^{*}$ is the
maximal time of existence, and that for $m>\frac{d}{2}+1$, if $\left\Vert v(t)\right\Vert _{H^{m}}$
stays bounded on $[0,T)$ for some $T\in\left(0,\infty\right)$, then
$T^{*}>T$.  It follows that if $\left\Vert v(t)\right\Vert _{C^{m}}$ stays bounded
on $[0,T)$ then $T^{*}>T$.

Let $\theta$ be a multi-index with $\left|\theta\right|=N\in\mathbb{N}_{1}$.
Then we have 
\begin{align*}
\left(\partial_{t}+\nabla_{v}\right)\partial^{\theta}v+\left[\partial^{\theta},\nabla_{v}\right]v+\nabla\partial^{\theta}p & =\partial^{\theta}f
\end{align*}
Let $\varepsilon>0$ be small.  Then by the transport estimate:
\begin{align*}
\left\Vert \partial^{\theta}v(t)\right\Vert _{\varepsilon} & \ls_{\varepsilon}\left\Vert v\left(0\right)\right\Vert _{N+\varepsilon}+\int_{0}^{t}\mathrm{d}s\;\left\Vert \left[\partial^{\theta},\nabla_{v}\right]v(s)\right\Vert _{\varepsilon}+\left\Vert p(s)\right\Vert _{N+1+\varepsilon}+\left\Vert f(s)\right\Vert _{N+\varepsilon}
\end{align*}
We observe that 
\[
\left\Vert \left[\partial^{\theta},\nabla_{v}\right]v(s)\right\Vert _{\varepsilon}\ls\left\Vert v(s)\right\Vert _{1+\varepsilon}\left\Vert v(s)\right\Vert _{N+\varepsilon}
\]
 while 
\begin{align*}
\left\Vert p(s)\right\Vert _{N+1+\varepsilon} & \ls\left\Vert \left(-\Delta\right)^{-1}\left(\nabla v*\nabla v\right)(s)\right\Vert _{N+1+\varepsilon}+\left\Vert \left(-\Delta\right)^{-1}\divop f(s)\right\Vert _{N+1+\varepsilon}\\
 & \ls\left\Vert v(s)\right\Vert _{N+\varepsilon}\left\Vert v(s)\right\Vert _{1+\varepsilon}+\left\Vert f(s)\right\Vert _{N+\varepsilon}
\end{align*}
So for $t\in[0,T)$ we have
\begin{equation}
\left\Vert v(t)\right\Vert _{N+\varepsilon}\ls_{N,\varepsilon}\left\Vert v\left(0\right)\right\Vert _{N+\varepsilon}+t\left\Vert f\right\Vert _{L_{t}^{\infty}C_{x}^{N+\varepsilon}}+\int_{0}^{t}\mathrm{d}s\;\left\Vert v(s)\right\Vert _{N+\varepsilon}\left\Vert v(s)\right\Vert _{1+\varepsilon}\label{eq:gronwall_local}
\end{equation}

\begin{lem}
Let $\varepsilon>0$ be small and $T>0$.  If $\left\Vert v(t)\right\Vert _{1+\varepsilon}$
stays bounded on $[0,T)$ then $T^{*}>T$.
\end{lem}

\begin{proof}
Let $N>\frac{d}{2}+1$ and $\left\Vert v(t)\right\Vert _{1+\varepsilon}\leq B$
on $[0,T)$.  From (\ref{eq:gronwall_local}), by Gr\"onwall, we have
\[
\left\Vert v(t)\right\Vert _{N+\varepsilon}\ls_{N,\varepsilon}\left(\left\Vert v\left(0\right)\right\Vert _{N+\varepsilon}+T\left\Vert f\right\Vert _{L_{t}^{\infty}C_{x}^{N+\varepsilon}}\right)e^{TB}
\]
for $t\in[0,T).$ Therefore, $T^{*}>T$.
\end{proof}

Without loss of generality, we can assume that $\left\Vert v\left(0\right)\right\Vert _{N+\varepsilon}+T\left\Vert f\right\Vert _{L_{t}^{\infty}C_{x}^{N+\varepsilon}}>0$
(otherwise the solution is just the zero solution).  Then, for $\left|\theta\right|=1$, by Lemma \ref{lem:lasalle}, for any $T\in[0,T^{*})$ and 
\[
0\leq t<\min\left\{ T,\frac{1}{2}\left(C^{2}\left\Vert v\left(0\right)\right\Vert _{1+\varepsilon}+C^{2}T\left\Vert f\right\Vert _{L_{t}^{\infty}C_{x}^{1+\varepsilon}}\right)^{-1}\right\} 
\]
 where $C=C\left(N,\varepsilon\right)$, we have 
\begin{align*}
\left\Vert v(t)\right\Vert _{1+\varepsilon} & \leq\frac{1}{\left(C\left\Vert v\left(0\right)\right\Vert _{1+\varepsilon}+CT\left\Vert f\right\Vert _{L_{t}^{\infty}C_{x}^{1+\varepsilon}}\right)^{-1}-Ct}\\
 & \leq2C\left\Vert v\left(0\right)\right\Vert _{1+\varepsilon}+2CT\left\Vert f\right\Vert _{L_{t}^{\infty}C_{x}^{1+\varepsilon}}
\end{align*}
Therefore, if we let 
\[
\tau\ls\min\left\{ \left\Vert v\left(0\right)\right\Vert _{1+\varepsilon}^{-1},\left\Vert f\right\Vert _{L_{t}^{\infty}C_{x}^{1+\varepsilon}}^{-1/2}\right\} 
\]
 then $\left\Vert v(t)\right\Vert _{1+\varepsilon}$ stays bounded
on $[0,\min\{T^{*},\tau\})$ and 
\[
\left\Vert v(t)\right\Vert _{1+\varepsilon}\ls_{\varepsilon}\left\Vert v\left(0\right)\right\Vert _{1+\varepsilon}+\tau\left\Vert f\right\Vert _{L_{t}^{\infty}C_{x}^{1+\varepsilon}}
\]
 This implies $T^{*}>\min\{T^{*},\tau\}$ and $T^{*}>\tau$.  So \Lemref{local_existence}
is proven for $N=1$.

For $N>1$, by Gr\"onwall, (\ref{eq:gronwall_local}) implies 
\begin{align*}
\left\Vert v(t)\right\Vert _{N+\varepsilon} & \ls_{N,\varepsilon}\left(\left\Vert v\left(0\right)\right\Vert _{N+\varepsilon}+\tau\left\Vert f\right\Vert _{L_{t}^{\infty}C_{x}^{N+\varepsilon}}\right)\exp\left(2C\tau\left\Vert v\left(0\right)\right\Vert _{1+\varepsilon}+2C\tau^{2}\left\Vert f\right\Vert _{L_{t}^{\infty}C_{x}^{1+\varepsilon}}\right)\\
 & \ls\left\Vert v\left(0\right)\right\Vert _{N+\varepsilon}+\tau\left\Vert f\right\Vert _{L_{t}^{\infty}C_{x}^{N+\varepsilon}}
\end{align*}
on $[0,\tau].$ Thus \Lemref{local_existence} is proven.

\section{Onsager exponent\protect\label{app:Onsager-exponent}}

In this appendix, we show (for completeness of our exposition) that the energy balance is conserved when the regularity is above $\frac{1}{3}$. 
\begin{prop}
Assume $\beta>\frac{1}{3}$, $u\in C_{t,x}^{\beta}$ and $F\in C_{t,x}^{2\beta}$
where 
\begin{align*}
\partial_{t}u+\mathbb{P}\divop\left(u\otimes u\right) & =\mathbb{P}\divop F\\
\nabla\cdot u & =0
\end{align*}
Writing $\left\langle \left\langle U,V\right\rangle \right\rangle {\,\coloneqq\,}\int_{\mathbb{T}^{d}}\left\langle U,V\right\rangle $
for vector fields $U,V$ (or tensors of the same rank), we have
\[
\frac{1}{2}\left\langle \left\langle u(t),u(t)\right\rangle \right\rangle -\frac{1}{2}\left\langle \left\langle u(0),u(0)\right\rangle \right\rangle =\int_{0}^{t}\left\langle \left\langle \divop F(s),u(s)\right\rangle \right\rangle \;\mathrm{d}s
\].
\end{prop}

\begin{rem*}
The right-hand side is well-defined, and 
\begin{align*}
\left|\left\langle \left\langle \divop F(s),u(s)\right\rangle \right\rangle \right|&\ls_{\beta,\mu}\left\Vert \divop F(s)\right\Vert _{B_{2,2}^{2\beta-1-\mu}}\left\Vert u(s)\right\Vert _{B_{2,2}^{\mu+1-2\beta}}\\
&\ls_{\beta,\mu}\left\Vert F(s)\right\Vert _{B_{\infty,\infty}^{2\beta}}\left\Vert u(s)\right\Vert _{B_{\infty,\infty}^{\beta}},
\end{align*}
where $\mu\in\left(0,3\beta-1\right)$ and $B_{p,q}^{s}$ are the
usual Besov spaces (see, for instance \cite{triebel_function_I}).
\end{rem*}
\begin{proof}
As in \Subsecref{mollifier}, let $\psi_{\ell}$ be a
smooth standard radial mollifier in space of length $\ell$.  For any
$\varepsilon>0$ small, we write
\begin{align*}
u^{\varepsilon} & =u*\psi_{\varepsilon} & \left(u\otimes u\right)^{\varepsilon} & =\left(u\otimes u\right)*\psi_{\varepsilon} & F^{\varepsilon}=F*\psi_{\varepsilon}.
\end{align*}
Observe that $\partial_{t}u^{\varepsilon}+\mathbb{P}\divop\left(u\otimes u\right)^{\varepsilon}=\mathbb{P}\divop F^{\varepsilon}$.
Then 
\begin{align*}
 & \frac{1}{2}\left\langle \left\langle u(t),u(t)\right\rangle \right\rangle -\frac{1}{2}\left\langle \left\langle u(0),u(0)\right\rangle \right\rangle \\
= & \lim_{\varepsilon\to0}\frac{1}{2}\left\langle \left\langle u^{\varepsilon}(t),u^{\varepsilon}(t)\right\rangle \right\rangle -\frac{1}{2}\left\langle \left\langle u^{\varepsilon}(0),u^{\varepsilon}(0)\right\rangle \right\rangle \\
= & \lim_{\varepsilon\to0}\int_{0}^{t}\left\langle \left\langle \partial_{t}u^{\varepsilon}(s),u^{\varepsilon}(s)\right\rangle \right\rangle \;\mathrm{ds}\\
= & \lim_{\varepsilon\to0}\int_{0}^{t}\left\langle \left\langle \left(u\otimes u\right)^{\varepsilon}(s),\nabla u^{\varepsilon}(s)\right\rangle \right\rangle \;\mathrm{ds}+\int_{0}^{t}\left\langle \left\langle \divop F^{\varepsilon}(s),u^{\varepsilon}(s)\right\rangle \right\rangle \;\mathrm{d}s.
\end{align*}
We observe that 
\begin{align*}
 & \left|\left\langle \left\langle \left(u\otimes u\right)^{\varepsilon}(s),\nabla u^{\varepsilon}(s)\right\rangle \right\rangle \right|=\left|\left\langle \left\langle \left(u\otimes u\right)^{\varepsilon}(s)-u^{\varepsilon}\otimes u^{\varepsilon}(s),\nabla u^{\varepsilon}(s)\right\rangle \right\rangle \right|\\
\ls & \left\Vert \left(u\otimes u\right)^{\varepsilon}(s)-u^{\varepsilon}\otimes u^{\varepsilon}(s)\right\Vert _{0}\left\Vert \nabla u^{\varepsilon}(s)\right\Vert _{0}\ls\varepsilon^{2\beta}\left\Vert u(s)\right\Vert _{\beta}^{2}\varepsilon^{\beta-1}\left\Vert u(s)\right\Vert _{\beta}
\end{align*}
where we used the commutator estimate (\ref{comm1}).  As $\beta>\frac{1}{3}$
we conclude 
\[
\lim_{\varepsilon\to0}\int\left\langle \left\langle u^{\varepsilon}\otimes u^{\varepsilon}(s),\nabla u^{\varepsilon}(s)\right\rangle \right\rangle \;\mathrm{ds}=0.
\]

Then, by letting $\mu\in\left(0,3\beta-1\right)$, we observe that
\begin{align*}
 & \left|\left\langle \left\langle \divop F^{\varepsilon},u^{\varepsilon}\right\rangle \right\rangle -\left\langle \left\langle \divop F,u\right\rangle \right\rangle \right|\\
\leq & \left|\left\langle \left\langle \divop\left(F^{\varepsilon}-F\right),u^{\varepsilon}\right\rangle \right\rangle \right|+\left|\left\langle \left\langle \divop F,u^{\varepsilon}-u\right\rangle \right\rangle \right|\\
\ls_{\beta,\mu} & \left\Vert F^{\epsilon}-F\right\Vert _{B_{2,2}^{2\beta-\mu}}\left\Vert u\right\Vert _{B_{2,2}^{-2\beta+1+\mu}}+\left\Vert F\right\Vert _{B_{2,2}^{2\beta-\mu}}\left\Vert u^{\varepsilon}-u\right\Vert _{B_{2,2}^{-2\beta+1+\mu}}\\
\ls_{\beta,\mu} & \left(\varepsilon^{\frac{1}{2}\mu}+\varepsilon^{\frac{1}{2}\left(3\beta-1-\mu\right)}\right)\left\Vert F\right\Vert _{B_{\infty,\infty}^{2\beta}}\left\Vert u\right\Vert _{B_{\infty,\infty}^{\beta}}
\end{align*}
As $\beta>\frac{1}{3}$, we conclude 
\begin{align*}
\lim_{\varepsilon\to0}\int_{0}^{t}\left\langle \left\langle \divop F^{\varepsilon}(s),u^{\varepsilon}(s)\right\rangle \right\rangle \;\mathrm{ds} & =\int_{0}^{t}\left\langle \left\langle \divop F(s),u(s)\right\rangle \right\rangle \;\mathrm{d}s,
\end{align*}
as desired.
\end{proof}

\end{document}